\documentclass[a4paper,11pt]{article}
\usepackage[latin1]{inputenc}
\usepackage[T1]{fontenc}
\usepackage[british,UKenglish,USenglish,american]{babel}
\usepackage{amsmath}
\usepackage{amsfonts}
\usepackage{hyperref}
\usepackage[T1]{fontenc}
\usepackage[latin1]{inputenc}
\usepackage{theorem}
\usepackage{graphics}
\usepackage{makeidx}
\usepackage{fancyhdr}
\usepackage{bbm}
\usepackage{bm}
\usepackage{amssymb}
\usepackage{xcolor}
\usepackage{stmaryrd}
\usepackage{comment}
\numberwithin{equation}{section}
\usepackage{amssymb}
\newcommand{\be}{\begin{equation}}
\newcommand{\ee}{\end{equation}}
\newcommand{\benn}{\begin{equation*}}
\newcommand{\eenn}{\end{equation*}}
\newcommand{\bea}{\begin{eqnarray}}
\newcommand{\eea}{\end{eqnarray}}
\newcommand{\beann}{\begin{eqnarray*}}
\newcommand{\eeann}{\end{eqnarray*}}

\newtheorem{theorem}{Theorem}[section]

\newtheorem{lemma}[theorem]{Lemma}

\newtheorem{definition}[theorem]{Definition}
\newtheorem{refproof}[theorem]{Proof}
\theorembodyfont{\rm}

\newtheorem{remark}[theorem]{Remark}
\newtheorem{example}[theorem]{Example}
\newtheorem{assumptions}[theorem]{Assumptions}
\newtheorem{notation}[theorem]{Notation}

\newcommand{\qed}{\hfill $\Box$\smallskip}

\def\R{\mathbb{R}}
\def\C{\mathbb{C}}
\def\N{\mathbb{N}}
\def\Z{\mathbb{Z}}
\def\P{\mathbb{P}}

\def\U{\mathbb{U}}
\def\W{\mathbb{W}}
\def\bfW{\textbf{W}}
\def\balpha{\boldsymbol{\alpha}}

\def\cB{\mathcal{B}}

\def\cD{\mathcal{D}}

\def\cF{\mathcal{F}}

\def\cL{\mathcal{L}}
\def\cM{\mathcal{M}}

\def\cO{\mathcal{O}}

\def\cV{\mathcal{V}}

\def\cX{\mathcal{X}}

\def\txtd{{\textnormal{d}}}
\def\txte{{\textnormal{e}}}
\def\txti{{\textnormal{i}}}

\def\txtD{{\textnormal{D}}}

\def\Id{{\textnormal{Id}}}

\def\I{\infty}
\def\eps{\varepsilon}
\newcommand{\norm}[1]{\left\lVert #1 \right\rVert}
\newcommand{\Dnorm}[1]{\left\lVert #1 \right\rVert_{\cD^{2\gamma}_W}}
\newcommand{\Bnorm}[1]{\left\lVert #1 \right\rVert_{BC^\eta\left(\cD^{2\gamma}_W\right)}}
\newcommand{\Xnorm}[1]{\left\lVert #1 \right\rVert_{\cX}}

\newcommand*\samethanks[1][\value{footnote}]{\footnotemark[#1]}

\date{}

\usepackage[scale=0.75,top=2.5cm,bottom=2.5cm]{geometry}

\title{Taylor-like approximations of center manifolds for rough differential equations}

\author{Alexandra Blessing Neam\c tu\thanks{University of Konstanz, Department of Mathematics and Statistics,  Universit\"atsstra\ss{}e~10 78464 Konstanz, Germany. E-Mail: alexandra.blessing@uni-konstanz.de, dennis.rudik@uni-konstanz.de}
~~and~~Dennis Rudik\samethanks } 
\begin{document}
	\maketitle	 
\begin{abstract}
	
The dynamics of rough differential equations (RDEs) has recently received a lot of interest.~For example, the existence of local random center manifolds for RDEs has been established. In this work, we present an approximation for local random center manifolds for RDEs driven by geometric rough paths.~To this aim, we combine tools from rough path and deterministic center manifold theory to derive Taylor-like approximations of local random center manifolds.~The coefficients of this approximation are stationary solutions of RDEs driven by the same geometric rough path as the original equation.~We illustrate our approach for stochastic differential equations (SDEs) with linear and nonlinear multiplicative noise.

 

\end{abstract}

    \section{Introduction}


    Center manifolds capture fundamental properties of dynamical systems being a key tool for the analysis of stability of equilibrium points or providing a reduction of the dimension of the essential dynamics~\cite{Carr,Holmes,KnoblauchW}.~However, little is known about the structure of center manifolds for stochastic systems.~Here we contribute to this aspect and establish an approximation result for random center manifolds using tools from rough path theory.~We consider RDEs of the form 
    \begin{align}\label{rde:intro}
        \txtd Y_t = [A Y_t + F(Y_t)]~\txtd t + G(Y_t)~\txtd \mathbf{W}_t,
    \end{align}
    where $\mathbf{W}=(W,\mathbb{W})$ is a $\gamma$-H\"older geometric rough path, for $\gamma\in(1/3,1/2]$ and the coefficients satisfy the suitable assumptions stated in Subsection~\ref{setting}.~Rough path theory~\cite{FritzHairer,Gubinelli,GubinelliTinde,Lyons} provides a pathwise interpretation of~\eqref{rde:intro} making the random dynamical system framework~\cite{Arnold} accessible.~This allows one to investigate dynamical phenomena for~\eqref{rde:intro} such as stability, bifurcations, invariant manifolds, synchronization or chaos.~Rough path theory was used to construct random dynamical systems for RDEs~\cite{BailleulRiedelScheutzow}, rough partial differential equations (RPDEs)~\cite{HN1,HN2}, McKean-Vlasov SDEs~\cite{Gess} and for Navier-Stokes equations with transport type noise~\cite{Martina2,Martina1}, whereas the existence of random invariant manifolds for R(P)DEs was established in~\cite{NA,GhaniVarzanehRiedel, Ibounds, KN23,KN23P}.~In~\cite{KN23,KN23P}, rough path tools were explored in order to set up a Lyapunov-Perron method to derive the existence of center manifolds for R(P)DEs, leaving the computation of such sets open.~Therefore, this work aims to make a first step in this direction, proving an abstract approximation result of random center manifolds using Taylor expansions, inspired by the deterministic setting stated in~\cite{Carr}. \\
    
    Such a situation remained unexplored in the stochastic case, where to the best of our knowledge, approximation results have been established for SDEs with linear multiplicative noise, using a Doss-Sussmann type transformation~\cite{ChekrounLiuWang, ChekrounLiuWang example, Chekroun2}.~These results can be extended to multiplicative noise satisfying a suitable structure as stated in~\cite[Theorem 2.1]{ChekrounLiuWang example}.~For  random dynamical systems generated by differential equations with random coefficients or linear multiplicative noise, an application of the techniques of~\cite{Carr} was considered in~\cite{Boxler1} entailing a polynomial approximation of the corresponding manifolds up to second order.~Our approach based on rough path theory allows one to incorporate nonlinear multiplicative noise interactions as well as to treat non-Markovian fluctuations without relying on any kind of flow transformations.~Another technical challenge that one has to overcome in the stochastic case is to find an appropriate parametrization of the random manifold,  as illustrated in~\cite{ArnoldImkeller}. After a suitable linearization one obtains a system of affine SDEs and computes its stationary solutions. The computation involves a random coordinate transformation $h(\omega)$ which conjugates two random dynamical systems $\varphi$ and $\psi$ meaning that
    \[ \varphi(t,\omega) h (\omega)= h(\theta_t\omega) \psi(t,\omega). \]
    Here $\theta_t\omega$ denotes the usual Wiener shift, see Section~\ref{s:rd} for more details. However, it is difficult to find such a map, even if recent progress on parametrization for invariant manifolds in the presence of noise has been done in~\cite{ChekrounL,Chekroun2}.~Moreover, the procedure developed by~\cite{ArnoldImkeller} gives rise to {anticipative} SDEs which contain coefficients depending on the future of the noise. Therefore pathwise techniques based on rough path theory are convenient \cite{CFV,HN24}.~Notably, the reduced flow derived in~\cite{ChekrounLiuWang, ChekrounLiuWang example, Chekroun2} contains non-Markovian terms incorporating the past of the noise.~This non-Markovianity is reflected by stationary solutions of the SDEs arising in the approximation procedure.~This phenomenon occurs even if the driving noise is given by a Brownian motion which is Markov.~Another advantage of rough path theory is that it allows one to consider non-Markovian noise as well.   \\
    
  The existence of invariant manifolds for rough (partial) differential equations was also established in   \cite{GhRiedel, GhaniVarzanehRiedel} using complementary techniques to~\cite{KN23,KN23P} based on the multiplicative ergodic theorem.~First, the corresponding RDE or RPDE is linearized along a stationary solution.~Thereafter based on sign information of the Lyapunov exponents, an application of the multiplicative ergodic theorem yields the existence of stable/unstable or center manifolds.~The works~\cite{KN23,KN23P} use the Lyapunov-Perron method entailing additional information about the structure of the manifold, i.e.~this can be represented by the graph of a locally Lipschitz function $h^c$, which can be determined from the fixed point of the Lyapunov-Perron map.~This feature is common in the context of invariant manifolds, see~\cite{ChekrounLiuWang, ChekrounLiuWang example, DuanLuSchmalfuss, GarridoLuSchmalfuss, LuSchmalfuss} and the references specified therein.~The starting point of this work is to approximate the function $h^c$ up to order $q\geq 2$ by Taylor polynomials of the form  $\sum\limits_{i=1}^q \alpha_i x^i$, 
where the coefficients $\alpha_i$ solve rough differential equations, see Subsection~\ref{sec:generalidea}.~We show that this ansatz is consistent with the results obtained in~\cite{ChekrounLiuWang, ChekrounLiuWang example} for linear multiplicative noise.~Our setting extends the results in~\cite{ChekrounLiuWang} by incorporating nonlinear multiplicative noise, which can be non-Markovian and improving the approximation order.~One of the main technical challenges is to set up a suitable function space incorporating the fixed point argument for the Lyapunov-Perron map as well as the polynomial dependence of the approximation in order to rely on the rough path machinery.\\

    Since the methods presented in this work are independent of the dimension, we expect that they carry over to rough PDEs, for which the existence of center manifolds was established in~\cite{KN23P}.~Approximation results of center manifolds for stochastic partial differential equations with linear multiplicative noise based on flow transformations are available for e.g.~Burgers equation and Rayleigh-B\'enard convection~\cite{ChekrounLiuWang,ChekrounLiuWang example}.~Another exciting direction is given by rough stochastic differential equations investigated in~\cite{Fritz} by means of the stochastic sewing lemma~\cite{Le}.~Combining perfection arguments for random dynamical systems together with the techniques in~\cite{KN23}, we believe that we can investigate center manifolds in this situation as well.~Their approximation based on the tools developed in this work would be highly challenging.~Naturally, we plan to use the center manifold reduction for the stability analysis of random equilibrium points and advance based on this reduction method the stochastic bifurcation theory.~Finally, we intend to develop tools for the computation of the approximation derived in Section~\ref{sec:main} possibly using signature methods~\cite{Lyons:S} and compare this for instance with forward-backward schemes developed in~\cite[Chapter 4]{Chekroun2}.~As already mentioned, these lead to non-Markovian approximations of the manifolds even if the driving noise is Markovian.~In our setting we would face two challenges: the non-Markovian nature of the approximation together with the non-Markovianity of the noise.\\

    Another natural question is given by the approximation of stable and unstable manifolds for RDEs.~In the deterministic setting, this aspect was investigated by~\cite{Kr}.~Moreover, in~\cite{PoetzscheRasmussen}, the approximation of such manifolds is considered for deterministic nonautonomous systems, where the nonautonomous forcing is smooth in time.~In this case, the coefficients of the Taylor approximation solve nonautonomous ODEs, which are derived from the Lyapunov-Perron map.~Such techniques could be used for RDEs, approximating the geometric rough path by a sequence of smooth paths, applying the the results of \cite{PoetzscheRasmussen} and investigating if the resulting sequence of manifolds converge to the manifold of the original problem, whose existence can be inferred by similar arguments to~\cite{KN23} and~\cite{GhRiedel,GhaniVarzanehRiedel}.\\
    
    We finally mention that approximations of center manifolds in the small noise regime using tools from multiscale analysis and amplitude equations have been obtained in~\cite{BloemkerHairer}. Moreover, for slow fast systems normal form coordinate transformations have been investigated in~\cite{Roberts}.~Stochastic center manifold theory for small noise is in particular relevant for applications in physics, as considered in~\cite{SchoenerHaken}.~In this case approximations of local center manifolds have been derived by~\cite{ArnoldImkeller,SchoenerHaken} using expansions of the nonlinear terms together with expansions of the small diffusion coefficient $\varepsilon$ controlling the noise intensity.


  \subsection*{Structure of the paper}
  This manuscript is structured as follows. In Section~\ref{sec:p} we collect preliminaries in rough paths and random dynamical systems with a particular emphasis on the existence of local random center manifolds established in~\cite{KN23}. Section~\ref{sec:main}
contains the main result of this work stated in Theorem~\ref{thm:approx:m} and is divided into several subsections explaining our setting and the steps of the proof. First, we specify that we are working in a center-stable situation, meaning that the linear part of the system under consideration has zero and strictly negative eigenvalues. Since we are only interested in establishing an approximation of the manifold in a neighborhood of the origin, similar to~\cite{KN23}, we use a cut-off argument in Subsection~\ref{sec:cutoff} to truncate the coefficients of the RDE.~In Subsection~\ref{sec:approx} we formulate the general ansatz which relies on a Taylor-like approximation of the manifold.~We provide in Subsection~\ref{sec:rpest} suitable estimates in the controlled rough path norm of the terms arising in the ansatz.~Section~\ref{sec:generalidea} contains the general idea of the approximation arguments together with the main steps of the proofs.~This technique is inspired by the deterministic setting investigated in~\cite{Carr}, which is also explained in detail in Appendix~\ref{sec:carr} for the convenience of the reader.~Similar to~\cite{KN23}, we need to discretize the terms appearing in the approximation ansatz, due to the fact that these contain rough integrals which are defined on an unbounded time interval.~This technical step is described in Subsection~\ref{sec:discretization}.~Finally, the necessary fixed point argument for the approximation result is contained in Subsection~\ref{sec:fp}.~The main challenge is to find a suitable function space in order to perform the fixed point argument.~The formal definition of this space is stated in Subsection~\ref{sec:generalidea} and is thereafter made rigorous in Subsection~\ref{sec:discretization}.   \\

Section~\ref{sec:app} is devoted to applications of the results established in Section~\ref{sec:main}.~First of all, for 
 linear multiplicative noise, we compare our ansatz with the results established in \cite{ChekrounLiuWang,ChekrounLiuWang example}.~As already mentioned, we can treat nonlinear, non-Markovian noise and
 illustrate this fact in Example~\ref{Chekroun nonlinear example}, which exclusively relies on rough path theory.\\

For the convenience of the reader we sketch the main ideas of the approximation of deterministic center manifolds in Appendix~\ref{sec:carr}. Auxiliary results which are necessary for the proofs in Section~\ref{sec:main} are collected in Appendix~\ref{a} and~\ref{Polynomial calculation}. As already mentioned, we plan to develop in a future work computational tools for the coefficients of the center manifold approximation, in particular of the terms derived in Appendix~\ref{Polynomial calculation}.~For the sake of completeness, we further provide in Appendix~\ref{ChekrounProofWithRP} an approximation result of the center manifold using the leading order terms of the coefficients of the RDE~\eqref{rde:intro}.~This is inspired by the setting of~\cite{ChekrounLiuWang, ChekrounLiuWang example, Chekroun2} treating SDEs with linear multiplicative noise and essentially simplifies the arguments in Section~\ref{sec:main}.~More precisely, given an approximation of $F$ and $G$ up to leading order, one can directly perform suitable estimates of the corresponding Lyapunov-Perron map in the rough path norm without relying on a fixed point argument as in Subsection~\ref{sec:fp}.~Since combining this technique with rough path tools is interesting in its own right, we provided further details on this topic in Appendix~\ref{ChekrounProofWithRP}.~The proof relies on additional estimates for the leading order term of the nonlinearities modified by a cut-off function in the controlled rough path norm. These are provided in Appendix~\ref{order:cutoff}. 
\\


 \subsection*{Acknowledgements}
 The authors have been supported by the  Deutsche Forschungsgemeinschaft (DFG, German Research Foundation) - Project ID 543163250.~A. Blessing acknowledges funding by the DFG CRC/TRR 388 "Rough Analysis, Stochastic Dynamics and Related Fields" - Project ID 516748464 and DFG CRC 1432 " Fluctuations and Nonlinearities in Classical and Quantum Matter beyond Equilibrium" - Project ID 425217212.~The authors thank Christian Kuehn and Giacomo Landi for useful discussions. 
 
\section{Preliminaries}\label{sec:p}

\subsection{Rough paths}\label{rp}
Let $\cX,\cV$ be two finite-dimensional vector spaces (e.g. $\cX=\R^n,~\cV=\R^d$ for $n,d\geq 1$). For convenience, we work throughout this subsection on the time interval $[0,1]$. We first introduce some basic concepts from the theory of rough paths.

\begin{definition}
   We fix $\gamma\in (1/3,1/2]$. A pair 
$\textbf{W}:=(W,\mathbb{W})$ is called  a $\gamma$-H\"older rough-path, if
$
W\in C^{\gamma}([0,1];\cV), 
\mathbb{W}\in C^{2\gamma}([0,1]^{2};\cV \otimes \cV)$
and the connection between $W$ and $\mathbb{W}$ is given by Chen's relation. 
This means that
\begin{align}
\label{chen}
\mathbb{W}_{s,t} -\mathbb{W}_{s,u} -\mathbb{W}_{u,t} = {W}_{s,u} \otimes W_{u,t}.
\end{align}
where $W_{s,t}:=W_{t} -W_{s}$.
\end{definition}
\begin{definition}
    Given two $\gamma$-H\"older rough paths $\mathbf{W}=(W,\mathbb{W})$ and $\mathbf{\tilde{W}}=(\tilde{W},\tilde{\mathbb{W}})$ we define the $\gamma$-H\"older rough path distance as
    \begin{align}\label{rpmetric}
    \rho_\gamma(\mathbf{W},\tilde{\mathbf{W}})
    = \sup\limits_{0\leq s< t \leq 1} \frac{|W_{s,t} - \tilde{W}_{s,t}|}{|t-s|^\gamma} + \sup\limits_{0\leq s< t \leq 1} \frac{|\mathbb{W}_{s,t} - \tilde{\mathbb{W}}_{s,t}|}{|t-s|^{2\gamma}}. 
    \end{align}
\end{definition}

\begin{definition}\label{def:geometricrp}
A $\gamma$-H\"older rough path $\mathbf{W}=(W,\mathbb{W})$ is geometric if and only if there exists a sequence of smooth paths $(W^n)_{n \geq 1}$ such that $\rho_\gamma(\mathbf{W},\mathbf{W}^n) \to 0 $ where $\mathbf{W}^n=(W^n,\mathbb{W}^n)$ and $\mathbb{W}^n_{s,t}=\int_s^t W^n_{s,r}~\otimes \txtd W^n_r $  for  all $0\leq s \leq t\leq 1$. In particular, we have that 
\[ \emph{ Sym}(\mathbb{W}_{s,t})=\frac{1}{2} W_{s,t}\otimes W_{s,t}.  \]
 \end{definition}

\begin{remark}
    For our aims we only consider geometric rough paths, since they satisfy the chain rule. This turns out to be essential for the arguments in Section~\ref{sec:main}. In particular, for the Brownian rough path $\mathbf{B}=(B,\mathbb{B})$, this means that $\mathbb{B}$ is the Stratonovich lift of the Brownian motion $B$. 
\end{remark}

 We now consider the rough differential equation (RDE) on a time interval $[0,1]$ given by 
\begin{equation}
\label{sde1}
\txtd U_t = [A U_t + F(U_t)] ~\txtd t + G(U_t) ~\txtd \textbf{W}_t,\qquad U_0=\xi\in\cX,
\end{equation}
where $A\in\cL(\cX)$ is a matrix and $\textbf{W}=(W,\mathbb{W})$ is a $\gamma$-H\"older rough path.~The assumptions on the coefficients $F$ and $G$ will be stated below.~The mild solution of~\eqref{sde1} is given by
\begin{equation}
\label{integraleq}
U_{t}= S(t) \xi + \int\limits_{0}^{t} S(t-r)F(U_{r})~\txtd r 
+ \int\limits_{0}^{t} S(t-r) G(U_{r}) ~\txtd \textbf{W}_{r},
\end{equation}
where $S(t):=\txte^{tA}$, and the last integral is given by Gubinelli's controlled 
rough integral defined below in Theorem \ref{gubinelliintegral}. \\

Throughout this section $\Xnorm{\cdot}$ denotes the usual Euclidean norm, 
and $\Xnorm{\cdot}$ also denotes the induced Euclidean norm of linear operators. We write    
$\|\cdot\|_{\infty}$ for the supremum norm and $\|\cdot\|_{\gamma}$ for the $\gamma$-H\"older 
semi-norm on $[0,1]$. 
\begin{definition}
\label{def:crp} 
A path $Y\in C^{\gamma}([0,1];\cX)$ is controlled by $W\in C^{\gamma}([0,1];\cV)$ 
if there exists $Y'\in C^{\gamma}([0,1];\cL(\cV,\cX))$ such that
\begin{align}
\label{def1}
Y_{t} = Y_{s} + Y'_{s} W_{s,t} + R^{Y}_{s,t},
\end{align}
where the remainder $R^{Y}$ has $2\gamma$-H\"older regularity. 
The space of controlled rough paths $(Y,Y')$ on the time interval $[0,1]$ is denoted by $\cD^{2\gamma}_{W}([0,1];\cX)=:\cD^{2\gamma}_{W}$. 
This space is endowed with the semi-norm 
\begin{align}\label{norm}
\|Y,Y'\|_{\cD^{2\gamma}_W}:= \|Y'\|_{\gamma} + \|R^{Y}\|_{2\gamma}.
\end{align}
\end{definition}
The norm on $\cD^{2\gamma}_W$ is given by $\Xnorm{Y_0} + \Xnorm{Y'_{0}} + \|Y'\|_{\gamma} + \|R^{Y}\|_{2\gamma} $ and will be denoted by $\|\cdot\|_{\cD^{2\gamma}_{W}([0,1];\cX)}$ (short $\|\cdot\|_{\cD^{2\gamma}_{W}}$). Obviously, if $Y_0=0$ and 
$Y'_0=0$, then~\eqref{norm} is a norm on $\cD^{2\gamma}_{W}$.\\

\begin{remark}
For convenience, we use throughout the manuscript the equivalent norm on $\cD^{2\gamma}_W$ given by  \begin{align}
    \|Y,Y'\|_{\cD^{2\gamma}_W}= \|Y\|_{\infty} + \|Y'\|_{\infty} + \|Y'\|_{\gamma} + \|R^{Y}\|_{2\gamma}. \label{Dnorm}
\end{align} 
\end{remark}
As a next step we explain the integral we use in \eqref{integraleq}
\begin{theorem}{\em{(\cite[Prop.~1]{Gubinelli})}}
\label{gubinelliintegral}
Let $(Y,Y')\in \cD^{2\gamma}_{W}([0,1];\cX)$ and $\bm{W}=(W,\mathbb{W})$ 
be an $\cV$-valued 
$\gamma$-H\"older rough path for some $\gamma\in(\frac{1}{3},\frac{1}{2}]$. Furthermore, 
$\mathcal{P}$ stands for a partition of $[0,1]$. Then, the integral of $Y$ against 
$\bm{W}$ is defined as
\begin{equation}\label{Gintegral}
	\int_{s}^{t} Y_{r} ~\txtd\textbf{W}_{r} 
	:= \lim\limits_{|\mathcal{P}|\to 0} \sum\limits_{[u,v]\in\mathcal{P}} 
	(Y_{u}W_{u,v} + Y'_{u}\mathbb{W}_{u,v} )
\end{equation}
and exists for every pair $s,t\in[0,1]$. Moreover, the estimate
\begin{equation}\label{estimate_g_integral}
	\left| \int_{s}^{t} Y_{r}~ \txtd\textbf{W}_{r} - Y_{s}W_{s,t} - 
	Y'_{s}\mathbb{W}_{s,t} \right| \leq C (\|W\|_{\gamma} \|R^{Y}\|_{2\gamma} 
	+ \|\mathbb{W} \|_{2\gamma} \|Y' \|_{\gamma} ) |t-s|^{3\gamma}
\end{equation} 
holds true for all $s,t\in[0,1]$.
\end{theorem}

For the sake of completeness and a better comprehension, we state some well-known results for sums and products of controlled rough paths, the deterministic and stochastic convolution together with some important estimates in the $\cD^{2\gamma}_W$-norm of the difference of two controlled rough paths.
\begin{lemma}\label{lem:addition and multiplication of RP}
    Let $(Y,Y'),(\tilde Y,\tilde Y')\in \cD^{2\gamma}_W([0,1],\cX)$. Then $(Y+\tilde Y,Y'+\tilde Y')\in \cD^{2\gamma}_W([0,1],\cX)$ and $(Y\tilde Y, Y'\tilde Y+ Y\tilde Y')\in\cD^{2\gamma}_W([0,1],\cX)$. Moreover, we have
    \begin{align*}
        \Dnorm{Y+\tilde Y,Y'+\tilde Y'} \le \Dnorm{Y,Y'}+\Dnorm{\tilde Y,\tilde Y'},
    \end{align*}
    and 
    \begin{align*}
        \Dnorm{Y\tilde Y, Y'\tilde Y+ Y\tilde Y'} \le C(1+\norm{W}_\gamma)^2 \Dnorm{Y,Y'}\Dnorm{\tilde Y,\tilde Y'}.
    \end{align*}
\end{lemma}
\begin{proof}
    The claim for addition follows directly by applying the definition of the $\cD^{2\gamma}_W$-norm. For multiplication we use the definition of the $\cD^{2\gamma}_W$-norm and get
    \begin{align*}
        &\Dnorm{Y\tilde Y, Y'\tilde Y+ Y\tilde Y'} \\
        &= \norm{Y\tilde Y}_\I + \norm{Y'\tilde Y+ Y\tilde Y'}_\I + \norm{Y'\tilde Y+ Y\tilde Y'}_\gamma + \norm{R^{Y\tilde Y}}_{2\gamma} \\
        &\le \norm{Y}_\I\norm{\tilde Y}_\I + \norm{Y'}_\I\norm{\tilde Y}_\I + \norm{Y}_\I\norm{\tilde Y'}_\I + \norm{Y'}_\gamma\norm{\tilde Y}_\I + \norm{Y}_\I \norm{\tilde Y'}_\gamma + \norm{R^{Y\tilde Y}}_{2\gamma}\\
        &\le C \Dnorm{Y,Y'}\Dnorm{\tilde Y,\tilde Y'} + \norm{R^{Y\tilde Y}}_{2\gamma}.
    \end{align*}
    It remains to compute $\norm{R^{Y\tilde Y}}_{2\gamma}$.~This results in
    \begin{align*}
        \norm{R^{Y\tilde Y}}_{2\gamma}
        &=\sup_{s,t\in[0,1]} \frac{\Xnorm{Y_t\tilde Y_t-Y_s\tilde Y_s - Y_s'\tilde Y_s W_{s,t}- Y_s\tilde Y_s'W_{s,t}}}{|t-s|^{2\gamma}} \\
        &=\sup_{s,t\in[0,1]} \frac{\Xnorm{Y_sR^{\tilde Y}_{s,t}+\tilde Y_s R^{Y}_{s,t}+(Y_t-Y_s)(\tilde Y_t-\tilde Y_s)}}{|t-s|^{2\gamma}} \\
        &\le \norm{Y}_\I\norm{R^{\tilde Y}}_{2\gamma} + \norm{\tilde Y}_\I \norm{R^Y}_{2\gamma} + \norm{Y}_\gamma\norm{\tilde Y}_\gamma \\
        &\le C \Dnorm{Y,Y'}\Dnorm{\tilde Y,\tilde Y'} + C\left(1+\norm{W}_\gamma\right)^2 \Dnorm{Y,Y'}\Dnorm{\tilde Y,\tilde Y'},
    \end{align*}
    where we used $\norm{Y}_\gamma\le C\left(1+\norm{W}_\gamma\right)\Dnorm{Y,Y'}$.  \qed\\
\end{proof}

\begin{lemma}{\em{(\cite[Lemma 2.10]{KN23})}}\label{D2gamma bound for non-random integral}
    Let $\xi\in\cX$ and $F:\cX \to \cX$ be in $C_b^1$. 
Then for $(Y,Y'),(\tilde Y,\tilde Y')\in \cD^{2\gamma}_{W}([0,1];\cX)$, a Gubinelli derivative is given by
\begin{align}\label{gderivative:drfift}
\left(S(\cdot)\xi + \int\limits_{0}^{\cdot} S(\cdot-r)\left(F(Y_{r})-F(\tilde Y_r)\right)~\txtd r  \right)'= 0.
\end{align}
Moreover, we have the estimate 
\begin{align}
\label{fdrift}
\left\| S(\cdot)\xi + \int\limits_{0}^{\cdot} S(\cdot-r)\left(F(Y_{r})-F(\tilde Y_r)\right)~\txtd r, 
0 \right\|_{\cD^{2\gamma}_{W}} &\leq  C[\Xnorm{A} ]\left(\Xnorm{\xi} + \|F(Y)-F(\tilde Y)\|_{\I}\right) \nonumber\\
&\le C[\Xnorm{A} ]\left(\Xnorm{\xi} + L_F\Dnorm{Y-\tilde Y,Y'-\tilde Y'}\right).
\end{align}
\end{lemma}

\begin{lemma}{\em{(\cite[Lemma 2.7, Lemma 2.8]{KN23}})}\label{D2gamma bound for stoch integral}
    Let $G\in C_b^3(\cX;\cL(\cV,\cX))$. Then for $(Y,Y'),(\tilde Y,\tilde Y')\in \cD^{2\gamma}_{W}([0,1];\cX)$ we get
    \begin{align*}
        &\Dnorm{\int_0^\cdot S(\cdot-r)\left(G(Y_r)-G(\tilde Y_r)\right)~\txtd \bfW_r, G(Y_r)-G(\tilde Y_r)}\\
        &\le C[\Xnorm{A}]\left(1+\norm{W}_\gamma\right)\left(\norm{W}_\gamma +\norm{\W}_{2\gamma}\right)\Dnorm{G(Y)-G(\tilde Y),G(Y)'-G(\tilde Y)'} \\
        &\le C[\Xnorm{A}]\left(1+\norm{W}_\gamma\right)\left(\norm{W}_\gamma +\norm{\W}_{2\gamma}\right)CN^2 \norm{G}_{C_b^3}\Dnorm{Y-\tilde Y,Y'-\tilde Y'},
    \end{align*}
    where $N$ is a positive constant such that $\Dnorm{Y,Y'}\le N$ and $\Dnorm{\tilde Y,\tilde Y'}\le N$.
\end{lemma}
Combining the previous Lemmas we get the following result.
\begin{lemma}\label{integral bound} 
Let $(Y,Y'),(\tilde Y,\tilde Y')\in \cD^{2\gamma}_{W}\left([0,1];\cX\right)$. Moreover, let $F\in C_b^1\left(\cX;\cX\right)$ and $G\in C_b^3(\cX;\cL\left(\cV,\cX\right))$. Then 
    \begin{align*}
        &\Dnorm{\int_0^\cdot S(\cdot-r)\left(F(Y_r)-F(\tilde Y_r)\right)~\txtd r + \int_0^\cdot S(\cdot-r)\left(G(Y_r)-G(\tilde Y_r)\right)~\txtd \textbf{W}_r, G(Y)-G(\tilde Y)} \\
        &\le C[\Xnorm{A}]\Bigg(\norm{F(Y)-F(\tilde Y)}_\I\\&~~~~+\left(1+\norm{W}_\gamma\right)\left(\norm{W}_\gamma+\norm{\mathbb{W}}_{2\gamma}\right)\Dnorm{G(Y)-G(\tilde Y),G(Y)'-G(\tilde Y)'} \Bigg) \\
        &\le CC[\Xnorm{A}]\left(L_F + \left(1+\norm{W}_\gamma\right)\left(\norm{W}_\gamma +\norm{\W}_{2\gamma}\right)N^2 \norm{G}_{C_b^3}\right)\Dnorm{Y-\tilde Y, Y' - \tilde Y'},
    \end{align*}
    where $N$ is a positive constant such that $\Dnorm{Y,Y'}\le N$ and $\Dnorm{\tilde Y,\tilde Y'}\le N$.
\end{lemma}
\begin{proof}
    The claim follows by putting Lemma \ref{D2gamma bound for non-random integral} and Lemma \ref{D2gamma bound for stoch integral} together. \qed\\
\end{proof}

In order to obtain center manifolds for~\eqref{sde1}, we further assume the following exponential dichotomy for the semigroup $(S_t)_{t\geq 0}$. 

\begin{assumptions}
\label{ass:linearpart} We assume that we are in a center-stable situation, namely there 
are eigenvalues $\{\lambda^c_j\}_{j=1}^{n_c}$ of the linear operator $A$ on 
the imaginary axis $\txti\R$ as well as eigenvalues $\{\lambda^s_j\}_{j=1}^{n_s}$ 
in the left-half plane $\{z\in\C:\textnormal{Re}(z)<0\}$. Upon counting multiplicities we
have $n_c+n_s=n$. Hence, there 
exists a decomposition of the phase space $\R^n=\cX=\cX^{c}\oplus \cX^{s}$, where 
the linear spaces $\cX^{c}$ and $\cX^{s}$ are spanned by the (generalized) 
eigenvectors with eigenvalues $\lambda^c_j$ and $\lambda^s_j$ respectively. 
We denote the restrictions of $A$ on $\cX^{c}$ and $\cX^{s}$ by $A^{c}:=A|_{\cX^{c}}$ and $A^{s}:=A|_{\cX^{s}}$. Then, $S^{c}(t):=e^{tA^{c}}$ and $S^{s}(t):=e^{tA^{s}}$ are groups of linear operators on $\cX^{c}$ respectively $\cX^{s}$. Moreover, there exist two bounded projections $P^{c}$ and $P^{s}$ associated 
to this splitting such that
	\begin{itemize} 
	\item [1)] $\Id= P^{s}+ P^{c}$;
	\item [2)] $P^{c}S(t)=S(t)P^{c}$ and $P^{s}S(t)=S(t)P^{s}$ 
	for $t\geq 0$.
\end{itemize}
Additionally, we assume that there exist two exponents $\nu$ and $\beta$ 
with $-\beta<0\leq \nu<\beta$ and constants $M_{c},M_{s}\geq 1$, such that
	\begin{align}
	&\Xnorm{S^{c}(t)  x} \leq M_{c} \txte^{\nu t} \Xnorm{x}, 
	~~\mbox{  for } t\leq 0 \mbox{ and } x\in \cX;\label{nu}\\
	& \Xnorm{S^{s}(t)x} \leq M_{s} \txte^{-\beta t} \Xnorm{x}, 
	~\mbox{for } t\geq 0 \mbox{ and } x\in \cX.\label{beta}
	\end{align}
\end{assumptions}

In this situation we have that
$\nu\geq 0$ and $-\beta<0$ which gives us the spectral gap $\nu+\beta>0$.
We also refer to $\cX^{c}$ and $\cX^{s}$ as center, respectively stable, subspace.\medskip

\subsection{Random dynamics}\label{s:rd}
We now state some basic concepts from the theory of random dynamical systems~\cite{Arnold}.
\begin{definition} 
Let $(\Omega,\mathcal{F},\mathbb{P})$ stand for a probability space and 
$\theta:\mathbb{R}\times\Omega\rightarrow\Omega$ be a family of 
$\mathbb{P}$-preserving transformations (i.e.,~$\theta_{t}\mathbb{P}=
\mathbb{P}$ for $t\in\mathbb{R}$) having the following properties:
\begin{description}
		\item[(i)] the mapping $(t,\omega)\mapsto\theta_{t}\omega$ is 
		$(\mathcal{B}(\mathbb{R})\otimes\mathcal{F},\mathcal{F})$-measurable, where 
		$\mathcal{B}(\cdot)$ denotes the Borel sigma-algebra;
		\item[(ii)] $\theta_{0}=\textnormal{Id}_{\Omega}$;
		\item[(iii)] $\theta_{t+s}=\theta_{t}\circ\theta_{s}$ for all 
		$t,s,\in\mathbb{R}$.
\end{description}
Then the quadruple $(\Omega,\mathcal{F},\mathbb{P},(\theta_{t})_{t\in\mathbb{R}})$ 
is called a metric dynamical system.
\end{definition}
 For a $\gamma$-H\"older rough path $\textbf{W}=(W,\mathbb{W})$ 
and $\tau\in\mathbb{R}$ let us define the time-shift $\Theta_{\tau}\textbf{W}
:=(\Theta_{\tau}W,\Theta_{\tau}\mathbb{W})$ by
\begin{align*}
& \Theta_{\tau} W_t : =W_{t+\tau} - W_{\tau}\\
& \Theta_{\tau}\mathbb{W}_{s,t}:=\mathbb{W}_{s+\tau,t+\tau}
\end{align*}
Note that 
$\Theta_{\tau} W_{s,t}=W_{t+\tau} - W_{s+\tau}$.
Moreover, we consider the following concept introduced in \cite{BailleulRiedelScheutzow}.\\

\begin{definition}\label{rough path cocycle}
    Let $(\Omega,\mathcal{F},\P,(\theta_t)_{t\in\R})$ be a metric dynamical system. We call $\bfW=(W,\W)$ a rough path cocycle if the following identities hold true for every $s,t\in[0,\infty)$ and $\omega\in\Omega$
   \begin{align*}
       &W_{s,s+t}(\omega)=W_t(\theta_s\omega)\\
    &\mathbb{W}_{s,s+t}(\omega)=\mathbb{W}_{0,t}(\theta_s\omega)
   \end{align*}
\end{definition}

\begin{example}\label{ex:rp cocycle}
According to~\cite[Section 2]{BailleulRiedelScheutzow} rough path lifts of various stochastic processes define cocycles. These include Gaussian processes with stationary increments under certain assumption on the covariance function~\cite[Chapter 7]{Fritz}.
Therefore we keep the following examples in mind:
\begin{itemize}
    \item [1)] We start with the canonical probability space $(C_0(\R,\R^d),\cB(C_0(\R,\R^d),\P)$, where $C_0(\R,\R^d)$ is the space of $\R^d$-valued continuous functions, which are $0$ in 0, endowed with the compact open topology and $\P$ is the Gauss measure. We define the shift
    $$(\Theta_t f)(\cdot) := f(t+\cdot)-f(t),~~~t\in\R,~~~f\in C_0(\R,\R^d).$$
    Furthermore, we restrict it to the set $\Omega:= C_0^{\gamma'}(\R,\R^d)$. We equip this set with the trace $\sigma$-algebra $\cF := \Omega\cap\cB(C_0(\R,\R^d))$ and also take the restriction of $\P$.
    Let $(B_t(\cdot))_{t\in[0,T]}$ be a $d$-dimensional Brownian motion on the probability space $(\Omega,\cB(C_0(\R,\R^d),\P)$ and consider the Stratonovich lift for a realisation $\omega\in\Omega$
        $$\mathbb{B}^{\textnormal{Strat}}_{s,t}(\omega):= \int_s^t (B_r(\omega)-B_s(\omega))\circ\txtd B_r(\omega).$$
        Then  $\textbf{B}^{\textnormal{Strat}}(\omega)=(B(\omega),\mathbb{B}^{\textnormal{Strat}}(\omega))$ is a $\gamma$-H\"older geometric rough path for $\gamma<1/2$. This forms a rough path cocycle according to \cite[Section 2]{BailleulRiedelScheutzow}.
    \item [2)] We consider a $d$-dimensional fractional Brownian motion $(B^H_t)_{t\in[0,T]}$ with Hurst index $H\in(1/3,1/2]$. Again we work with the metric dynamical system $(\Omega,\cB(C_0(\R,\R^d)),\P,\Theta)$, where now $\P$ is the fractional Gauss measure. We can approximate $B^H(\omega)$ with a sequence of piecewise dyadic linear functions $(B^{H,n}(\omega))_{n\in\N}$. For each $B^{H,n}(\omega)$ the iterated integral 
    $$\mathbb{B}^{H,n}_{s,t}(\omega):= \int_s^t (B_r^{H,n}(\omega)-B_s^{H,n}(\omega))~\txtd B_r^{H,n}(\omega)$$
    exists in the classical sense for all $s,t\in [0,T]$. According to ~\cite[Theorem 2]{CoutinQian}, the sequence of smooth paths $(\textbf{B}^{H,n}(\omega))_{n\in\N}=(B^{H,n}(\omega),\mathbb{B}^{H,n}(\omega))_{n\in\N}$ converges to a rough path $\textbf{B}^H(\omega)=(B^H(\omega),\mathbb{B}^H(\omega))$ in the rough path distance $\rho_\gamma$ defined in \eqref{rpmetric}. We get the rough path on the whole real line by gluing $\textbf{B}^H(\omega)=(B^H(\omega),\mathbb{B}^H(\omega))$ from each compact interval together. This again forms a rough path cocycle according to~\cite[Section 2]{BailleulRiedelScheutzow}. For more details we refer to \cite[Example 3.5]{KN23}. 
\end{itemize}  
\end{example}
Next we define random dynamical systems, which allow us to investigate the dynamics of the RDE~\eqref{sde1}.
\begin{definition}
\label{rds} 
A random dynamical system on $\cX$ over a metric dynamical 
system $(\Omega,\mathcal{F},\mathbb{P},(\theta_{t})_{t\in\mathbb{R}})$ 
is a mapping $$\varphi:[0,\I)\times\Omega\times \cX\to \cX,
\mbox{  } (t,\omega,x)\mapsto \varphi(t,\omega,x), $$
which is $(\mathcal{B}([0,\I))\times\mathcal{F}\times
\mathcal{B}(\cX),\mathcal{B}(\cX))$-measurable and satisfies:
	\begin{description}
		\item[(i)] $\varphi(0,\omega,\cdot{})=\textnormal{Id}_{\cX}$ 
		for all $\omega\in\Omega$;
		\item[(ii)]$ \varphi(t+\tau,\omega,x)=
		\varphi(t,\theta_{\tau}\omega,\varphi(\tau,\omega,x)), 
		\mbox{ for all } x\in \cX, ~t,\tau\in[0,\I),~\omega\in\Omega;$
		\item[(iii)] $\varphi(t,\omega,\cdot{}):\cX\to \cX$ is 
		continuous for all $t\in[0,\I)$ and all $\omega\in\Omega$.
	\end{description}
\end{definition}
We further define random variables that satisfy certain growth conditions.   
\begin{definition}\label{def:tempered}
A random variable $\widetilde{R}:\Omega\to (0,\infty)$ is called tempered with 
respect to a metric dynamical system $(\Omega,\mathcal{F},\mathbb{P},
(\theta_{t})_{t\in\mathbb{R}})$, if
\begin{equation}
\label{tempered}
\limsup\limits_{t\to\pm\infty}\frac{\ln \widetilde{R} (\theta_{t}\omega)}{t}=0, 
\quad \mbox{ for all } \omega\in\Omega.
\end{equation}	
Furthermore we call a random variable $\tilde{R}:\Omega\to (0,\infty)$ tempered from above if 
\begin{equation}
\label{tempered from above}
\limsup\limits_{t\to\pm\infty}\frac{\ln^+ \widetilde{R} (\theta_{t}\omega)}{t}=0, 
\quad \mbox{ for all } \omega\in\Omega,
\end{equation}	
where $\ln^+a:=\max\left\{\ln a,0\right\}$.
A random variable is called tempered from below if $1/\widetilde{R}$ is tempered from above.
\end{definition}
In conclusion, a random variable is tempered if and only if it is tempered from above and from below.\\

We further emphasize that temperedness is equivalent to {subexponential growth} and allows one to control the growth of the random input along the orbits $(\theta_t)_{t\in\R}$. 
Note that the set of all $\omega\in\Omega$ satisfying~\eqref{tempered from above} 
is invariant with respect to any shift map $(\theta_{t})_{t\in\mathbb{R}}$,
which is an observation applicable to our case when $\theta_t=\Theta_t$. A 
sufficient condition for temperedness from above is according to~\cite[Prop.~4.1.3]{Arnold} 
that
\begin{equation}\label{c:temp}
\mathbb{E} \sup\limits_{t\in[0,1]}  \widetilde{R}(\theta_{t}\omega)<\infty.
\end{equation}
\begin{remark}
\item [1)]  Similar to \cite{KN23} we will use the notation $W\in \Omega$ instead of the standard notation from RDS where one identifies the path $W_t(\omega)=\omega_t$, for $\omega$ belonging to a canonical probability space $(\Omega,\cF,\P)$.
    \item [2)] For a rough path cocycle $\bfW=(W,\W)$, the random variables $\norm{W}_\gamma$ and $\norm{\W}_{2\gamma}$ are tempered from above. Moreover, every polynomial of $\norm{W}_\gamma$ and $\norm{\W}_{2\gamma}$ is tempered from above. Both statements easily follow verifying~\eqref{c:temp}.
    \item [3)] Given a rough path cocycle $\bfW=(W,\W)$, we will consider random variables $R(\bfW)$ which are given by the inverse of polynomials containing  $\norm{W}_\gamma$ and $\norm{\W}_{2\gamma}$. In this case, $R(\bfW)$ is tempered from below and  $1/R(\bfW)$ is tempered from above. 
\end{remark}

For our aims we need the following characterization of temperedness from above/below. Since we perform a discretization argument in Subsection~\ref{sec:discretization}, we state the following well-known results for a random dynamical system with discrete time $\Z^-$.
\begin{lemma}{\em(\cite[Lemma 5.1]{GarridoLuSchmalfuss})}\label{lemma:tbelow} Let $r$ be a positive tempered from below random variable. Then there exists a positive tempered from below random variable $\rho$ and positive constants $\varepsilon$ and $C$ such that
\begin{align}\label{tbelow}
    C\rho(\omega)e^{\varepsilon i} \leq r(\theta_{i}\omega),~~\text{ for all }i\in \Z^{-}.
\end{align}
\end{lemma}
Since for a random variable $r$ which is tempered from below, we obtain that $1/r$ is tempered from above, we can infer
\begin{lemma}
    Let $\tilde{R}$ be a positive, tempered from above random variable. Then there exists a positive tempered from above random variable $\tilde{\rho}$ and positive constants $\varepsilon$ and $\tilde{C}$ such that
    \begin{align}\label{tabove}
        \tilde{R}(\theta_{i}\omega) \leq \tilde{C}\tilde{\rho}(\omega) e^{-\varepsilon i},~~\text{ for all }i\in \Z^{-}. 
    \end{align}
\end{lemma}
Last we state a result about the product of tempered from above random variables.
\begin{lemma}\label{lem:product of tempered variables}
    Let $\tilde R,\hat R$ be positive, tempered from above random variables. Then the product $\tilde R\hat R$ is also tempered from above. If $\tilde R,\hat R$ are tempered from below, then the product $\tilde R\hat R$ is also tempered from below.
\end{lemma}
\begin{proof}
    We want to show that 
    \begin{align*}
        \limsup\limits_{t\to\pm\I}\frac{\ln^+\left(\tilde R(\theta_t\omega)\hat R(\theta_t\omega)\right)}{t}=0,~~~\textnormal{for all }\omega\in\Omega.
    \end{align*}
    Let $\omega\in\Omega$ be arbitrary. We use that $\tilde R,\hat R$ are both tempered from above and therefore get 
    \begin{align*}
        0&\le\limsup\limits_{t\to\pm\I}\frac{\ln^+\left(\tilde R(\theta_t\omega)\hat R(\theta_t\omega)\right)}{t} \\
        &\le\limsup\limits_{t\to\pm\I}\frac{\ln^+\tilde R(\theta_t\omega) + \ln^+\hat R(\theta_t\omega)}{t} \\
        &\le\limsup\limits_{t\to\pm\I}\frac{\ln^+\tilde R(\theta_t\omega) }{t} + \limsup\limits_{t\to\pm\I}\frac{\ln^+\hat R(\theta_t\omega) }{t}
        =0.
    \end{align*}
    If $\tilde R,\hat R$ are tempered from below we consider $\frac{1}{\tilde R},\frac{1}{\hat R}$, which are tempered from above, and the claim follows. \qed
\end{proof}

\subsection{Existence of local random center manifolds}
In this subsection we give a short overview on the main results of \cite{KN23} on the existence of local center manifolds for~\eqref{sde1}. 
\begin{definition}
    We call a random set $\cM^c(\bfW)$, which is invariant with respect to $\varphi$, a center manifold if this can be represented as
    \begin{align*}
        \cM^c(\bfW) = \left\{ \xi+ h^c(\xi,\bfW): \xi\in\cX^c  \right\},
    \end{align*}
    where $h^c(\cdot,\bfW):\cX^c\to\cX^s$ is Lipschitz continuous and differentiable in zero.~Moreover, $h^c(0,\bfW)=0$ and $\cM^c(\bfW)$ is tangent to $\cX^c$ at the origin, meaning that the tangency condition $\txtD h^c(0,\bfW)=0$ is satisfied.\\
    We call a random set $\cM^c_{loc}(\bfW)$, which is invariant with respect to $\varphi$, a local center manifold if 
    \begin{align}\label{local center manifold def}
        \cM^c_{loc}(\bfW) = \left\{ \xi+ h^c(\xi,\bfW): \xi\in B_{\cX^c}(0,\rho(\bfW))  \right\},
    \end{align}
    for some tempered from below radius $\rho(\textbf{W})$ and $h^c(\cdot,\bfW)$ fulfills the same conditions as for $\cM^c(\bfW)$.
\end{definition}
The next results were established under the following conditions for $F$ and $G$:
\begin{description}
	\item [(\textbf{F})\label{f}] $F:\cX\to \cX\in C^{1}_{b}$ is Lipschitz continuous with $F(0)=\txtD F(0)=0$;
	\item [(\textbf{G})\label{gi}] $G:\cX\to\mathcal{L}(\cV,\cX)\in C^{3}_{b}$ is Lipschitz continuous with $G(0)=\txtD G(0)=\txtD^2 G(0)=0$.
\end{description}

\begin{remark}
     We conjecture that the assumption $\txtD G(0)=\txtD^2G(0)=0$ could be dropped, in particular in the small noise regime. The existence of center manifolds for RDEs has been established under weaker assumptions in~\cite[Theorem 2.15]{GhaniVarzanehRiedel}. However, these techniques do not rely on the Lyapunov-Perron method and the graph structure of the manifold cannot be directly obtained. Since the approximation of the function representing the graph structure of the manifold is the starting point of this work, we stick to the stronger assumptions on $G$ imposed in~\cite{KN23}. 
\end{remark}

As already mentioned, the existence of a local center manifold given by \eqref{local center manifold def} relies on the Lyapunov-Perron method. The Lyapunov-Perron map $J$ associated to the RDE~\eqref{sde1} is formally given by 
\begin{align*}
    J(\bfW,U,\xi)[t] &= S^c(t) \xi + \int_0^t S^c(t-r) P^cF(U_r) ~\txtd r + \int_{0}^t S^c(t-r) P^cG(U_r) ~\txtd \bfW_r\\& + \int_{-\I}^t S^s(t-r) P^sF(U_r) ~\txtd r 
    + \int_{-\I}^t S^s(t-r) P^sG(U_r) ~\txtd \bfW_r.
\end{align*}
Because we are working in a rough path setting, the estimates of the integrals with respect to $\bfW$ depend on the H\"older norms of $W$ and $\mathbb{W}$. In order to deal with these expressions, the integrals will be discretized. Moreover, we need a suitable function space to work with after the discretization. This is discussed in detail in Section \ref{sec:discretization}. For now, we state the following theorems in an informal way, such that we can give an overview of the results in~\cite{KN23} without getting into technical details.
\begin{theorem}{\em{(\cite[Theorem 4.7]{KN23}})}
\label{contraction} Let Assumptions~\ref{ass:linearpart},~\nameref{f}  and~\nameref{gi} 
hold true.~Under a suitable gap condition, the discretized Lyapunov-Perron map $J$ possesses a unique fixed-point 
$\Gamma=\Gamma(\xi, \bfW)$. 
\end{theorem}

\begin{lemma}{\em{(\cite[Lemma 4.11]{KN23}})}
\label{localcman} There exists a tempered from below random variable 
$\rho({\bfW})$ such that the local center manifold of~\eqref{sde1} is given 
by the graph of a Lipschitz function, namely
\begin{equation}
\cM^{c}_{loc} ({\bfW})=\{ \xi + h^{c}(\xi,\bfW) 
: \xi\in B_{\cX^{c}}(0,\rho({\bfW})) \}.
\end{equation}
where we define
\begin{align*}
&	h^{c}(\xi,\bfW):=P^{s}\Gamma(\xi,\bfW)[0],
\end{align*}
consequently 
\begin{align}\label{h:cm}
	h^{c}(\xi,{\bfW}) &=  
	\int\limits_{-\I}^{0} S^{s}(-r) P^{s }F(\Gamma(\xi,{\bfW})[r]) ~\txtd r \nonumber \\
    &+\int\limits_{-\I}^{0} 
	S^{s}(-r)P^{s } G(\Gamma(\xi,{\bfW})[r]) ~\txtd \bm{{W}}_{r}.
\end{align}
\end{lemma}

The following statement establishes the smoothness of the manifold. 
\begin{theorem}{\em{(\cite[Theorem 5.2]{KN23}})}
\label{glattheit} Assume that $F$ is $C^{m}$ and $G$ is $C^{m+3}_{b}$ 
for $m\geq 1$. If a suitable gap condition holds, then $\cM^{c}_{loc}({\bfW})$ is a local $C^{m}$-center manifold.
\end{theorem}
Based on this result, one naturally expects to approximate $h^c$ by a suitable Taylor-like expansion. This is the topic of the next section and the main result of this work. 
\section{Main Result}\label{sec:main}
\subsection{Our setting}\label{setting}
The main goal is to approximate $h^c$ by polynomials, following a similar approach as in the deterministic case described in Appendix~\ref{sec:carr}. Here we derive such an approximation result in the stochastic setting based on rough path theory.\\

We first fix a $\cV$-valued $\gamma$-H\"older geometric rough path cocycle $\bfW=(W,\W)$ as specified in Definition~\ref{rough path cocycle}. Consider the following system of coupled RDEs given by 
\begin{align}\label{SDE without cut-off}
\begin{cases}
    &\txtd x = (A^c x + F^c(x,y))~\txtd t + G^c(x,y)~\txtd \textbf{W}_t\\
&    \txtd y = (A^s y + F^s(x,y))~\txtd t + G^s(x,y)~\txtd \textbf{W}_t \\
   & x_0=\xi\in B_{\cX^c}(0,\rho(\bfW)),~ y_0\in \cX^s,
    \end{cases}
 \end{align}
where $\cX^c,~\cX^s$ both have dimension one and $y_0$ belongs to the manifold $\cM^c_{loc}(\bfW)$. Moreover, the radius $\rho(\bfW)$ is tempered from below and smaller than one. Later in Lemma \ref{lem:omit cut-off} we will specify the assumption for $\rho(\bfW)$. 
\begin{remark}
    We consider the setting in which $x$ and $y$ are one-dimensional not to overload the notation. However, most of the proofs and ideas remain the same in higher dimensions. In particular, the proofs in Section \ref{sec:fp} are independent of the dimension of the considered system. 
    In contrast to $\cX$, $\cV$ is not necessarily of dimension two but is $d$-dimensional for some $d\in\N$.
\end{remark}
\begin{notation}\label{notation:X norm}
 Throughout the manuscript we use the following convention 
 \[ \|\cdot\|_{\cX} =\|\cdot\|_{\cX^c} +\|\cdot\|_{\cX^s}. \]
 For an element $\xi\in \cX^c$ we have that $(\xi,0)\in \cX$ and analogously if $\xi\in \cX^s$ then $(0,\xi)\in \cX$.
\end{notation}

\begin{assumptions}    \item[(\textbf{F})\label{F}] $F^c:\cX^c \times \cX^s \to \cX^c\in C^m$ and $F^s:\cX^c \times \cX^s \to \cX^s\in C^m$ for $m\ge1$ are locally Lipschitz continuous in both variables with $F^c(0,0)=\txtD F^c(0,0)= F^s(0,0)=\txtD F^s(0,0)=0$. 
\item[(\textbf{G})\label{G}] $G^c: \cX^c\times\cX^s\to \cL(\cV,\cX^c)\in C^{m+3}$ and $G^s: \cX^c\times\cX^s\to \cL(\cV,\cX^s)\in C^{m+3}$ for $m\ge1$ are locally Lipschitz continuous in both variables with $G^c(0,0)=\txtD G^c(0,0)=\txtD^2G^c(0,0)=G^s(0,0)=\txtD G^s(0,0)=\txtD^2G^s(0,0)=0$.
\end{assumptions}

\begin{remark}\label{cutofflocal}
Note that $F^{c/s}$ and $G^{c/s}$ can be made globally Lipschitz continuous by a standard cut-off argument.~This is enough for our aims since we are only interested in the dynamics in a neighbourhood of the origin.~We perform a similar argument in the next subsection at the level of paths and not in the phase space $\cX$ as required here.~Therefore we refrain from performing the cut-off argument in $\cX$ here. 
\end{remark}

\subsection{A truncation argument}\label{sec:cutoff}
We are only interested in a local approximation of the center manifold, i.e.~in a neighborhood of the origin, which is the fixed point of the system~\eqref{SDE without cut-off}. Hence, we truncate the coefficients $F^s, F^c, G^s, G^c$ outside a ball around the origin, similar to~\cite[Section 2.1]{KN23}. In contrast to the classical cut-off techniques, where one performs the truncation argument at the level of vector fields, we work here at the level of paths in order to incorporate the space $\cD^{2\gamma}_W$.  \\

First of all, we set \[\cD^{2\gamma}_W([0,1];\cX):= \cD^{2\gamma}_W([0,1],\cX^c)\times \cD^{2\gamma}_W([0,1],\cX^s)\]
    and emphasize that by $((x,y), (x,y)')\in \cD^{2\gamma}_W([0,1],\cX^c)\times \cD^{2\gamma}_W([0,1],\cX^s)$ we mean that $(x,x')\in \cD^{2\gamma}_W([0,1],\cX^c)$ and $(y,y')\in \cD^{2\gamma}_W([0,1],\cX^s)$. 
\begin{definition}\label{cut-off def}
    Let $\chi:\cD^{2\gamma}_W([0,1],\cX^c)\times \cD^{2\gamma}_W([0,1],\cX^s)\to \cD^{2\gamma}_W([0,1],\cX^c)\times \cD^{2\gamma}_W([0,1],\cX^s)$ be a Lipschitz function defined as
\begin{align*}
    \chi(x,y)[\cdot] := \begin{cases}
        (x,y) & \Dnorm{(x,y),(x',y')}\le 1/2 \\
        0 & \Dnorm{(x,y),(x',y')}\ge 1.
    \end{cases}
\end{align*}
For a positive number $R>0$ we define $\chi_R(x,y)[\cdot]:=R\chi((x,y)/R)[\cdot]$. Hence
\begin{align*}
    \chi_R(x,y)[\cdot] := \begin{cases}
        (x,y) & \Dnorm{(x,y),(x',y')}\le R/2 \\
        0 & \Dnorm{(x,y),(x',y')}\ge R .
    \end{cases}
\end{align*}
We now compose the nonlinear terms $F^{c/s}$ and $G^{c/s}$ with $\chi_R$. This means that we define 
$$F^{c/s}_R(x,y)[\cdot]:=F^{c/s}\circ\chi_R(x,y)[\cdot],~~~G^{c/s}_R(x,y)[\cdot]:=G^{c/s}\circ\chi_R(x,y)[\cdot].$$
\end{definition}
\begin{lemma}
    Let $(x,y),(\tilde{x},\tilde y)\in \cD^{2\gamma}_W([0,1];\cX)$ and let $H\in C^1$ be locally Lipschitz continuous with $H(0,0)=\txtD H(0,0)=0$. Then
    \begin{align*}
        &\Dnorm{H_R(x,y)[\cdot]-H_R(\tilde x,\tilde y)[\cdot], H_R(x,y)[\cdot]'-H_R(\tilde x,\tilde y)[\cdot]'} \\
        &\le C_H R\left(\Dnorm{x-\tilde{x}, x' - \tilde x'}+\Dnorm{y-\tilde y, y' - \tilde y'}\right),
    \end{align*}
    where $C_H$ depends on the Lipschitz constants of $H$ and $\txtD H$.
\end{lemma}
\begin{proof}
Since $\txtD H(0,0)=0$ we get by Taylor's formula that 
    \begin{align*}
        &\Dnorm{H_R(x,y)[\cdot]-H_R(\tilde x,\tilde y)[\cdot], H_R(x,y)[\cdot]'-H_R(\tilde x,\tilde y)[\cdot]'} \\
        &\le \int_0^1\Dnorm{\txtD H\left(r\chi_R(x,y)[\cdot] +(1-r)\chi_R(\tilde x,\tilde y)[\cdot] \right), \txtD H\left(r\chi_R(x,y)[\cdot] +(1-r)\chi_R(\tilde x,\tilde y)[\cdot] \right)'} ~\txtd r\\
        &~~~~\Dnorm{(x,y)-(\tilde{x},\tilde y), (x',y')-(\tilde{x}',\tilde y')} .        
    \end{align*}
    Moreover, we can use the local Lipschitz continuity of $\txtD H$ with Lemma \ref{D2gamma bound for non-random integral} and the definition of $\chi_R$ to infer that
    \begin{align*}
        &\Dnorm{H_R(x,y)[\cdot]-H_R(\tilde x,\tilde y)[\cdot], H_R(x,y)[\cdot]'-H_R(\tilde x,\tilde y)[\cdot]'} \\
        &\le C_H\max\left\{\Dnorm{\chi_R(x,y)[\cdot], \chi_R(x,y)[\cdot]'},\Dnorm{\chi_R(\tilde{x},\tilde y)[\cdot], \chi_R(\tilde{x},\tilde y)[\cdot]'} \right\}\\
        &~~~~~~\times\Dnorm{(x,y)-(\tilde{x},\tilde y), (x',y')-(\tilde{x}',\tilde y')}\\
        &\le C_H R\Dnorm{(x,y)-(\tilde{x},\tilde y), (x',y')-(\tilde{x}',\tilde y')}\\
        &\le C_H R\left(\Dnorm{x-\tilde{x}, x' - \tilde x'}+\Dnorm{y-\tilde y, y' - \tilde y'}\right).
    \end{align*}
    \qed \\
\end{proof}
Based on this lemma, the following bounds can be shown for $F_R$ and $G_R$ by similar computations as in~\cite{KN23}. 
\begin{lemma}{\em{(\cite[Lemma 2.14]{KN23}})}\label{lip:fr}
Let $((x,y),(x',y')), (\left(\tilde{x},\tilde y\right),\left(\tilde{x}',\tilde y'\right))\in \cD^{2\gamma}_W([0,1],\cX^c)\times \cD^{2\gamma}_W([0,1],\cX^s)$. Then there exists a constant $C_F=C[F,\chi]$ such that
\begin{align}\label{fr}
\|F_{R}(x,y)[\cdot] - F_{R}\left(\tilde{x},\tilde y\right)[\cdot]\|_{\infty}  &\leq C_F R\left(\Dnorm{x-\tilde{x}, x' - \tilde x'}+\Dnorm{y-\tilde y, y' - \tilde y'}\right).
\end{align}
\end{lemma}
\begin{lemma}{\em{(\cite[Lemma 2.15]{KN23}})}\label{composition:cutoff} 
	Let $((x,y),(x',y')), (\left(\tilde{x},\tilde y\right),\left(\tilde{x}',\tilde y'\right))\in \cD^{2\gamma}_W([0,1],\cX^c)\times \cD^{2\gamma}_W([0,1],\cX^s)$. Then there exists a constant $C_G=C[\|W\|_{\gamma},\norm{\W}_{2\gamma},G,\chi]$ such that
\begin{align}\label{composition:r}
	&\|G_{R}(x,y)[\cdot] - G_R(\tilde x, \tilde y)[\cdot], (G_{R}(x,y)[\cdot]- G_R(\tilde x, \tilde y)[\cdot])'\|_{\cD^{2\gamma}_{W}} \nonumber \\
    &\leq C_G R \left(\Dnorm{x-\tilde{x}, x' - \tilde x'}+\Dnorm{y-\tilde y, y' - \tilde y'}\right).
	\end{align}
\end{lemma}
\begin{remark}
    Later we will perform the truncation argument for a random variable $R=R(W)$,  such that we can keep track of the size of the noise reflected by the H\"older norms $\|W\|_\gamma$ respectively $\|\mathbb{W}\|_{2\gamma}$. In particular, we will see that $R$ is tempered from below. 
\end{remark}

From now on we consider the original system \eqref{SDE without cut-off} with the cut-off applied to the nonlinearities given by 
\begin{align}\label{SDE}
\begin{cases}
    &\txtd x = (A^c x + F^c_R(x,y)[t])~\txtd t + G_R^c(x,y)[t]~\txtd \textbf{W}_t\\
&    \txtd y = (A^s y + F^s_R(x,y)[t])~\txtd t + G_R^s(x,y)[t]~\txtd \textbf{W}_t \\
   & x(0)=\xi\in B_{\cX^c}(0,\rho(\bfW)),~ y(0)\in \cX^s.
    \end{cases}
\end{align}
\begin{remark}
\begin{itemize}
    \item [1)] We apply the cut-off function specified in Definition \ref{cut-off def} uniformly in $x$ and $y$ for $F^{c/s}$ and $G^{c/s}$. 
    This is justified by the fact that our results are local, i.e.~hold in a suitable neighborhood of the origin. We further assume that the cut-off radius $R$ is smaller than one and determine later a more precise bound for $R$ depending on the relevant parameters of the system and on the rough path ${\bf W}=(W,\mathbb{W})$.
    \item [2)] We emphasize that the nonlinearities $F^{c/s}, ~G^{c/s}$ are initially time-independent. After applying the cut-off we obtain time-dependent nonlinearities, since the cut-off function truncates the $\|\cdot\|_{\cD^{2\gamma}_W}$-norm and therefore depends on the whole path of the argument. We write $F_R^{c/s}(x,y)[t]$ and $G_R^{c/s}(x,y)[t]$ in order to emphasize this path-dependence.
    \item [3)] As the cut-off bounds the $\cD^{2\gamma}_W$-norm, it also bounds $\|\cdot\|_\infty$. Hence, our path-dependent coefficients are bounded, i.e.~satisfy \[(F_R^{c/s},(F_R^{c/s})')\in C_b^m\left(\cD^{2\gamma}_W([0,1],\cX^c) \times \cD^{2\gamma}_W([0,1],\cX^s),\cD^{2\gamma}_W([0,1],\cX^{c/s}) \right)\] and \[(G_R^{c/s}, (G_R^{c/s})')\in C^{m+3}_b\left(\cD^{2\gamma}_W([0,1],\cX^c) \times \cD^{2\gamma}_W([0,1],\cX^s),\cD^{2\gamma}_W([0,1],\cL(\cV, \cX^{c/s})) \right).\] 
    \item [4)] The well-posedness of the modified system~\eqref{SDE} follows by a fixed-point argument similar to~\cite[Theorem 2.16]{KN23}.
\end{itemize}
\end{remark}

\subsection{Approximation of local center manifolds}\label{sec:approx}
The goal is to approximate in a neighbourhood of the origin the function $h^c(x,\bfW)$, which represents the graph of the $C^m$ center manifold given by 
$$\cM^c_{loc}(\bfW) = \{x+h^c(x,\bfW): x\in B_{\cX^c}(0,\rho(\bfW))\}$$
as stated in Lemma~\ref{localcman}.

\paragraph{Ansatz.} We approximate $h^c$ by a Taylor expansion of order $q$ for $2\le q \le m$, given by 
\begin{equation}\label{ansatz}
y_t=h^c(x_t,\bfW_t)=\phi(x_t):=\sum\limits_{i=1}^{q} \alpha_i(\bfW_t) x_t^i,
\end{equation}
where the coefficients $\alpha_i=\alpha_i(\bfW_t)\in\cL(\cX^c,\cX^s)$ solve controlled RDEs given by
\begin{align}\label{alpha RDE}
    \txtd \alpha_i = [A^{\alpha_i}\alpha_i + f_i]~\txtd t + g_i ~\txtd \textbf{W}_t,~~i\in\{1,\ldots q\}. 
\end{align}
\begin{remark}
\begin{enumerate}
    \item  The dependence of the solutions $\alpha_i$ on the rough path $\bfW$ i.e.~$\alpha_i=\alpha_i(\bfW_t)$, will be dropped for notational simplicity, whenever this is clear from the context. Similarly we write for the coefficients $f_i=f_i(\alpha_1,\ldots, \alpha_{i-1})$ respectively $g_i=g_i(\alpha_1,\ldots,\alpha_{i})$ for $i\in\{1,\ldots,q\}$. 
    \item The system of RDEs can be solved recursively.~This will be seen later from the equations defining $f_i$ \eqref{M def} and $g_i$ \eqref{tilde M def} and in the Examples \ref{Chekroun linear example} and \ref{Chekroun nonlinear example}.~The strategy is to first determine the coefficients $f_1,g_1$ and afterwards solve the RDE for $\alpha_1$.~Thereafter we determine $f_2,g_2$, which can potentially depend on $\alpha_1$, and solve the RDE for $\alpha_2$. This procedure can be repeated until $\alpha_q$ is computed.~An algorithm for this computation will be provided in a future work.
    \item The ansatz~\eqref{ansatz} is inspired by a local Taylor expansion of $h^c(x,\bfW)$ for $x\in B_{\cX^c}(0,\rho(\bfW))$ 
    \begin{align*}
        h^c(x_t,\bfW_t) = \sum_{i=1}^q \txtD^i_\xi h^c(\xi,\bfW_t)|_{\xi=0} x^i_t,
    \end{align*}
    where the coefficients $\alpha_i(\bfW_t):=\txtD^i_{\xi} h^c(\xi,\bfW_t)|_{\xi=0}$ for $i\in\{1,\ldots,q\}$ can be determined from the Lyapunov-Perron map using~\eqref{h:cm}. 
    \item A similar procedure is exploited in~\cite{PoetzscheRasmussen} for the approximation of stable and unstable manifolds for deterministic nonautonomous systems.~The coefficients of the Taylor approximation solve nonautonomous ODEs, which are derived from the Lyapunov-Perron map.~The ODEs have also a hierarchical structure, so the coefficients of the approximation only depend on coefficients of lower order. Hence, the system is solved recursively, similar to the case investigated here. 
\end{enumerate}
\end{remark}
\begin{notation}
    From here on we only write $\cD^{2\gamma}_W$ suppressing the time interval $[0,1]$ and the phase space whenever they are clear from the context.~We emphasize that $(x,x')\in\cD^{2\gamma}_W\left([0,1];\cX^c \right)$, $(y,y')\in\cD^{2\gamma}_W\left([0,1];\cX^s \right)$ and for all $i\in\{1,\dots,q\}$ we have $(\alpha_i,\alpha_i')\in\cD^{2\gamma}_W\left([0,1];\cL\left(\cX^c,\cX^s\right) \right)$.
\end{notation}

\begin{assumptions}\label{ass:alpha}
 \begin{itemize}
 \item [1)] We assume that there exist $\delta_i>0$  for $i\in \{1,\dots,q\}$ such that the semigroups generated by $A^{\alpha_i}$ are exponentially stable, meaning that
\begin{align}
    \Xnorm{e^{A^{\alpha_i}t}x}\le M_{A^{\alpha_i}} e^{-\delta_i t}\Xnorm{x}, ~\mbox{for } t\geq 0 \mbox{ and } x\in \cX.\label{delta}
\end{align}
     \item [2)]    We further assume that the RDEs~\eqref{alpha RDE} admit global solutions. Since these solutions are defined pathwise, they automatically generate RDS, see~\cite[Lemma 3.7]{KN23}. 
     \item [3)]  Moreover, these solutions are assumed to be stationary in the RDS sense. This means that for the corresponding RDS $\varphi_i$ generated by $\alpha_i$, for $i\in\{1,\ldots,q\}$, we have
\[ \varphi_i(t,W,\alpha_i(\bfW))=\alpha_i(\Theta_t \bfW), \]
where $W\in\Omega$, where the canonical probability space $(\Omega,\cF,\P,(\theta_t)_{t\in\R})$ is introduced in Example \ref{ex:stationary}. 
An example of a stationary solution is also provided below in Example~\ref{ex:stationary}. 
\end{itemize}
    
\end{assumptions}

\begin{remark}
     The stationarity of the solutions of~\eqref{alpha RDE} is required in order to guarantee the invariance of the corresponding local center manifolds with respect to the shift $\Theta$, i.e.~$\varphi(t,W,\cM^c_{loc}(\bfW))\subseteq \cM^c_{loc}(\Theta_t \bfW)$.~This assumption is natural in the context of invariant manifolds for stochastic systems, see~\cite{ChekrounLiuWang} and~\cite{ChekrounLiuWang example} for a similar setting for SDEs with linear multiplicative noise. Moreover, in \cite{GhaniVarzanehRiedel,Ibounds,NA}, R(P)DEs are linearized along a stationary solution in order to apply the multiplicative ergodic theorem and infer the existence of invariant manifolds based on the sign of the top Lyapunov exponent. 
\end{remark}
\begin{example}\label{ex:stationary}
    A standard example for a stationary solution of a linear SDE driven by a two-sided Brownian motion $(B_t)_{t\in\R}$
\begin{align}\label{lin:sde}
    \txtd z = - z~\txtd t + \txtd B_t
\end{align}
is given by the stationary Ornstein-Uhlenbeck process
\begin{align*}
    z(t,B):=z(\theta_tB) = \int_{-\I}^t e^{-(t-s)}\txtd B_s = \int_{-\I}^0 e^s ~\txtd \theta_tB_s = \int_{-\I}^0e^s\theta_tB_s~\txtd s.
\end{align*}
Here $\theta$ denotes the usual Wiener-shift defined on the canonical probability space $(C_0(\R),\cB(\R),\mathbb{P} )$ associated to a two-sided Brownian motion, i.e.~$\theta_t B_s:= B_{t+s} - B_s$ for $s,t\in\R.$ Based on this property and denoting by $\Phi$ the RDS generated by~\eqref{lin:sde}, one easily checks that the stationary Ornstein-Uhlenbeck process is a random fixed point for $\Phi$, meaning that
\begin{align*} \Phi(t,B,z(B))&= e^{-t} z(B) +\int_0^t e^{-(t-s)}~\txtd B_s = e^{-t} \int_{-\infty}^0 e^{s}~\txtd B_s + \int_0^t e^{-(t-s)}~\txtd B_s\\
&=\int_{-\infty}^t e^{-(t-s)}~\txtd B_s =z(\theta_tB).  \end{align*}
For more details on the existence of stationary solutions for S(P)DEs with fractional Brownian motion we refer to~\cite{MS} and for RDEs to \cite[Section 2.2]{GhRiedel}.~The results in~\cite{GhRiedel} are stated for rough delay differential equations and setting the delay term equal to zero, we recover our framework.
\end{example}

\subsection{Rough path estimates}\label{sec:rpest}

Throughout this subsection we consider the system~\eqref{SDE} on the time interval $[0,1]$. Afterwards we extend our results to arbitrary time intervals. In this section we provide suitable estimates in the rough path norm for the solution of~\eqref{SDE} and of the systems of controlled RDEs~\eqref{alpha RDE}.~Due to Remark~\ref{cutofflocal} we can assume that the nonlinear terms are globally Lipschitz.\\

Furthermore, throughout this and the following sections $C$ denotes an universal constant which varies from line to line.~The dependence on the constant $C$ on certain parameters is denoted by $C[\cdot]$. 
\begin{lemma}\label{x is 2 gamma and xi bound}
    Let $x$ be the solution of \eqref{SDE} with initial condition $\xi\in B_{\cX^c}(0,\rho(\bfW))$ on the time interval $[0,1]$. Then   $(x,x')\in\cD^{2\gamma}_W$ and there exists a constant $C_x:=C[\Xnorm{A^c}]\ge 1$ such that
    $$\Dnorm{x,x'}\le C_x\Xnorm{\xi}.$$
\end{lemma}
\begin{proof}
    Applying the variation of constants formula to \eqref{SDE} for $t\in[0,1]$, we get
    \[ x_t = S^c(t)\xi + \int_0^t S^c(t-r) F^c_R(x,y)[r]~\txtd r \int_0^t S^c(t-r) G^c_R(x,y)[r]~\txtd \bfW_r .\]
    By Lemma \ref{integral bound} we have that $x' = G^c_R(x,y)[\cdot]$. Furthermore, using Lemma \ref{integral bound} we infer that
    \begin{align*}
        \Dnorm{x,x'}
        &\le C[\Xnorm{A^c}]\Big(\Xnorm{\xi} + \norm{F_R^c(x,y)[\cdot]}_\I \\
        &~~+ \left(1+\norm{W}_\gamma\right)\left(\norm{W}_\gamma+\norm{\mathbb{W}}_{2\gamma}\right)\Dnorm{G_R^c(x,y)[\cdot],G_R^c(x,y)[\cdot]'}\Big).
    \end{align*}
    Moreover, we can use Lemma \ref{lip:fr} to bound $\norm{F_R^c(x,y)[\cdot]}_\I$. Since we are interested in $y\in\cM_{loc}^c(\bfW)$, we know that $y=h^c(x,\bfW)$. This further results in 
    \begin{align*}
        \norm{F_R^c(x,y)[\cdot]}_\I
        &\le C_{F^c} R \left(\Dnorm{x,x'} + \Dnorm{y,y'}\right) \nonumber\\
        &\le C_{F^c} R \left(\Dnorm{x,x'} + \Dnorm{h^c(x,W), h^c(x,W)'}\right) \nonumber\\
        &\le C_{F^c} R \left(1 + L_{h^c}\right) \Dnorm{x,x'}.
    \end{align*}
    Here $L_{h^c}$ denotes the Lipschitz constant of $h^c$. 
    Analogously we get with Lemma \ref{composition:cutoff}
    \begin{align*}
        \Dnorm{G_R^c(x,y)[\cdot],G_R^c(x,y)[\cdot]'}
        &\le C_{G^c} R \left(\Dnorm{x,x'} + \Dnorm{y,y'}\right) \nonumber\\
        &\le C_{G^c} R \left(\Dnorm{x,x'} + \Dnorm{h^c(x,W), h^c(x,W)'}\right) \nonumber\\
        &\le C_{G^c} R \left(1 + L_{h^c}\right) \Dnorm{x,x'}.
    \end{align*}
  Hence we infer
    \begin{align*}
        \Dnorm{x,x'}
        &\le C[\Xnorm{A^c}]\Xnorm{\xi} + C[\Xnorm{A^c}] \Big(C_{F^c}  \\ &~~+\left(1+\norm{W}_\gamma\right)\left(\norm{W}_\gamma+\norm{\mathbb{W}}_{2\gamma}\right) C_{G^c}  \Big) \left(1 + L_{h^c}\right)R \Dnorm{x,x'} .
    \end{align*}
    By rearranging the terms we obtain that 
    \begin{align*}
        \Dnorm{x,x'}
        \le \frac{C[\Xnorm{A^c}]}{1-Q(\bfW)R}\Xnorm{\xi},
    \end{align*}
    where 
    \begin{align}\label{def:Q}
        Q(\bfW):=C[\Xnorm{A^c}] \Big(C_{F^c} +\left(1+\norm{W}_\gamma\right)\left(\norm{W}_\gamma+\norm{\mathbb{W}}_{2\gamma}\right) C_{G^c}\Big)\left(1 + L_{h^c}\right).
    \end{align}
    By choosing \begin{align}\label{def:R_1}
        R_1(\bfW)\le\min\left\{\frac{1}{2Q(\bfW)},1\right\},
    \end{align}
    the denominator above is positive and we get 
    \begin{align*}
        \Dnorm{x,x'}
        \le C[\Xnorm{A^c}]\Xnorm{\xi}.
    \end{align*}    
    \qed \\
\end{proof}
Recalling our ansatz~\eqref{ansatz} we state the following result. 
\begin{lemma}\label{phi D2gamma estimate}
    Let $(x,x')\in \cD^{2\gamma}_{W}$ be the solution of the first equation of \eqref{SDE} with initial condition $\xi\in B_{\cX^c}(0,\rho(\bfW))$ and let $(\alpha_i,\alpha_i')\in \cD^{2\gamma}_{W}$ for $i\in\{1,\ldots q\}$ be the stationary solutions of the system \eqref{alpha RDE}. We further set 
    $$\phi(x_t) := \sum_{i=1}^{q}\alpha_ix_t^i.$$
    Then 
    $$\left(\phi(x),\phi(x)' \right)\in\cD^{2\gamma}_W,$$
    and
    $$\Dnorm{\phi(x),\phi(x)'}\le  CC_x^qC_\alpha q^2\left(1+\norm{W}_\gamma\right)^2 \Xnorm{\xi} ,$$
where we set
    \begin{align}\label{C_alpha def}
        C_\alpha:=\max\limits_{i\in\{1,\dots,q\}} \Dnorm{\alpha_i,\alpha_i'}.
    \end{align}
\end{lemma}
\begin{proof}
    We use Lemma \ref{lem:addition and multiplication of RP} and get
    \begin{align*}
        \Dnorm{\sum_{i=1}^{q}\alpha_ix^i, \sum_{i=1}^{q}\alpha_i' x^i + i\alpha_ix^{i-1}x'}
        &\le C\left(1+\norm{W}_\gamma\right)^2\sum_{i=1}^q \Dnorm{\alpha_i,\alpha_i'}\Dnorm{x^i,ix^{i-1}x'}.
    \end{align*}
    Next we simplify $\Dnorm{x^i,ix^{i-1}x'}$. For this we use the definition of the $\cD^{2\gamma}_W$-norm \eqref{Dnorm}.
    \begin{align*}
        \Dnorm{x^i,ix^{i-1}x'} 
        &= \norm{x^i}_\I + \norm{ix^{i-1}x'}_\I + \norm{ix^{i-1}x'}_\gamma + \norm{R^{x^i}}_{2\gamma} \\
        &\le \norm{x}_\I^i + i\norm{x}^{i-1}_\I\norm{x'}_\I + i\norm{x}^{i-1}_\I\norm{x'}_\gamma + \norm{R^{x^i}}_{2\gamma}.
    \end{align*}
    For $\norm{R^{x^i}}_{2\gamma}$ we get 
    \begin{align*}
        \norm{R^{x^i}}_{2\gamma} &= \sup_{s,t\in[0,1]} \frac{\Xnorm{x^i_t-x_s^i-ix^{i-1}_sx_s'W_{s,t}}}{|t-s|^{2\gamma}} \\
        &\le i\norm{x}_\I^{i-1}\norm{R^x}_{2\gamma}.
    \end{align*}
    Plugging this in we get
    \begin{align*}
        \Dnorm{x^i,ix^{i-1}x'} \le i \Dnorm{x,x'}^i.
    \end{align*}
    Combining the previous calculations, we obtain 
    \begin{align*}
        \Dnorm{\sum_{i=1}^{q}\alpha_ix^i, \left(\sum_{i=1}^{q}\alpha_ix^i\right)'} 
        &\le Cq\left(1+\norm{W}_\gamma\right)^2\sum_{i=1}^{q}\Dnorm{\alpha_i,\alpha_i'}\Dnorm{x,x'}^i .
    \end{align*}
    Using Lemma \ref{x is 2 gamma and xi bound} we get
    \begin{align*}
        &\Dnorm{\sum_{i=1}^{q}\alpha_ix^i, \left(\sum_{i=1}^{q}\alpha_ix^i\right)'} \\
        &\le Cq\left(1+\norm{W}_\gamma\right)^2C_x^q\sum_{i=1}^{q}\Dnorm{\alpha_i,\alpha_i'} \Xnorm{\xi}^{i},
    \end{align*} 
    where we use that $C_x\ge1$. 
    Now we note that for all $i\in\{1,\dots,q\}$ we have $\Dnorm{\alpha_i,\alpha_i'}<\I$. 
    Hence, 
    $$C_\alpha=\max_{i\in\{1,\dots,q\}}  \Dnorm{\alpha_i,\alpha_i'}<\infty.$$ 
    Then we get with $\Xnorm{\xi}\le \rho(\bfW)\le 1$
    \begin{align*}
        &\Dnorm{\sum_{i=1}^{q}\alpha_ix^i, \left(\sum_{i=1}^{q}\alpha_ix^i\right)'} \\
        &\le Cq\left(1+\norm{W}_\gamma\right)^2C_x^q\sum_{i=1}^{q}\Dnorm{\alpha_i,\alpha_i'} \Xnorm{\xi}^{i} \\
        &\le Cq\left(1+\norm{W}_\gamma\right)^2C_x^qC_\alpha\sum_{i=1}^{q} \Xnorm{\xi}^{i}\\
        &= Cq\left(1+\norm{W}_\gamma\right)^2C_x^qC_\alpha\Xnorm{\xi} \sum_{i=0}^{q-1} \Xnorm{\xi}^{i} \\
        &\le Cq^2\left(1+\norm{W}_\gamma\right)^2C_x^qC_\alpha\Xnorm{\xi} ,
    \end{align*}
    which proves the statement.  \qed \\
\end{proof}
\begin{lemma}\label{G(phi) in D2gamma}
    Let $(x,x')\in \cD^{2\gamma}_{W}$ be the solution of the first equation of \eqref{SDE} with initial condition $\xi\in B_{\cX^c}(0,\rho(\bfW))$, let $(\alpha_i,\alpha_i')\in \cD^{2\gamma}_{W}$ for $i\in\{1,\ldots q\}$ be the solutions of the system \eqref{alpha RDE} and let $$\phi(x_t) = \sum_{i=1}^{q}\alpha_ix_t^i.$$ Moreover, we consider a nonlinear function $G\in C^3_b$. Then
    the composition of the controlled rough path $(\phi(x),\phi(x)')\in \cD^{2\gamma}_W$ with $G$ is again a controlled rough path, i.e.
    $$(G(\phi(x),G(\phi(x))')\in\cD^{2\gamma}_W,$$
    and satisfies the estimate
    $$\Dnorm{G(\phi(x)), G(\phi(x))'}\le C\norm{G}_{C_b^2}(1+\norm{W}_\gamma)^6C_x^{2q}C_\alpha^2q^4\Xnorm{\xi}^2.$$
\end{lemma}
\begin{proof}
    From \cite[Lemma~7.3]{FritzHairer} we know that for $G\in C_b^3$ and $(Y,Y')\in\cD^{2\gamma}_W$ we have
    \begin{align*}
        \Dnorm{G(Y),G(Y)'}&=\Dnorm{G(Y),DG(Y)Y'}\\
        &\le C\norm{G}_{C_b^2}\left(\Xnorm{Y_0'}+\Dnorm{Y,Y'}\right)^2(1+\norm{W}_\gamma)^2.
    \end{align*}
    Since the $\cD^{2\gamma}_W$-norm contains  $\norm{Y'}_\I$, we have
    \begin{align*}
        \Dnorm{G(Y),G(Y)'}
        &\le C\norm{G}_{C_b^2}\Dnorm{Y,Y'}^2(1+\norm{W}_\gamma)^2.
    \end{align*}
    Now applying this inequality and Lemma \ref{phi D2gamma estimate} we get
    \begin{align*}
        \Dnorm{G(\phi(x)),G(\phi(x))'}
        &\le C\norm{G}_{C_b^2}(1+\norm
        {W}_\gamma)^2\Dnorm{\phi(x),\phi(x)'}^2 \\
        &\le C\norm{G}_{C_b^2}(1+\norm{W}_\gamma)^6 C_x^{2q}C_\alpha^2q^4\Xnorm{\xi}^2.
    \end{align*}
    \qed
\end{proof}

\begin{lemma}\label{G-g in D2gamma}
    Let $(x,x')\in\cD^{2\gamma}_W$ be the solution of the first equation of \eqref{SDE} with initial condition $\xi\in B_{\cX^c}(0,\rho(\bfW))$. Further let $(\alpha_i,\alpha_i')\in \cD^{2\gamma}_{W}$ for $i\in\{1,\ldots q\}$ be the solutions of the system \eqref{alpha RDE}, $ (z(x),z(x)')\in\cD^{2\gamma}_W$ and $$\phi(x_t) = \sum_{i=1}^{q}\alpha_ix_t^i.$$
    Moreover, let $G\in C_b^3$. Then
    $$\left(G(z(x) + \phi(x)) - \sum_{i=1}^{q} g_ix^i, \left(G(z(x) + \phi(x)) - \sum_{i=1}^{q} g_ix^i\right)' \right) \in \cD^{2\gamma}_W,$$
\end{lemma}
\begin{proof}
    Due to Lemma \ref{lem:addition and multiplication of RP} we know that sums and products of $\cD^{2\gamma}_W$ functions are again $\cD^{2\gamma}_W$ functions. Moreover, the composition of $G\in C_b^3$ and a $\cD^{2\gamma}_W$ function is a $\cD^{2\gamma}_W$ function due to Lemma \ref{G(phi) in D2gamma}. Hence, the claim follows. \qed
\end{proof}

\subsection{The general idea}\label{sec:generalidea}

The aim of this subsection is to formally illustrate the main idea required in order to show that~\eqref{ansatz} represents an approximation of $h^c$ by polynomials of order $q$. The idea is a modified version of the proof in the deterministic case \cite[Theorem.~3]{Carr}, which is summarized in Appendix \ref{sec:carr}.
Given our ansatz $y_t=\phi(x_t)$, the first step is to compute $\txtd y$ given by~\eqref{SDE} in two ways. First of all we obtain
\begin{align*}
    \txtd y_t = (A^s \phi(x_t) + F^s_R(x,\phi(x))[t])~\txtd t + G_R^s(x,\phi(x))[t] ~\txtd \textbf{W}_t.
\end{align*}
Moreover, since $\bfW =(W,\mathbb{W)}$ is a geometric rough path, we use the chain rule and get
\begin{align*}
    \txtd y_t &= \txtd \phi(x_t) = \sum_{i=1}^{q}\left(\alpha_i~ \txtd(x_t^i) + x_t^i\txtd \alpha_i\right)\\
    &= \sum_{i=1}^{q} i\alpha_i x_t^{i-1}~\txtd x_t + \sum_{i=1}^{q} \left(A^{\alpha_i}\alpha_i+f_i\right)x_t^i~\txtd t + \sum_{i=1}^{q}g_ix_t^i~\txtd \textbf{W}_t\\
    &= \left[\sum_{i=1}^{q} i\alpha_i x_t^{i-1}\left[A^c x_t + F^c_R(x,\phi(x))[t]\right] + \sum_{i=1}^{q} \left(A^{\alpha_i}\alpha_i+f_i\right)x_t^i\right]~\txtd t \\ 
    &+ \left[\sum_{i=1}^{q} i\alpha_i x_t^{i-1} G^c_R(x,\phi(x))[t] + \sum_{i=1}^{q}g_ix_t^i\right]~\txtd \textbf{W}_t.
\end{align*}
Here we recall that $f_i:=f_i(\alpha_1,\ldots \alpha_{i-1})$, $g_i:=g_i(\alpha_1,\dots,\alpha_{i})$ and define 
\begin{align}\label{tilde phi def}
    \tilde{\phi}(x_t):= \sum\limits_{i=1}^{q} \left(A^{\alpha_i}\alpha_i+f_i\right)x_t^i.
\end{align}
Comparing coefficients for the equations for $y$ given in~\eqref{SDE} and above, we get the invariance equations
\begin{align*}
    A^s\phi(x_t) + F^s_R(x,\phi(x))[t] 
    &= \sum_{i=1}^{q} i\alpha_i x_t^{i-1}[A^c x_t + F^c_R(x,\phi(x))[t]] + \tilde{\phi}(x_t), \\
    G_R^s(x,\phi(x))[t] &= \sum_{i=1}^{q} i\alpha_i x_t^{i-1} G^c_R(x,\phi(x))[t] + g_i x_t^i .
\end{align*}
We can eliminate the linear parts from the invariance equation by choosing $A^{\alpha_i}= A^s - i A^c$ and get the following invariance equations
\begin{align}
    A^{\alpha_i}&= A^s - i A^c \label{invariance equation A}\\
    \sum_{i=1}^q f_i x_t^i
    &=  F^s_R(x,\phi(x))[t] - \sum_{i=1}^{q} i\alpha_i x_t^{i-1} F^c_R(x,\phi(x))[t] , \label{invariance equation f_i}\\
    \sum_{i=1}^{q} g_i x_t^i &= G_R^s(x,\phi(x))[t] - \sum_{i=1}^{q} i\alpha_i x_t^{i-1} G^c_R(x,\phi(x))[t] \label{invariance equation g_i} .
\end{align}
We rewrite the previous expressions as
\begin{align}
    M\phi(x_t) &:= \sum_{i=1}^q f_i x_t^i - F^s_R(x,\phi(x))[t] + \sum_{i=1}^{q} i\alpha_i x_t^{i-1} F^c_R(x,\phi(x))[t] \label{M def}\\
    \tilde{M}\phi(x_t) &:= \sum_{i=1}^{q} g_ix_t^i - G_R^s(x,\phi(x))[t]  + \sum_{i=1}^{q} i\alpha_i x_t^{i-1} G^c_R(x,\phi(x))[t] . \label{tilde M def}
\end{align}
Here $M\phi(x)$ and $\tilde M\phi(x)$ indicate how good our approximation is, since $Mh^c(x,\bfW)=0$ and $\tilde Mh^c(x,\bfW)=0$. So the goal is to choose the coefficients $A^{\alpha_i},f_i$ and $g_i$ such that $M\phi(x)=0$ and $\tilde M\phi(x)=0$.
We will later see that we can control $M\phi$ and $\tilde M\phi$ choosing $q$ large and $A^{\alpha_i},~f_i,~g_i$ correctly. This means that we solve the system resulting from \eqref{M def} and \eqref{tilde M def} by setting the left-hand side equal to zero.
\begin{remark}
    For any system of the form~\eqref{SDE} it holds $\alpha_1\equiv0$. We know that $f_1=0$ because the nonlinearities $F_R^{c/s}$ have only terms of order at least two. Moreover, due to the assumption $G^{c/s}(0,0)=\txtD G^{c/s}(0,0)=0$ we exclude additive or linear multiplicative noise (which can be treated by flow transformations, see Section~\ref{examples}). So, $g_1$ is equal to zero. Hence $\alpha_1$ solves
    $$\txtd\alpha_1 = \left(A^s-A^c\right)\alpha_1~\txtd t.$$
    The equation is solved by $\alpha_1\equiv0$. With this we infer $\txtD\phi\left(0\right)=0$. This assumption is consistent with the deterministic autonomous case \ref{Carr Theorem}. As $\alpha_1\equiv0$, we have $$\phi(x_t)=\sum_{i=2}^q \alpha_i x_t^i.$$ 
\end{remark}

Keeping this in mind, we formally describe the main idea of the approximation proof of $h^c(x,\bfW)$ by $\phi(x)$. We remind the reader that our end goal is to approximate $h^c(\xi,\bfW)$ for $\xi\in B_{\cX^c}(0,\rho(\bfW))$ by $\phi(x_0)$, where $x$ solves \eqref{SDE} with $x_0=\xi$. 
\paragraph{Formal proof strategy.}
For $t\le 0$, $\xi\in \cX^c$ and $U_t=U_t(\xi,\bfW)\in\cX^s$ we let $J$ be the Lyapunov-Perron map associated to~\eqref{SDE} given by
\begin{align*}
    J(\bfW,x,U,\xi)[t] &= S^c(t) \xi + \int_0^t S^c(t-r) F^c_R(x, U)[r] ~\txtd r + \int_{0}^t S^c(t-r) G_R^c(x,U)[r] ~\txtd \bfW_r\\& + \int_{-\I}^t S^s(t-r) F^s_R(x, U)[r] ~\txtd r 
    + \int_{-\I}^t S^s(t-r) G_R^s(x,U)[r] ~\txtd \bfW_r.
\end{align*}
From \cite[Section 4.2]{KN23} we know that $h^c(\xi,\bfW) = P^s\Gamma(\xi,\bfW)[0],$ where $\Gamma = (P^c\Gamma,P^s\Gamma)$ is the fixed point of $J$. As we are only interested in the solution on the invariant manifold, we fix $x_t=P^c\Gamma(\xi,\bfW)[t]$ which is a solution of $$\txtd x_t = \left(A^c x_t + F^c_R(x,P^s\Gamma(\xi,\bfW)[\cdot])[t]\right)\txtd t + G_R^c(x,P^s\Gamma(\xi,\bfW)[\cdot])[t]~\txtd \bfW_t,~~~ x_0 = \xi.$$
Hence, $J$ becomes only a function of $U$ and $\xi$.~So from now on we consider the following Lyapunov-Perron map
\begin{align}\label{J cont def}
    J(\bfW,U,\xi)[t] &= S^c(t) \xi + \int_0^t S^c(t-r) F^c_R(x, U)[r] ~\txtd r  + \int_{0}^t S^c(t-r) G_R^c(x,U)[r] ~\txtd \bfW_r\nonumber\\& + \int_{-\I}^t S^s(t-r) F^s_R(x, U)[r] ~\txtd r 
    + \int_{-\I}^t S^s(t-r) G_R^s(x,U)[r] ~\txtd \bfW_r,
\end{align}
which is a contraction with a fixed point.~Next we define the map $\tilde{J}$ as
\begin{align}\label{tilde J formal def}
    \tilde{J}(\bfW,U,\xi)[t] :=P^sJ(\bfW,U+\phi(x),\xi)[t] - \phi(x_t).
\end{align}
We want to show that $\tilde{J}$ is a contraction on 
a suitable subspace $V$, which according to Banach's fixed point theorem has a fixed point $U^*$. For now we assume that such a fixed point $U^*$ exists. Since $U^*$ is a fixed-point of $\tilde J$, we also know that $U^*+\phi(x)$ is a fixed point of $P^s J$. But $P^sJ$ has a unique fixed point $P^s\Gamma$. Hence, it must hold $P^s\Gamma=U^*+\phi(x).$
Due to the identity $P^s\Gamma = U^*+\phi(x)$, we now consider $x$ as the solution of
\begin{align}\label{center equation}
    \txtd x_t = \left(A^c x_t + F^c_R\left(x,U + \phi(x)\right)[t]\right)~\txtd t + G_R^c(x,U + \phi(x))[t]~\txtd \bfW_t, ~~~~x_0=\xi,
\end{align}
for $U\in V$.
Moreover, we will impose a suitable assumption on the subspace $V$ such that we get for $U\in V$ a bound of the form $\Xnorm{U_0}\le K\Xnorm{\xi}^{q+1}$ and especially
\begin{align}\label{tilde J formal bound}
    \Xnorm{\tilde J(\bfW,U^*,\xi)[0]}\le K\Xnorm{\xi}^{q+1},
\end{align}
for a suitable constant $K\ge1$. 
This result will entail the local approximation of $h^c(\xi,\bfW)$ with polynomials of order $q$. More precisely, the aim is to eventually show that 
\begin{align*}
    \Xnorm{h^c(\xi,\bfW)-\phi(x_0)} &= \Xnorm{P^s\Gamma(\xi,\bfW)[0]-\phi(x_0)} \\
    &= \Xnorm{P^sJ(\bfW,P^s\Gamma(\xi,\bfW),\xi)[0]-\phi(x_0)} \\
    &= \Xnorm{P^sJ(\bfW,U^*+\phi(x),\xi)[0]-\phi(x_0)} \\
    &= \Xnorm{\tilde J(\bfW,U^*(\xi,W),\xi)[0]}\\
    &\le K\Xnorm{\xi}^{q+1},
\end{align*}
where we used the definition of $h^c$, the fact that $P^s\Gamma(\xi,\bfW)=U^*+\phi(x)$ and that both terms are fixed points of $P^s J$. This shows that $\phi$ is a good approximation of the graph of the center manifold given by $h^c$.\\

In conclusion, the main goal is to find an appropriate subspace $V$ and prove that $\tilde{J}$ is a contraction on it. One main difficulty in defining the subspace $V$ and setting up the fixed-point argument is given by the presence of the rough integrals in the Lyapunov-Perron maps $J$ and $\tilde{J}$ given by  \eqref{J cont def} and \eqref{tilde J formal def}. 
In order to use the results established in Subsection~\ref{rp} we first have to discretize $J$ and $\tilde{J}$ as in~\cite[Section 4.1]{KN23}.  \\

\subsection{Discretization}\label{sec:discretization}
In the last subsection we defined the Lyapunov Perron map $J$ \eqref{J cont def}. In the definition we have a stochastic integral from minus infinity to zero. With the methods introduced in Subsection~\ref{rp} we cannot work with this integral. Hence, the goal is to transform it in such a way that we have integrals on $[0,1]$. To this aim, we introduce the following space of sequences of controlled rough paths defined on the interval $[0,1]$ and discretize the integrals in~\eqref{J cont def} as in \cite[Section 4.1]{KN23}. 

\begin{definition}\label{BC def}
Let $\eta<0$. We say that a sequence of controlled rough paths $\mathbb{U}:=\left(\left(U^{j-1}, 
\left(U^{j-1}\right)'\right)\right)_{j\in\mathbb{Z}^{-}}$ with $U^{j-1}_{0}=U^{j-2}_{1}$ belongs to the space 
$BC^{\eta}\left(\cD^{2\gamma}_{{W}}\right) $ if
\begin{equation}\label{bcnorm}
\|\mathbb{U}\|_{BC^{\eta}\left(\cD^{2\gamma}_{{W}}\right)}:=
	\sup\limits_{j\in\mathbb{Z}^{-}} \txte^{-\eta (j-1)} \left\|U^{j-1},
	\left(U^{j-1}\right)'\right\|_{\cD^{2\gamma}_{{W}}}<\infty.
\end{equation} 
\end{definition}
For a sequence $\U = \left(U^{j-1},\left(U^{j-1}\right)'\right)_{j\in\Z^-}\in BC^\eta\left(\cD^{2\gamma}_W\right)$ we know that each element of the sequence belongs to $\cD^{2\gamma}_W$. This means that $U^{j-1}_t$ is a function defined for $t\in[0,1]$ and $j\in\Z^{-}$ denotes the position within the sequence. 
\begin{assumptions}\label{beta assumption}
    For our aims we fix $-\beta<\eta<0$, where $\beta$ was introduced in \eqref{beta}. 
\end{assumptions}

Our next step is to discretize all our variables, so we can work in the function space $BC^\eta\left(\cD^{2\gamma}_W\right)$. We start by discretizing the solution $x$ of the RDE~\eqref{center equation}.
For the discretized solution $x$ of~\eqref{center equation} we show a similar estimate as in Lemma \ref{x is 2 gamma and xi bound}. 
\begin{lemma}\label{xj bound}    Let $x_t$ be the solution of \eqref{center equation} with initial condition $\xi\in\cX^c$ on the time interval $(-\I,0]$. Let $\left(x^{j-1}\right)_{j\in\Z^-}$ be a sequence such that for all $t\in[0,1]$
$$x^{j-1}_t:=x_{j-1+t}.$$
Then for all $j\in\Z^-$ we know $\left(x^{j-1},\left(x^{j-1}\right)'\right)\in\cD^{2\gamma}_W$ and 
    $$\Dnorm{x^{j-1},\left(x^{j-1}\right)'}\le \tilde C_x \Xnorm{\xi},$$
    where $\tilde C_x$ is defined below in \eqref{def:tilde C_x}.
\end{lemma} 
\begin{proof}
    The Lyapunov Perron map $J$ defined in \eqref{J cont def} exists and due to Theorem \ref{contraction} $J$ has a fixed point $\Gamma(\xi,\bfW)[t]$. We discretize $\Gamma$ such that for $k\in\Z^-,t\in[0,1]$ we have $$\Gamma(\xi,\bfW)[k-1,t]:=\Gamma(\xi,\bfW)[k-1+t],$$
    where again the first index indicates the position within the sequence and the second one refers to the time variable.~We further use the same arguments as in Lemma \ref{x is 2 gamma and xi bound}. We have
    \begin{align*}
        &x_t^{j-1} = x_{j-1+t}\\
        &= S^c(j-1+t)\xi - \int_0^{j-1+t} S^c(j-1+t-r)F_R^c(x,P^s\Gamma(\xi,\bfW)[\cdot])[r]~\txtd r \\
        &~~~~- \int_0^{j-1+t} S^c(j-1+t-r)G_R^c(x,P^s\Gamma(\xi,\bfW)[\cdot])[r]~\txtd \bfW_r\\
        &= S^c(j-1+t)\xi - \sum_{k=0}^{j+1}S^c(j-1+t-k) \Bigg[\int_0^{1} S^c(1-r)F_R^c(x,P^s\Gamma(\xi,\bfW)[\cdot])[r+k-1]~\txtd r \\
        &~~~~- \int_0^{1} S^c(1-r)G_R^c(x,P^s\Gamma(\xi,\bfW)[\cdot])[r+k-1]~\txtd \Theta_{k-1}\bfW_r \Bigg]\\
        &~~~~- \int_t^{1} S^c(1-r)F_R^c(x,P^s\Gamma(\xi,\bfW)[\cdot])[r+j-1]~\txtd r \\
        &~~~~- \int_t^{1} S^c(1-r)G_R^c(x,P^s\Gamma(\xi,\bfW)[\cdot])[r+j-1]~\txtd \Theta_{j-1}\bfW_r \\
        &= S^c(j-1+t)\xi - \sum_{k=0}^{j+1}S^c(j-1+t-k) \Bigg[\int_0^{1} S^c(1-r)F_R^c\left(x^{k-1},P^s\Gamma(\xi,\bfW)[k-1,\cdot]\right)[r]~\txtd r \\
        &~~~~- \int_0^{1} S^c(1-r)G_R^c\left(x^{k-1},P^s\Gamma(\xi,\bfW)[k-1,\cdot]\right)[r]~\txtd \Theta_{k-1}\bfW_r\Bigg] \\
        &~~~~- \int_t^{1} S^c(1-r)F_R^c\left(x^{j-1},P^s\Gamma(\xi,\bfW)[j-1,\cdot]\right)[r]~\txtd r \\
        &~~~~- \int_t^{1} S^c(1-r)G_R^c\left(x^{j-1},P^s\Gamma(\xi,\bfW)[j-1,\cdot]\right)[r]~\txtd \Theta_{j-1}\bfW_r.
    \end{align*}
    Now we use the bound for the linear part \eqref{nu} and Lemma \ref{integral bound} to infer that
    \begin{align*}
            \Dnorm{x^{j-1},\left(x^{j-1}\right)'} &\le M_c e^{\nu (j-1)}\Xnorm{\xi} + \sum_{k=0}^{j} M_c e^{\nu (j-1-k)} Q(\Theta_{k-1}\bfW)R(\Theta_{k-1}\bfW) \Dnorm{x^{k-1},\left(x^{k-1}\right)'},
    \end{align*}
    where $Q(\bfW)$ is defined in \eqref{def:Q}. We define 
    \begin{align}\label{def:R_2}
        R_2(\bfW):=\min\left\{\frac{e^\nu}{2M_cQ(\bfW)},R_1(\bfW) \right\},
    \end{align}
    where $R_1(\bfW)$ is defined in \eqref{def:R_1}. With this we get
    \begin{align*}
        \Dnorm{x^{j-1},\left(x^{j-1}\right)'} &\le M_c e^{\nu (j-1)}\Xnorm{\xi} + \frac{1}{2}\sum_{k=0}^{j} e^{\nu (j-1-(k-1))} \Dnorm{x^{k-1},\left(x^{k-1}\right)'} \\
        &\le M_c e^{\nu (j-1)} \Xnorm{\xi} + \frac{1}{2}\sum_{k=0}^{j+1} \Dnorm{x^{k-1},\left(x^{k-1}\right)'} + \frac{1}{2}\Dnorm{x^{j-1},\left(x^{j-1}\right)'},
    \end{align*}
    where for the second inequality we used $ e^{\nu (j-1-(k-1))}\le 1$ for all $k\in\{0,\dots,j-1\}$. 
    So we get
    \begin{align}\label{ineq:recursive formula x^j-1}
        \Dnorm{x^{j-1},\left(x^{j-1}\right)'} &\le  C  M_c e^{\nu (j-1)} \Xnorm{\xi} +   \sum_{k=0}^{j+1} \Dnorm{x^{k-1},\left(x^{k-1}\right)'}.
    \end{align}
    With this recursive formula we can show 
    \begin{align*}
        \Dnorm{x^{j-1},\left(x^{j-1}\right)'} \le C M_c e^{\nu (j-1)} \left( 1 + \frac{e^\nu}{1-2e^\nu}\right)\Xnorm{\xi} 
        \le C M_c \left( 1 + \frac{e^\nu}{1-2e^\nu}\right)\Xnorm{\xi}.
    \end{align*}
    So we define
    \begin{align}\label{def:tilde C_x}
        \tilde C_x := CM_c \left( 1 + \frac{e^\nu}{1-2e^\nu}\right),
    \end{align}
    and the claim follows.
    \qed\\ 
\end{proof}
Next, we discretize the stationary solutions of the RDEs \eqref{alpha RDE}. To this aim, we let $i\in\{2,\dots,q\}$ and set $S^{\alpha_i}(t):=e^{A^{\alpha_i}t}$. Rewriting \eqref{alpha RDE} in integral form gives us for $s\le 0$ 
\begin{align*}
    \alpha_i(s) :=\alpha_i(\Theta_s\bfW)= \int_{-\I}^s S^{\alpha_i}(s-r) f_i(r)~\txtd r + \int_{-\I}^s S^{\alpha_i}(s-r) g_i(r)~\txtd \bfW_r.
\end{align*}
The coefficients $A^{\alpha_i},~f_i$ and $g_i$ are given by the invariance equations \eqref{invariance equation A}, \eqref{invariance equation f_i} and \eqref{invariance equation g_i}. For the discretization we make a change of variable $s=j-1+t$ where $t\in[0,1]$ and $j\in\Z^-$. With this change of variable we get 
\begin{align*}
    &\alpha_i(j-1+t) \\
    &= \int_{-\I}^{j-1+t} S^{\alpha_i}(j-1+t-r) f_i(r)~\txtd r + \int_{-\I}^{j-1+t} S^{\alpha_i}(j-1+t-r) g_i(r)~\txtd \bfW_r \\
    &= \sum_{k=-\I}^{j-1}S^{\alpha_i}(j-1+t-k)\left(\int_{0}^{1} S^{\alpha_i}(1-r) f_i(k-1+r)~\txtd r + \int_{0}^{1} S^{\alpha_i}(1-r) g_i(k-1+r)~\txtd \Theta_{k-1}\bfW_r \right)\\
    &~~~~ +\int_{0}^{t} S^{\alpha_i}(t-r) f_i(j-1+r)~\txtd r + \int_{0}^{t} S^{\alpha_i}(t-r) g_i(j-1+r)~\txtd \Theta_{j-1}\bfW_r.
\end{align*}
Now we define $\alpha_i^{j-1}(t):=\alpha_i(j-1+t)$, $f_i^{j-1}(t):=f_i(j-1+t)$ and $g_i^{j-1}(t):=g_i(j-1+t)$. With this notation we get
\begin{align}\label{discretized:alpha}
    \alpha_i^{j-1}(t) &= \sum_{k=-\I}^{j-1}S^{\alpha_i}(j-1+t-k)\left(\int_{0}^{1} S^{\alpha_i}(1-r) f_i^{k-1}(r)~\txtd r + \int_{0}^{1} S^{\alpha_i}(1-r) g_i^{k-1}(r)~\txtd \Theta_{k-1}\bfW_r \right) \nonumber\\
    &~~~~ +\int_{0}^{t} S^{\alpha_i}(t-r)f_i^{j-1}(r)~\txtd r + \int_{0}^{t} S^{\alpha_i}(t-r)g_i^{j-1}(r)~\txtd \Theta_{j-1}\bfW_r.
\end{align}
From now on we suppress again the dependence of $\alpha_i^{j-1},~f_i^{j-1},~g_i^{j-1}$ on the time $t$ when this is clear from the context. \\

After discretizing $\alpha_i$ it only remains to discretize $\phi$ defined in~\eqref{ansatz}. 
\begin{definition}
    We define for $t\in[0,1]$ and $j\in\Z^-$
    \begin{align}\label{phi discretized}
        \phi^{j-1}\left(x^{j-1}_t\right)=\sum_{i=2}^{q}\alpha_i^{j-1}\left(x^{j-1}_t\right)^i
    \end{align}
    and introduce the notation 
    \begin{align}\label{Phi def}
        \Phi(x):=\left(\phi^{j-1}\left(x^{j-1}\right), \phi^{j-1}\left(x^{j-1}\right)'\right)_{j\in\Z^-}.
    \end{align}
\end{definition}

Next, we show a bound for $M\phi^{j-1}$ and $\tilde M\phi^{j-1}$ defined as 
\begin{align}
    M\phi^{j-1}\left(x^{j-1}_t\right) &= \sum_{i=2}^qf_i^{j-1}\left(x_t^{j-1}\right)^i  - F^s_R\left(x^{j-1},\phi^{j-1}\left(x^{j-1}\right)\right)[t] \nonumber\\
    &~~~~+ \sum_{i=2}^{q} i\alpha_i^{j-1} \left(x_t^{j-1}\right)^{i-1} F^c_R\left(x^{j-1},\phi^{j-1}\left(x^{j-1}\right)\right)[t] \label{def:discretized M}\\
    \tilde{M}\phi^{j-1}\left(x^{j-1}_t\right) &= \sum_{i=2}^{q} g_i^{j-1} \left(x_t^{j-1}\right)^i - G_R^s(x^{j-1},\phi^{j-1}\left(x^{j-1}\right))[t]  \nonumber\\
    &~~~~+ \sum_{i=2}^q i\alpha_i^{j-1} \left(x_t^{j-1}\right)^{i-1} G^c_R\left(x^{j-1},\phi^{j-1}\left(x^{j-1}\right)\right)[t] \label{def:discretized tilde M}. 
\end{align}  
To this aim, we assume a bound for $\alpha_i$, which is reasonable since we perform a local analysis around the stationary point zero of~\eqref{SDE}. With this we then can show a bound for $M\phi(x)$ that depends on $x$ similarly to the bound assumed in the deterministic autonomous case specified in Theorem \ref{Carr Theorem}.

\begin{assumptions}\label{M assumption}
   \begin{enumerate}
       \item [1)]
    Let $i\in\{2,\dots,q\},~j\in\Z^-$. We assume that $\alpha_i^{j-1}$ defined by \eqref{discretized:alpha} fulfills
    $$\Dnorm{\alpha_i^{j-1},\left(\alpha_i^{j-1}\right)'}\le C~R(\Theta_{j-1}\bfW).$$
    Given this we infer that 
    \begin{align}\label{C_alpha bound}
        \tilde C_\alpha^{j-1}:=\max_{i\in\{2,\dots,q\}}\Dnorm{\alpha_i^{j-1}, \left(\alpha_i^{j-1}\right)'} \le C~R(\Theta_{j-1}\bfW).
    \end{align}
   \item [2)] Moreover, we assume that $F_R^{c/s}$ and $G_R^{c/s}$ are polynomials in all variables with highest order $k\le q$.
   \end{enumerate}
\end{assumptions}
We first discuss these assumptions.
\begin{remark}
\begin{enumerate}
    \item Let $z$ be the stationary Ornstein-Uhlenbeck process introduced in Example~\ref{ex:stationary}. Then it is well-known that $|z(B)|_{\cX}$ is tempered, so in particular tempered from below. 
    Here we impose a similar condition for the coefficients $\alpha_i^{j-1}$ but with respect to the stronger $\cD^{2\gamma}_W$-norm.  
    \item One can perform a cut-off argument as in Subsection~\ref{sec:cutoff} in order to obtain that $$\Dnorm{\alpha_i^{j-1},\left(\alpha_i^{j-1}\right)'}\le C~R(\Theta_{j-1}\bfW).$$~This suffices for our aims, since we only develop a local approximation theory. 
    \item Moreover, if the previous assumptions are not satisfied by $F_R^{c/s}$ and $G_R^{c/s}$, we can consider a Taylor approximation of order $q$ and proceed with this Taylor approximation. 
\end{enumerate}
\end{remark}
Based on Assumption~\ref{M assumption},  we can show the following bound for $\phi^{j-1}\left(x^{j-1}\right)$.
\begin{lemma}\label{phi BC estimate}
Let $x_t$ be the solution of \eqref{center equation} with initial condition $\xi$ on the time interval $(-\I,0]$. Let $\left(x^{j-1}\right)_{j\in\Z^-}$ be a sequence such that for all $t\in[0,1]$ and $j\in\Z^-$
$$x^{j-1}_t=x_{j-1+t}.$$ Moreover, for $i\in\{2,\dots,q\}$ let $\left(\alpha_i^{j-1},\left(\alpha_i^{j-1}\right)'\right)_{j\in\Z^-}$ be the solution of \eqref{discretized:alpha} and $\left(\phi^{j-1}\right)_{j\in\Z^-}$ as defined in \eqref{phi discretized}. 
Then, for all $j\in\Z^-$
$$\left(\phi^{j-1}\left(x^{j-1}\right),\phi^{j-1}\left(x^{j-1}\right)' \right)\in\cD^{2\gamma}_W,$$
and satisfies
\begin{align}\label{phi discretized bound}
    \Dnorm{\phi^{j-1}\left(x^{j-1}\right),\phi^{j-1}\left(x^{j-1}\right)'}\le  C\tilde C_x^q q^2\left(1+\norm{\Theta_{j-1}W}_\gamma\right)^2 \Xnorm{\xi} .
\end{align}
\end{lemma} 
\begin{proof}
    Due to Assumption \ref{M assumption} we get for all $j\in\Z^-$ that $\left(\alpha_i^{j-1},\left(\alpha_i^{j-1}\right)'\right)\in \cD_W^{2\gamma} $. The claim follows applying Lemma \ref{phi D2gamma estimate} and using the bound \eqref{C_alpha bound} for $\Dnorm{\alpha_i^{j-1},\left(\alpha_i^{j-1}\right)'}$.\qed\\
\end{proof}
\begin{remark}\label{rem:form of bound}
Our goal is to find a bound for $\Dnorm{M\phi^{j-1}\left(x^{j-1}\right),M\phi^{j-1}\left(x^{j-1}\right)'}$ and $\Dnorm{\tilde M\phi^{j-1}\left(x^{j-1}\right),\tilde M\phi^{j-1}\left(x^{j-1}\right)'}$ of the form $CR(\Theta_{j-1}\bfW)\Xnorm{\xi}^{q+1}$, where $C$ may depend on $\norm{\Theta_{j-1}W}_\gamma$ and $\norm{\Theta_{j-1}\W}_{2\gamma}$. Such a bound makes it possible to appropriately choose $R(\bfW)$ in \eqref{def:R_4}.
\end{remark}
To this end, we first analyze the terms $F^{c/s}_R\left(x^{j-1},\phi^{j-1}\left(x^{j-1}\right)\right)[t]$ and $$\sum_{i=2}^qi \alpha^{j-1}_i\left(x^{j-1}_t\right)^{i-1} F^c_R\left(x^{j-1},\phi^{j-1}\left(x^{j-1}\right)\right)[t].$$ For $G^{c/s}$ the statements will follow analogously.
We first show that the cut-off parameter $R$ can be omitted, so $F_R^{c/s} = F^{c/s}$ and $G_R^{c/s} = G^{c/s}$.
\begin{lemma}\label{lem:omit cut-off}
    Under the assumptions of Lemma \ref{phi BC estimate} there exists a tempered from below random variable $\rho(\bfW)$ such that for $\xi\in B_{\cX^c}(0,\rho(\bfW))$, we have for all $j\in\Z^-$ 
    \begin{align*}
        \Dnorm{\left(x^{j-1},\phi^{j-1}\left(x^{j-1}\right)\right), \left(x^{j-1},\phi^{j-1}\left(x^{j-1}\right)\right)'}\le R(\Theta_{j-1}\bfW).
    \end{align*}
    In particular we have
    \begin{align*}
        &F^{c/s}_R\left(x^{j-1},\phi^{j-1}\left(x^{j-1}\right)\right)[t] = F^{c/s}\left(x^{j-1}_t,\phi^{j-1}\left(x^{j-1}_t\right)\right) ,
    \end{align*}
    and
    \begin{align*}
        G_R^{c/s}\left(x^{j-1},\phi^{j-1}\left(x^{j-1}\right)\right)[t] = G^{c/s}\left(x^{j-1}_t,\phi^{j-1}\left(x^{j-1}_t\right)\right) .
    \end{align*}
\end{lemma}
\begin{proof}
We define $\hat\rho(\Theta_{j-1}\bfW):=\frac{R(\Theta_{j-1}\bfW)}{\tilde C_x  + C\tilde C_x^q q^2 (1+\norm{\Theta_{j-1}W}_\gamma)^2}$. As $R(\bfW)$ and $\frac{1}{\tilde C_x  + C\tilde C_x^q q^2 (1+\norm{W}_\gamma)^2}$ are tempered from below random variables, the product is also tempered from below due to Lemma \ref{lem:product of tempered variables}.
    Since $\hat\rho(\bfW)$ is tempered from below, there exists by Lemma~\ref{lemma:tbelow} a random variable $\rho(\bfW)$ such that 
    $$\hat\rho(\Theta_{j-1}\bfW)\ge \rho(\bfW)e^{\eta (j-1)}.$$
    We use Lemma \ref{xj bound} and \ref{phi BC estimate} to infer that
    \begin{align*}
        &\Dnorm{\left(x^{j-1},\phi^{j-1}\left(x^{j-1}\right)\right), \left(x^{j-1},\phi^{j-1}\left(x^{j-1}\right)\right)'} \\
        &\le \Dnorm{x^{j-1}, \left(x^{j-1}\right)'} + \Dnorm{\phi^{j-1}\left(x^{j-1}\right), \phi^{j-1}\left(x^{j-1}\right)'} \\
        &\le \tilde C_x \Xnorm{\xi} + C\tilde C_x^q q^2 (1+\norm{\Theta_{j-1}W}_\gamma)^2\Xnorm{\xi}. 
    \end{align*}
    Now we use $\xi\in B_{\cX^c}(0,\rho(\bfW))$ and our choice of $\hat\rho$ 
    \begin{align*}
        &\Dnorm{\left(x^{j-1},\phi^{j-1}\left(x^{j-1}\right)\right), \left(x^{j-1},\phi^{j-1}\left(x^{j-1}\right)\right)'} \\
        &\le \tilde C_x \Xnorm{\xi} + C\tilde C_x^q q^2 (1+\norm{\Theta_{j-1}W}_\gamma)^2\Xnorm{\xi} \\
        &\le (\tilde C_x  + C\tilde C_x^q q^2 (1+\norm{\Theta_{j-1}\bfW}_\gamma)^2)e^{\eta(j-1)}\rho(\bfW) \\
        &\le (\tilde C_x  + C\tilde C_x^q q^2 (1+\norm{\Theta_{j-1}\bfW}_\gamma)^2)\hat\rho(\Theta_{j}\bfW). \\
        &\le R(\Theta_{j-1}\bfW).
    \end{align*}
    This allows us to drop the cut-off parameter $R$ from $F^{c/s}_R$ and $G^{c/s}_R$. \\
    \qed\\
\end{proof}

Next we obtain a representation for $F^{c/s}(\cdot,\cdot)$ respectively $G^{c/s}(\cdot,\cdot)$ which will lead to an appropriate choice of coefficients $f_i$ and $g_i$ in~\eqref{def:discretized M} and~\eqref{def:discretized tilde M}.
\begin{lemma}\label{lem:F reformulated}
    Let the assumptions of Lemma \ref{phi BC estimate} hold and let $\xi\in B_{\cX^c}(0,\rho(\bfW))$. Then 
    \begin{align}\label{eq:F reformulated}
        &F^{c/s}\left(x^{j-1}_t,\phi^{j-1}\left(x^{j-1}_t\right)\right) \nonumber\\
        &= \sum_{i=2}^q \left(P^{F^{c/s}}_i\left(\alpha^{j-1}_2,\dots,\alpha^{j-1}_{i-1}\right) + C\right) \left(x^{j-1}_t\right)^i + \sum_{i=q+1}^{q^2} P^{F^{c/s}}_i\left(\alpha^{j-1}_2,\dots,\alpha^{j-1}_{q}\right) \left(x^{j-1}_t\right)^i,
    \end{align}
    where $P^{F^{c/s}}_i$ is a polynomial depending on the $\alpha_k^{j-1}$ for $k<i$, the coefficients of $F^{c/s}$ and on constants that can depend on $i$. For $P^{F^{c/s}}_i$ we have for $2\le i\le q$ the estimate    \begin{align*}
        \Dnorm{P^{F^{c/s}}_i\left(\alpha^{j-1}_2,\dots,\alpha^{j-1}_{i-1}\right),P^{F^{c/s}}_i\left(\alpha^{j-1}_2,\dots,\alpha^{j-1}_{i-1}\right)'} \le C[\norm{\Theta_{j-1}W}_\gamma,F^{c/s},i] R(\Theta_{j-1}\bfW),
    \end{align*}
    and for $i>q$
    \begin{align*}
        \Dnorm{P^{F^{c/s}}_i\left(\alpha^{j-1}_2,\dots,\alpha^{j-1}_{q}\right),P^{F^{c/s}}_i\left(\alpha^{j-1}_2,\dots,\alpha^{j-1}_{q}\right)'} \le C[\norm{\Theta_{j-1}W}_\gamma,F^{c/s},i] R(\Theta_{j-1}\bfW).
    \end{align*}
    In both cases the constant $C[\norm{\Theta_{j-1}W}_\gamma,F^{c/s},i]$ is increasing in $i$ and $C[\norm{\Theta_{j-1}W}_\gamma]=P^W[\norm{\Theta_{j-1}W}_\gamma]$, for a polynomial $P^W$.
    The same claim holds for $G^{c/s}$. 
\end{lemma}
\begin{proof}
   See Appendix \ref{proof:F reformulated}. 
\end{proof}\\

Moreover we get
\begin{lemma}\label{lem:sum F reformulated}
    Let the assumptions of Lemma \ref{phi BC estimate} hold and let $\xi\in B_{\cX^c}(0,\rho(\bfW))$. Then 
    \begin{align}\label{eq:sum F reformulated}
        &\sum_{i=2}^q i \alpha_i^{j-1} \left(x^{j-1}_t\right)^{i-1} F^{c}\left(x^{j-1}_t,\phi^{j-1}\left(x^{j-1}_t\right)\right) \nonumber \\
        &= \sum_{i=2}^q \left(\tilde P^{F^{c}}_i\left(\alpha^{j-1}_2,\dots,\alpha^{j-1}_{i-1}\right) + C\right) \left(x^{j-1}_t\right)^i + \sum_{i=q+1}^{q^3-q^2} \tilde P^{F^{c}}_i\left(\alpha^{j-1}_2,\dots,\alpha^{j-1}_{q}\right) \left(x^{j-1}_t\right)^i,
    \end{align}
    where $\tilde P^{F^{c}}_i$ is a polynomial depending on the $\alpha_k^{j-1}$ for $k<i$, the coefficients of $F^{c}$ and on constants that can depend on $i$. For $\tilde P^{F^{c}}_i$ we have for $i\le q$ the bound
    \begin{align*}
        \Dnorm{\tilde P^{F^{c}}_i\left(\alpha^{j-1}_2,\dots,\alpha^{j-1}_{i-1}\right),\tilde P^{F^{c}}_i\left(\alpha^{j-1}_2,\dots,\alpha^{j-1}_{i-1}\right)'}  \le C[\norm{\Theta_{j-1}W}_\gamma,F^{c},i] R(\Theta_{j-1}\bfW),
    \end{align*}
    and for $i>q$
    \begin{align*}
        \Dnorm{\tilde P^{F^{c}}_i\left(\alpha^{j-1}_2,\dots,\alpha^{j-1}_{q}\right),\tilde P^{F^{c}}_i\left(\alpha^{j-1}_2,\dots,\alpha^{j-1}_{q}\right)'} \le C[\norm{\Theta_{j-1}W}_\gamma,F^{c},i] R(\Theta_{j-1}\bfW).
    \end{align*}
    In both cases the constant $C[\norm{\Theta_{j-1}W}_\gamma,F^{c/s},i]$ is increasing in $i$ and $C[\norm{\Theta_{j-1}W}_\gamma]=P^W(\norm{\Theta_{j-1}W}_\gamma)$ for a polynomial $P^W$.
    The same claim holds for $G^c$.
\end{lemma}
\begin{proof}
See Appendix \ref{proof:sum F reformulated}.    
\end{proof}\\

The previous two Lemmas allow us to choose the coefficients of the RDEs in \eqref{discretized:alpha} such that we get $$\Dnorm{M\phi^{j-1}\left(x^{j-1}\right),M\phi^{j-1}\left(x^{j-1}\right)'} \le C \Dnorm{x^{j-1},\left(x^{j-1}\right)'}^{q+1},$$ and $$\Dnorm{\tilde M\phi^{j-1}\left(x^{j-1}\right),\tilde M\phi^{j-1}\left(x^{j-1}\right)'} \le C\Dnorm{x^{j-1},\left(x^{j-1}\right)'}^{q+1}.$$ For the linear parts of the RDEs~\eqref{alpha RDE} we already stated in Section \ref{sec:generalidea} that $$A^{\alpha_i} = A^s - iA^c.$$
Since we are working with the discretization, we now choose the coefficients as sequences $\left(f_i^{j-1}\right)_{j\in\Z^-},\left(g_i^{j-1}\right)_{j\in\Z^-}$ for all $j\in\Z^-,~t\in[0,1],~i\in\{2,\dots,q\}$ as
\begin{align}\label{def:f_i^{j-1}}f_i^{j-1} := P^{F^s}_i\left(\alpha_2^{j-1},\dots,\alpha_{i-1}^{j-1}\right) - \tilde P^{F^c}_i\left(\alpha_2^{j-1},\dots,\alpha_{i-1}^{j-1}\right) + C,\end{align}
\begin{align}\label{def:g_i^{j-1}}g_i^{j-1} := P^{G^s}_i\left(\alpha_2^{j-1},\dots,\alpha_{i-1}^{j-1}\right) - \tilde P^{G^c}_i\left(\alpha_2^{j-1},\dots,\alpha_{i-1}^{j-1}\right) + C.\end{align}

With these preparations we now can show a useful estimate for $M\phi^{j-1}(x^{j-1})$ and $\tilde M\phi^{j-1}(x^{j-1})$ in the form that is discussed in Remark \ref{rem:form of bound}.
\begin{lemma}\label{lem:M bound}
    Let the assumptions of Lemma \ref{phi BC estimate} hold and let $\xi\in B_{\cX^c}(0,\rho(\bfW))$. Then for $j\in\Z^-$ we get 
    \begin{align}
        \Dnorm{M\phi^{j-1}(x^{j-1}),M\phi^{j-1}(x^{j-1})'} &\le C_M[\norm{\Theta_{j-1}W}_\gamma]R(\Theta_{j-1}\bfW)\Xnorm{\xi}^{q+1} , \label{M bound}\\ 
        \Dnorm{\tilde M\phi^{j-1}(x^{j-1}),\tilde M\phi^{j-1}(x^{j-1})'} &\le C_{\tilde M}[\norm{\Theta_{j-1}W}_\gamma]R(\Theta_{j-1}\bfW)\Xnorm{\xi}^{q+1}. \label{tilde M bound}
    \end{align}
\end{lemma}
\begin{proof}
    We show the claim for $\Dnorm{M\phi^{j-1}(x^{j-1}),M\phi^{j-1}(x^{j-1})'}$. The claim for $\Dnorm{\tilde M\phi^{j-1}(x^{j-1}),\tilde M\phi^{j-1}(x^{j-1})'}$ follows analogously. We omit the Gubinelli derivative for notational simplicity when this is clear from the context. \\
    
    We use the choice of $f_i^{j-1}$ specified in \eqref{def:f_i^{j-1}} to infer that
    \begin{align*}
        &M\phi^{j-1}\left(x^{j-1}_t\right) \\
        &= \sum_{i=2}^qf_i^{j-1}\left(x_t^{j-1}\right)^i  - F^s\left(x^{j-1}_t,\phi^{j-1}\left(x^{j-1}_t\right)\right) 
        + \sum_{i=2}^{q} i\alpha_i^{j-1} \left(x_t^{j-1}\right)^{i-1} F^c\left(x^{j-1}_t,\phi^{j-1}\left(x^{j-1}_t\right)\right) \\
        &= \sum_{i=q+1}^{q^3-q^2} \tilde P_i^{F^c}\left(\alpha_2^{j-1},\dots,\alpha_{q}^{j-1}\right) \left(x_t^{j-1}\right)^i - \sum_{i=q+1}^{q^2}  P_i^{F^s}\left(\alpha_2^{j-1},\dots,\alpha_{q}^{j-1}\right) \left(x_t^{j-1}\right)^i.
    \end{align*}
This combined with Lemma \ref{lem:addition and multiplication of RP} entail
    \begin{align*}
        &\Dnorm{M\phi^{j-1}(x^{j-1}),M\phi^{j-1}(x^{j-1})'} \\
        &\le \Dnorm{\sum_{i=q+1}^{q^2}  \left(\tilde P_i^{F^c}\left(\alpha_2^{j-1},\dots,\alpha_{q}^{j-1}\right) - P_i^{F^s}\left(\alpha_2^{j-1},\dots,\alpha_{q}^{j-1}\right)\right) \left(x_t^{j-1}\right)^i} \\
        &~~~~+ \Dnorm{\sum_{i=q^2+1}^{q^3-q^2}  \tilde P_i^{F^c}\left(\alpha_2^{j-1},\dots,\alpha_{q}^{j-1}\right)  \left(x_t^{j-1}\right)^i} \\
        &\le Cq^2\left(1+\norm{\Theta_{j-1}W}_\gamma\right)^2\tilde C_x^{q^2}\sum_{i=q+1}^{q^2}\left(\Dnorm{\tilde P_i^{F^c}\left(\alpha_2^{j-1},\dots,\alpha_{q}^{j-1}\right)} + \Dnorm{P_i^{F^s}\left(\alpha_2^{j-1},\dots,\alpha_{q}^{j-1}\right)}\right) \Xnorm{\xi}^i \\
        &~~~~+ Cq^3\left(1+\norm{\Theta_{j-1}W}_\gamma\right)^2\tilde C_x^{q^3}\sum_{i=q^2+1}^{q^3-q^2}\Dnorm{\tilde P_i^{F^c}\left(\alpha_2^{j-1},\dots,\alpha_{q}^{j-1}\right)}\Xnorm{\xi}^i.
    \end{align*}
    Due to Lemma \ref{lem:F reformulated} and Lemma \ref{lem:sum F reformulated} we infer that
    \begin{align*}
        &\Dnorm{M\phi^{j-1}(x^{j-1}),M\phi^{j-1}(x^{j-1})'} \\
        &\le Cq^2\left(1+\norm{\Theta_{j-1}W}_\gamma\right)^2\tilde C_x^{q^2}\sum_{i=q+1}^{q^2} C[\norm{\Theta_{j-1}W}_\gamma,F^c,F^s,i] R(\Theta_{j-1}\bfW) \Xnorm{\xi}^i \\
        &~~~~+ Cq^3\left(1+\norm{\Theta_{j-1}W}_\gamma\right)^2\tilde C_x^{q^3}\sum_{i=q^2+1}^{q^3-q^2} C[\norm{\Theta_{j-1}W}_\gamma,F^c,i] R(\Theta_{j-1}\bfW)\Xnorm{\xi}^i \\
        &\le C[\norm{\Theta_{j-1}W}_\gamma,F^c,F^s,q] \tilde C_x^{q^3} R(\Theta_{j-1}\bfW) \sum_{i=q+1}^{q^3-q^2} \Xnorm{\xi}^i .
    \end{align*}
    As $\Xnorm{\xi}\le1$ we know that $\Xnorm{\xi}^i\le\Xnorm{\xi}^{q+1}$ for all $i\ge q+1$, so
    \begin{align*}
        \Dnorm{M\phi^{j-1}(x^{j-1}),M\phi^{j-1}(x^{j-1})'} 
        &\le C[\norm{\Theta_{j-1}W}_\gamma,F^c,F^s,q] \tilde C_x^{q^3} R(\Theta_{j-1}\bfW) \Xnorm{\xi}^{q+1} .
    \end{align*}
    Setting $C_M[\norm{\Theta_{j-1}W}_\gamma]:=C[\norm{\Theta_{j-1}W}_\gamma,F^c,F^s,q] \tilde C_x^{q^3}$ and $C_{\tilde M}[\norm{\Theta_{j-1}W}_\gamma]:=C[\norm{\Theta_{j-1}W}_\gamma,G^c,G^s,q] \tilde C_x^{q^3}$, proves the claim.
    \qed\\
\end{proof}

Now we can define the mapping $J:BC^\eta\left(\cD^{2\gamma}_W\right)\times \cX^c\to BC^\eta\left(\cD^{2\gamma}_W\right)$ as the discretization of the Lyapunov-Perron map~\eqref{J cont def} replacing $t$ by $t+j-1$ for $j\in\Z^-$ and $t\in[0,1]$ obtaining 
\begin{align}\label{J discr def}
&J(\bfW,\U,\xi)[j-1,t] \nonumber\\
&= S^{c}(t+j-1) \xi  - \sum\limits_{k=0}^{j+1} S^{c} (t+j-1-k)  
\Bigg(\int\limits_{0}^{1} S^{c}(1-r) F^c_R\left(x^{k-1}, U^{k-1}\right)[r] ~\txtd r \nonumber\\
&+ \int\limits_{0}^{1} S^{c}(1-r) G^c_R\left(x^{k-1}, U^{k-1}\right)[r] ~\txtd \Theta_{k-1}\bfW_r \Bigg) \nonumber\\
&-\int\limits_{t}^{1} S^{c} (1-r) F^c_R \left(x^{j-1}, U^{j-1}\right)[r] ~\txtd r - \int\limits_{t}^{1} S^{c}(1-r) G^c_R\left(x^{j-1}, U^{j-1}\right)[r] ~\txtd \Theta_{j-1}\bfW_r \nonumber\\
& +\sum\limits_{k=-\infty}^{j-1}  S^{s} (t+j-1-k)
\Bigg( \int\limits_{0}^{1} S^{s}(1-r) F^s_R\left(x^{k-1},U^{k-1}\right)[r] ~\txtd r \nonumber\\
&+  \int\limits_{0}^{1} S^{s}(1-r) G_R^s\left(x^{k-1},U^{k-1}\right)[r]
 ~\txtd \Theta_{k-1}\bfW_{r} \Bigg) \nonumber\\
& +\int\limits_{0}^{t}S^{s}(t-r) F^s_R\left(x^{j-1},U^{j-1}\right)[r] ~\txtd r 
+ \int\limits_{0}^{t} S^{s}(t-r)G_R^s\left(x^{j-1},U^{j-1}\right)[r] 
~\txtd \Theta_{j-1}\bfW_{r}.
\end{align}
Therefore, the Lyapunov-Perron map can be viewed as a map on the spaces of sequences $BC^\eta\left(\cD^{2\gamma}_W\right)$. \\

Given this we define the map $\tilde{J}: BC^\eta\left(\cD^{2\gamma}_W\right)\times\cX^c\to BC^\eta\left(\cD^{2\gamma}_W\right)$
\begin{align}\label{tilde J def}
    \tilde{J}(\bfW,\U,\xi)[j-1,t]:=P^sJ(\bfW,\U+\Phi(x),\xi)[j-1,t] - \phi^{j-1}_t\left(x_t^{j-1}\right).
\end{align}
Moreover, we recall that the center manifold $\cM^c_{loc}(\bfW)$ is given by the graph of the Lipschitz function $h^c(\xi,\bfW)=P^s\Gamma(\xi,\bfW)[-1,1]$ due to Lemma \ref{localcman}, where $\Gamma$ is the fixed point of $J$. Moreover the center manifold is invariant meaning that if 
$$\varphi(0,W,\Gamma(\xi,\bfW)[-1,1])=\Gamma(\xi,\bfW)[-1,1]\in \cM^c_{loc}(\bfW) = \{\xi+h^c(\xi,\bfW):\xi\in B_{\cX^c}(0,\rho(\bfW))\},$$
then
$$\varphi(j,W,\Gamma(\xi, \bfW)[-1,1])\in \cM^c_{loc}(\Theta_j \bfW) = \{\xi+h^c(\xi,\Theta_j\bfW):\xi\in B_{\cX^c}(0,\rho(\Theta_j\bfW))\}.$$ 
We want to show that $\tilde{J}$ is a contraction on
$$V:=\left\{ \U\in BC^\eta\left(\cD^{2\gamma}_W\right): \Dnorm{U^{j-1},\left(U^{j-1}\right)'}\le K\Xnorm{\xi}^{q+1}~~\textnormal{ for all } j\in\Z^-\right\}.$$ 
\begin{remark}
    \begin{enumerate}
        \item [1)] In order to sketch the main idea of the approximation, in \eqref{tilde J formal bound} we formally imposed a bound for the norm of $U$ in the phase space $\cX$. This is a natural choice in the deterministic case speficied in \ref{Carr Theorem}.
        \item [2)]  In our setting, we assume a bound for the $\cD^{2\gamma}_W$-norm in the definition of the space $V$ for two reasons. Firstly, we can bound the spatial norm of the path component by its $\cD^{2\gamma}_W([0,1])$-norm. Namely, for an arbitrary controlled rough path $(Y,Y')\in \cD^{2\gamma}_W$ we get using~\eqref{def:crp} for all  $t\in[0,1]$ that
$$\Xnorm{Y_t}\le \norm{Y}_\I\le \Dnorm{Y,Y'}.$$
Secondly, the bounds for the rough integral depend on the $\cD^{2\gamma}_W([0,1])$-norm.~Therefore it is meaningful to impose a uniform bound on this norm in the definition of $V$.~Obviously, this condition also implies a bound for the $BC^\eta$-norm as
\begin{align*}
    \Bnorm{\U} &= \sup_{j\in\Z^-}e^{-\eta(j-1)}\Dnorm{U^{j-1},\left(U^{j-1}\right)'}\\
&\le \sup_{j\in\Z^-}e^{-\eta(j-1)}K\Xnorm{\xi}^{q+1}\\
&\le K\Xnorm{\xi}^{q+1}.
\end{align*}
This is useful since the fixed point argument in Section~\ref{sec:fp} is set up in the function space $BC^\eta\left(\cD^{2\gamma}_W\right)$.
    \end{enumerate}
\end{remark}

In order to show that $\tilde J$ is a contraction on $V$, we first show an alternative representation of $\tilde J$. To this aim, we first introduce some notations for clarity.
\begin{notation}
First, for the discretized solution of the system~\eqref{alpha RDE} we write \[\boldsymbol{\alpha}^{k-1}:=\left(\alpha_2^{k-1},\dots,\alpha_{q}^{k-1}\right) \text{ for } k\in\Z^-\]
keeping track of the coefficients $\alpha_2,\ldots,\alpha_{q}$ of the ansatz~\eqref{discretized:alpha}. 
    We further introduce some notations for the following terms because they turn out to be part of the discretization of $\tilde J$. We define for $k\in\Z^-$ and $t\in[0,1]$ 
    \begin{align}\label{def I}
    &I_t\left(x^{k-1},U^{k-1},\boldsymbol{\alpha}^{k-1}\right) \nonumber\\
    &:= F_R^s\left(x^{k-1},U^{k-1} + \phi^{k-1}\left(x^{k-1}\right)\right)[t] + A^s\phi^{k-1}\left(x_t^{k-1}\right) \nonumber\\
    &~~~~- \sum_{i=2}^{q} i\alpha^{k-1}_i \left(x_t^{k-1}\right)^{i-1} \left[A^cx_t^{k-1} + F_R^c\left(x^{k-1},U^{k-1} + \phi^{k-1}\left(x^{k-1}\right)\right)[t]\right] - \tilde{\phi}^{k-1}\left(x_t^{k-1}\right),
    \end{align}
    and
    \begin{align}\label{def II}
        II_t\left(x^{k-1},U^{k-1},\boldsymbol{\alpha}^{k-1}\right) 
        &:= G_R^s\left(x^{k-1},U^{k-1} + \phi^{k-1}\left(x^{k-1}\right)\right)[t] - \sum_{i=2}^{q} g^{k-1}_i \left(x_t^{k-1}\right)^{i} \nonumber\\
    &~~~~- \sum_{i=2}^{q} i\alpha^{k-1}_i \left(x_t^{k-1}\right)^{i-1} G_R^c\left(x^{k-1},U^{k-1} + \phi^{k-1}\left(x^{k-1}\right)\right)[t].
    \end{align}
   We finally set for $t\in[0,1]$
    \begin{align}\label{def T}
        T_t\left(\Theta_{k-1}\bfW,x^{k-1},U^{k-1},\boldsymbol{\alpha}^{k-1}\right) 
        &:= \int_0^tS^s(t-r)I_r\left(x^{k-1},U^{k-1},\balpha^{k-1}\right) ~\txtd r \nonumber\\
        &~~~~+ \int_0^tS^s(t-r) II_r\left(x^{k-1},U^{k-1},\balpha^{k-1}\right)~\txtd\Theta_{k-1}\bfW_r
    \end{align}
    and remark that $II$ is the Gubinelli derivative of $T$.
\end{notation}
Lastly, we state the discretization for $\tilde{J}(\bfW,\U,\xi)$.
\begin{lemma}\label{tilde J}
    Let $\U\in V, \tilde{J}, \Phi$ as defined in \eqref{tilde J def} and \eqref{Phi def}. Let further $x$  be the solution of \eqref{center equation} on $(-\I,0]$ with initial condition $\xi\in B_{\cX^c}(0,\rho(\bfW))$ and $x_t^{j-1}=x_{j-1+t}$ be the discretization for $j\in\Z^-$ and $t\in[0,1]$. Then we find the following expression for $\tilde{J}$ for $j\in\Z^-$ and $t\in[0,1]$
    \begin{align*}
    \tilde{J}(\bfW,\U,\xi)[j-1,t] 
    &= \sum_{k=-\I}^{j-1} S^s(t+j-1-k) ~T_1\left(\Theta_{k-1}\bfW,x^{k-1},U^{k-1},\balpha^{k-1}\right)\\
    &~~~~+T_t\left(\Theta_{j-1}\bfW,x^{j-1},U^{j-1},\balpha^{j-1}\right). 
\end{align*}
\end{lemma}
The proof of this Lemma is given in Appendix \ref{a}.

\subsection{Fixed point argument}\label{sec:fp}
After discretizing $\tilde{J}$ we now focus on proving that $\tilde{J}$ maps $V$ into $V$. To this end, we establish for $\U\in V$ and for all $j\in\Z^-$ that
$$\Dnorm{\tilde{J}(\bfW,\U,\xi)[j-1,\cdot],\tilde{J}(\bfW,\U,\xi)[j-1,\cdot]'}\le K\Xnorm{\xi}^{q+1}.$$ To this aim, our goal is to find a bound of the form $CRK\Xnorm{\xi}^{q+1}$ where the constant $C$ does not depend on $R$ and $K$ as we mentioned in Remark \ref{rem:form of bound}. After finding such a bound we then can choose $R$ small enough such that $CR<1$.~This is the topic of the next result. 

\begin{lemma}\label{tilde J maps into V}
    Let $\xi\in B_{\cX^c}(0,\rho(\bfW))$, let $x$ be the solution of \eqref{center equation} and $\U\in V$. Then for all $j\in\Z^-$ we get
    $$\Dnorm{\tilde J(\bfW,\U,\xi)[j-1,\cdot], \tilde J(\bfW,\U,\xi)[j-1,\cdot]'}\le K\Xnorm{\xi}^{q+1}.$$
\end{lemma}
\begin{proof}
    From Lemma \ref{tilde J} we know 
    \begin{align*}
        \tilde J(\bfW,\U,\xi)[j-1,t] 
        &= \sum_{k=-\I}^{j-1}\left[S^s(t+j-1-k) ~T_1\left(\Theta_{k-1}\bfW,x^{k-1},U^{k-1},\boldsymbol{\alpha}^{k-1}\right)\right]\\
        &~~~~+T_t\left(\Theta_{j-1}\bfW,x^{j-1},U^{j-1},\boldsymbol{\alpha}^{j-1}\right).
    \end{align*}
    Using the triangle inequality and \eqref{beta} we get
    \begin{align*}
        &\Dnorm{\tilde J(\bfW,\U,\xi)[j-1,\cdot], \tilde J(\bfW,\U,\xi)[j-1,\cdot]'}\\
        &\le \sum_{k=-\I}^{j-1}M_se^{-\beta(j-1-k)}\Dnorm{ T\left(\Theta_{k-1}\bfW,x^{k-1},U^{k-1},\balpha^{k-1}\right), T\left(\Theta_{k-1}\bfW,x^{k-1},U^{k-1},\balpha^{k-1}\right)'}\\
        &~~~~+ \Dnorm{T\left(\Theta_{j-1}\bfW,x^{j-1},U^{j-1},\boldsymbol{\alpha}^{j-1}\right),T\left(\Theta_{j-1}\bfW,x^{j-1},U^{j-1},\boldsymbol{\alpha}^{j-1}\right)'}.
    \end{align*}
We first fix $k\in\Z^-$ and use Lemma \ref{integral bound} to estimate $T$.~This is possible since $I\left(x^{k-1}, U^{k-1},\balpha^{k-1}\right)\in C_b^1$ and $II\left(x^{k-1}, U^{k-1},\balpha^{k-1}\right)\in C^3_b$ with respect to $U^{k-1}$ and $\balpha^{k-1}$, which immediately follows from the definitions of $I$ and $II$ specified in \eqref{def I} and \eqref{def II}.~Therefore, we use Lemma \ref{integral bound} with $Y=(U^{k-1},\balpha^{k-1})$ and $\tilde Y=0$ to get
    \begin{align*}
        &\Dnorm{T\left(\Theta_{k-1}\bfW,x^{k-1},U^{k-1},\balpha^{k-1}\right),II\left(x^{k-1},U^{k-1},\balpha^{k-1}\right)}\\
        &\le C[\Xnorm{A}]\Big(\norm{I\left(x^{k-1},U^{k-1},\balpha^{k-1}\right)}_\I + Z(\Theta_{k-1}\bfW)\Dnorm{II\left(x^{k-1},U^{k-1},\balpha^{k-1}\right)}\Big),
    \end{align*}
    where $Z(\Theta_{k-1}\bfW)$ is defined as 
    \begin{align}\label{Z def}
        Z(\Theta_{k-1}\bfW): = \left(1+\norm{\Theta_{k-1}W}_\gamma\right)\left(\norm{\Theta_{k-1}W}_\gamma + \norm{\Theta_{k-1}\W}_{2\gamma} \right).
    \end{align}    
    Next we derive an estimate for $\norm{I\left(x^{k-1},U^{k-1},\balpha^{k-1}\right)}_\I$ using the definition of $M\phi$ specified in \eqref{M def} to rewrite $I$. This leads to 
    \begin{align*}
        &\norm{I\left(x^{k-1},U^{k-1},\balpha^{k-1}\right)}_\I \\
        &\le \norm{F_R^s\left(x^{k-1},U^{k-1} + \phi^{k-1}\left(x^{k-1}\right)\right)[\cdot] - F_R^s\left(x^{k-1},\phi^{k-1}\left(x^{k-1}\right)\right)[\cdot]}_\I + \norm{M\phi^{k-1}\left(x^{k-1}\right)}_\I \\
        &+\norm{\sum_{i=2}^{q} i\alpha^{k-1}_i \left(x^{k-1}\right)^{i-1} \left[F_R^c\left(x^{k-1}, \phi^{k-1}\left(x^{k-1}\right)\right)[\cdot] - F_R^c\left(x^{k-1},U^{k-1} + \phi^{k-1}\left(x^{k-1}\right)\right)[\cdot]\right]}_\I.
    \end{align*}
    We consider each summand separately. For the first summand we use Lemma \ref{lip:fr} together with the fact that $\mathbb{U}\in V$ to get
    \begin{align*}
        &\norm{F_R^s\left(x^{k-1},U^{k-1} + \phi^{k-1}\left(x^{k-1}\right)\right)[\cdot] - F_R^s\left(x^{k-1},\phi^{k-1}\left(x^{k-1}\right)\right)[\cdot]}_\I\\
        &\le C_{F^s}R(\Theta_{k-1}\bfW)\Dnorm{U^{k-1}, \left(U^{k-1}\right)'}\\
        &\le C_{F^s}R(\Theta_{k-1}\bfW)K\Xnorm{\xi}^{q+1}.
    \end{align*}
    For $M\phi^{k-1}$ we use \eqref{M bound} to obtain
    \begin{align*}
        \norm{M\phi^{k-1}\left(x^{k-1}\right)}_\I 
        &\le \Dnorm{M\phi^{k-1}\left(x^{k-1}\right),M\phi^{k-1}\left(x^{k-1}\right)'}\\
        &\le C_M[\norm{\Theta_{k-1}W}_\gamma] R(\Theta_{k-1}\bfW)\Xnorm{\xi}^{q+1}.
    \end{align*}
    In order to deal with the last term, we first treat for $k\in \Z^-$  $$\norm{\sum_{i=2}^{q}i\alpha_i^{k-1}\left(x^{k-1}\right)^{i-1}}_\I.$$
    Using similar arguments to Lemma \ref{phi D2gamma estimate} and using Lemma \ref{xj bound}, we get 
    \begin{align*}
        \norm{\sum_{i=2}^{q}i\alpha_i^{k-1}\left(x^{k-1}\right)^{i-1}}_\I 
        &\le q \tilde C_x^{q-1}\sum_{i=0}^{q-1}\Xnorm{\xi}^{i} \\
        &\le q^2 \tilde C_x^{q-1},
    \end{align*}
    where we used inequality \eqref{C_alpha bound} and $\Xnorm{\xi}\le\rho(\bfW)\le1$. 
    This estimate together with Lemma \ref{lip:fr} and the fact that $\U\in V$, imply that
    \begin{align}\label{sum i alpha bound}
        &\norm{\sum_{i=2}^{q} i\alpha^{k-1}_i \left(x^{k-1}\right)^{i-1} \left[F_R^c\left(x^{k-1}, \phi^{k-1}\left(x^{k-1}\right)\right)[\cdot] - F_R^c\left(x^{k-1},U^{k-1} + \phi^{k-1}\left(x^{k-1}\right)\right)[\cdot]\right]}_\I\nonumber\\
        &\le q^2 \tilde C_x^{q-1} \norm{F_R^c\left(x^{k-1}, \phi^{k-1}\left(x^{k-1}\right)\right)[\cdot] - F_R^c\left(x^{k-1},U^{k-1} + \phi^{k-1}\left(x^{k-1}\right)\right)[\cdot]}_\I \nonumber\\
        &\le q^2 C_{F^c} \tilde C_x^{q-1}  R(\Theta_{k-1}\bfW) \Dnorm{U^{k-1},\left(U^{k-1}\right)'}\nonumber\\
        &\le q^2 C_{F^c} \tilde C_x^{q-1} R(\Theta_{k-1}\bfW)K\Xnorm{\xi}^{q+1}.
    \end{align}
    Adding up all these terms, we infer that
    \begin{align*}
        &\norm{I\left(x^{k-1},U^{k-1},\balpha^{k-1}\right)}_\I\\
        &\le \left(C_{F^s} + C_M[\norm{\Theta_{k-1}W}_\gamma] + q^2C_{F^c}\tilde C_x^{q-1} \right)R(\Theta_{k-1}\bfW)K\Xnorm{\xi}^{q+1}\\
        &=C_I[\norm{\Theta_{k-1}W}_\gamma]R(\Theta_{k-1}\bfW)K\Xnorm{\xi}^{q+1},
    \end{align*}
    where we set
    $C_I: =\left(C_{F^s} + C_M[\norm{\Theta_{k-1}W}_\gamma] + q^2C_{F^c} \tilde C_x^{q-1} \right)$.  
    Next we consider $\Dnorm{II\left(x^{k-1},U^{k-1},\balpha^{k-1}\right)}$, where we will omit the Gubinelli derivative for clarity. For $\Dnorm{II\left(x^{k-1},U^{k-1},\balpha^{k-1}\right)}$ we can show with the same steps as for $\norm{I\left(x^{k-1},U^{k-1},\balpha^{k-1}\right)}_\I$
    \begin{align*}
        &\Dnorm{II\left(x^{k-1},U^{k-1},\balpha^{k-1}\right)} \\
        &\le \Dnorm{G_R^s\left(x^{k-1},U^{k-1} + \phi^{k-1}\left(x^{k-1}\right)\right) - G_R^s\left(x^{k-1},\phi^{k-1}\left(x^{k-1}\right)\right)} + \Dnorm{\tilde M\phi^{k-1}\left(x^{k-1}\right)} \\
        &+\Dnorm{\sum_{i=2}^{q} i\alpha^{k-1}_i \left(x^{k-1}\right)^{i-1} \left[G_R^c\left(x^{k-1}, \phi^{k-1}\left(x^{k-1}\right)\right) - G_R^c\left(x^{k-1},U^{k-1} + \phi^{k-1}\left(x^{k-1}\right)\right)\right]} \\
        &\le \left(C_{G^s}[\norm{\Theta_{k-1}W}_\gamma]
        + C_{\tilde M}[\norm{\Theta_{k-1}W}_\gamma] + q^2 C_{G^c}[\norm{\Theta_{k-1}W}_\gamma]\tilde C_x^{q-1}\left(1+\norm{\Theta_{k-1}W}\right)^6\right)\\
        &~~~~\times R(\Theta_{k-1}\bfW)K\Xnorm{\xi}^{q+1}\\
        &=C_{II}[\norm{\Theta_{k-1}W}_\gamma]R(\Theta_{k-1}\bfW)K\Xnorm{\xi}^{q+1},
    \end{align*}
    where \begin{align}\label{C_II def}
        C_{II}[\norm{\Theta_{k-1}W}_\gamma]:=C_{G^s}[\norm{\Theta_{k-1}W}_\gamma]
        + C_{\tilde M}[\norm{\Theta_{k-1}W}_\gamma] + q^2 C_{G^c}[\norm{\Theta_{k-1}W}_\gamma]\tilde C_x^{q-1}\left(1+\norm{\Theta_{k-1}W}\right)^6
    \end{align}
    is tempered from above. 
    Next we put the estimates for $I$ and $II$ together 
    \begin{align*}
        &\Dnorm{T\left(\Theta_{k-1}\bfW,x^{k-1},U^{k-1},\balpha^{k-1}\right),II\left(x^{k-1},U^{k-1},\balpha^{k-1}\right)}\\
        &\le C[\Xnorm{A}]\Big(C_I [\norm{\Theta_{k-1}W}_\gamma]+ Z(\Theta_{k-1}\bfW)C_{II}[\norm{\Theta_{k-1}W}_\gamma]
        \Big)R(\Theta_{k-1}\bfW)K\Xnorm{\xi}^{q+1}.
    \end{align*}
   Let $C_T>0$, where its particular choice of $C_T$ will be specified below.~Then we denote the unique solution $\tilde R_3(\bfW)$  of the equation
    \begin{align}\label{def:tilde R_3}
        \tilde R_3({\bf W}) C[\Xnorm{A}]\Big(C_I[\norm{W}_\gamma]+ Z(\bfW)C_{II}[\norm{W}_\gamma]
        \Big) = C_T.
    \end{align}
    In particular, this implies for all $k\in \Z^{-}$ that 
 \begin{align*}
        \tilde R_3(\Theta_{k-1}\bfW) C[\Xnorm{A}]\Big(C_I[\norm{\Theta_{k-1}W}_\gamma]+ Z(\Theta_{k-1}\bfW)C_{II}[\norm{\Theta_{k-1}W}_\gamma]
        \Big) = C_T.
    \end{align*}
    
    The choice of $C_T$, specified below, will therefore also determine $\tilde R_3(\bfW)$. We note that $\tilde R_3(\bfW)$ is tempered from below, as it is the inverse of a polynomial in $\norm{W}_\gamma$ and $\norm{\W}_{2\gamma}$. 
    The choice of $\tilde R_3(\bfW)$ entails
    \begin{align*}
        &\Dnorm{T\left(\Theta_{k-1}\bfW,x^{k-1},U^{k-1},\balpha^{k-1}\right),II\left(x^{k-1},U^{k-1},\balpha^{k-1}\right)}\\
        &\le C_TK\Xnorm{\xi}^{q+1}.
    \end{align*}
    Using the bound for $T$ obtained above, we get based on Assumption~\ref{ass:linearpart}
    \begin{align*}
        &\Dnorm{\tilde J(\bfW,\U,\xi)[j-1,\cdot], \tilde J(\bfW,\U,\xi)[j-1,\cdot]'}\\
        &\le \sum_{k=-\I}^{j-1}M_se^{-\beta(j-1-k)}\Dnorm{ T\left(\Theta_{k-1}\bfW,x^{k-1},U^{k-1},\boldsymbol{\alpha}^{k-1}\right), II\left(x^{k-1},U^{k-1},\balpha^{k-1}\right)}\\
        &~~~~+ \Dnorm{T\left(\Theta_{j-1}\bfW,x^{j-1},U^{j-1},\boldsymbol{\alpha}^{j-1}\right),II\left(x^{j-1},U^{j-1},\balpha^{j-1}\right)}\\
        &\le \sum_{k=-\I}^{j-1}\left(M_se^{-\beta(j-1-k)} C_TK\Xnorm{\xi}^{q+1}\right) + C_TK\Xnorm{\xi}^{q+1} \\
        &=C_T\left( \frac{M_s}{1-e^{-\beta}} + 1\right)K\Xnorm{\xi}^{q+1}.
    \end{align*}
    Our goal is to ensure that
    $$C_T\left(\frac{M_s}{1-e^{-\beta}} + 1\right)\le 1,$$
    which implies that $\Dnorm{\tilde J(\bfW,\U,\xi)[j-1,\cdot], \tilde J(\bfW,\U,\xi)[j-1,\cdot]'}\le K\Xnorm{\xi}^{q+1}$. 
  Therefore, we choose $C_T$ as
    \begin{align}\label{def C_T}
        C_T:=\frac{1}{1+\frac{M_s}{1-e^{-\beta}}}.
    \end{align}
    For this $C_T$ we let $\tilde R_3(\bfW)$ be the unique solution of \eqref{def:tilde R_3} and further set\begin{align}\label{def:R_3}
        R_3(\bfW) := \min\left\{\tilde R_3(\bfW),R_2(\bfW)\right\},
    \end{align}
    where $R_2(\bfW)$ is defined in \eqref{def:R_2}. 
    With this choice of $C_T$ and $R_3(\bfW)$ we get for all $j\in\Z^-$ that
    $$\Dnorm{\tilde J(\bfW,\U,\xi)[j-1,\cdot], \tilde J(\bfW,\U,\xi)[j-1,\cdot]'}\le K\Xnorm{\xi}^{q+1},$$
    as required.
    \qed\\
\end{proof}

After showing that $\tilde J$ maps $V$ into $V$ it remains to show that $\tilde J$ is a contraction on $V$. 
\begin{lemma}\label{tilde J contraction}
    Let $\xi\in B_{\cX^c}(0,\rho(\bfW))$, let $x$ be the solution of \eqref{center equation} and $\U,\tilde\U\in V$. Then there exists $\tilde C<1$ such that
    \begin{align*}
        \Bnorm{\tilde J(\bfW,\U,\xi) - \tilde J(\bfW,\tilde\U,\xi)}\le \tilde C \Bnorm{\U-\tilde\U}.
    \end{align*}
\end{lemma}
\begin{proof}
    In this proof we will omit the Gubinelli derivative from the $\Dnorm{\cdot}$ for clarity. We again start with 
    \begin{align*}
        \tilde J(\bfW,\U,\xi)[j-1,t] 
        &= \sum_{k=-\I}^{j-1}\left[S^s(t+j-1-k) ~T_1\left(\Theta_{k-1}\bfW,x^{k-1},U^{k-1},\boldsymbol{\alpha}^{k-1}\right)\right]\\
        &~~~~+T_t\left(\Theta_{j-1}\bfW,x^{j-1},U^{j-1},\boldsymbol{\alpha}^{j-1}\right),
    \end{align*}
    where $T_t\left(\Theta_{k-1}\bfW,x^{k-1},U^{k-1},\boldsymbol{\alpha}^{k-1}\right)$ is defined in \eqref{def T}.\\
    Hence, with the triangle inequality and \eqref{beta} we get
    \begin{align*}
        &\Bnorm{\tilde J(\bfW,\U,\xi)-\tilde J(\bfW,\tilde\U,\xi)} \\
        &= \sup_{j\in\Z^-}e^{-\eta(j-1)}\Dnorm{\tilde J(\bfW,\U,\xi)[j-1,\cdot]-\tilde J(\bfW,\tilde\U,\xi)[j-1,\cdot]}\\
        &\le \sup_{j\in\Z^-}e^{-\eta(j-1)}\Bigg[\sum_{k=-\I}^{j-1}M_se^{-\beta(j-1-k)} \\
        &~~~~ \Dnorm{ T\left(\Theta_{k-1}\bfW,x^{k-1},U^{k-1},\balpha^{k-1}\right) - T\left(\Theta_{k-1}\bfW,x^{k-1},\tilde U^{k-1},\balpha^{k-1}\right)} \\
        &~~~~ + \Dnorm{ T\left(\Theta_{j-1}\bfW,x^{j-1},U^{j-1},\balpha^{j-1}\right) - T\left(\Theta_{j-1}\bfW,x^{j-1},\tilde U^{j-1},\balpha^{j-1}\right)}\Bigg].        
    \end{align*}
    To this goal, we find a bound of the form $CR\Dnorm{U^{k-1}-\tilde U^{k-1}}$ for $\Dnorm{ T\left(\Theta_{k-1}\bfW,x^{k-1},U^{k-1},\balpha^{k-1}\right) - T\left(\Theta_{k-1}\bfW,x^{k-1},\tilde U^{k-1},\balpha^{k-1}\right)}$, for $k\in\Z^-$.
    First, we use Lemma \ref{integral bound} to bound this difference in terms of $I$ and $II$ defined in \eqref{def I} respectively \eqref{def II}. This immediately entails
    \begin{align*}
        &\Dnorm{T\left(\Theta_{k-1}\bfW,x^{k-1},U^{k-1},\balpha^{k-1}\right)- T\left(\Theta_{k-1}\bfW,x^{k-1},\tilde U^{k-1},\balpha^{k-1}\right)}\\
        &\le C[\Xnorm{A}]\Bigg(\norm{I\left(x^{k-1},U^{k-1},\balpha^{k-1}\right) - I\left(x^{k-1},\tilde U^{k-1},\balpha^{k-1}\right)}_\I \\
        &~~+ Z(\Theta_{k-1}\bfW) \Dnorm{II\left(x^{k-1},U^{k-1},\balpha^{k-1}\right) - II\left(x^{k-1},\tilde U^{k-1},\balpha^{k-1}\right)}\Bigg),
    \end{align*}
    where $Z(\Theta_{k-1}\bfW)$ is specified in \eqref{Z def}.\\
    
    Next we find bounds for $\norm{I\left(x^{k-1},U^{k-1},\balpha^{k-1}\right) - I\left(x^{k-1},\tilde U^{k-1},\balpha^{k-1}\right)}_\I$ and $\Dnorm{II\left(x^{k-1},U^{k-1},\balpha^{k-1}\right) - II\left(x^{k-1},\tilde U^{k-1},\balpha^{k-1}\right)}$. \\
    
    By the triangle inequality, we get
    \begin{align*}
        &\norm{I\left(x^{k-1},U^{k-1},\balpha^{k-1}\right) - I\left(x^{k-1},\tilde U^{k-1},\balpha^{k-1}\right)}_\I \\
        &\le \Big\lVert F_R^s\left(x^{k-1},U^{k-1} + \phi^{k-1}\left(x^{k-1}\right)\right) - F_R^s\left(x^{k-1},\tilde U^{k-1} + \phi^{k-1}\left(x^{k-1}\right)\right) \Big\rVert_\I \\
        &+ \left\lVert \sum_{i=2}^{q} i\alpha^{k-1}_i \left(x^{k-1}\right)^{i-1} \left[ F_R^c\left(x^{k-1},U^{k-1} + \phi^{k-1}\left(x^{k-1}\right)\right) - F_R^c\left(x^{k-1},\tilde U^{k-1} + \phi^{k-1}\left(x^{k-1}\right)\right) 
        \right]\right\rVert_\I.
    \end{align*}
    Now we examine each summand. For the first summand we use Lemma \ref{lip:fr} to get
    \begin{align*}
        &\norm{F_R^s\left(x^{k-1},U^{k-1} + \phi^{k-1}\left(x^{k-1}\right)\right) - F_R^s\left(x^{k-1},\tilde U^{k-1} + \phi^{k-1}\left(x^{k-1}\right)\right)}_\I \\
        &\le C_{F^s}R(\Theta_{k-1}\bfW)\Dnorm{U^{k-1} - \tilde U^{k-1}}.
    \end{align*}
    Moreover, using \eqref{sum i alpha bound} we obtain for the second term
    \begin{align*}
        &\norm{\sum_{i=2}^{q} i\alpha^{k-1}_i \left(x^{k-1}\right)^{i-1} \left[F_R^c\left(x^{k-1}, U^{k-1} + \phi^{k-1}\left(x^{k-1}\right)\right) - F_R^c\left(x^{k-1},\tilde U^{k-1} + \phi^{k-1}\left(x^{k-1}\right)\right)\right]}_\I\\
        &\le q^2 C_{F^c} \tilde C_x^{q-1} R(\Theta_{k-1}\bfW) \Dnorm{ U^{k-1} - \tilde U^{k-1}}.
    \end{align*}
    Putting the estimates together we infer
    \begin{align*}
        &\norm{I\left(x^{k-1},U^{k-1},\balpha^{k-1}\right)}_\I \\
        &\le \left(C_{F^s} + q^2 C_{F^c} \tilde C_x^{q-1} \right)R(\Theta_{k-1}\bfW)\Dnorm{U^{k-1} - \tilde U^{k-1}}.
    \end{align*}
    From the definition of $II$ \eqref{def II} we get
    \begin{align*}
        &\Dnorm{II\left(x^{k-1},U^{k-1},\balpha^{k-1}\right) - II\left(x^{k-1},\tilde U^{k-1},\balpha^{k-1}\right)} \\
        &\le \Dnorm{G_R^s\left(x^{k-1},U^{k-1} + \phi^{k-1}\left(x^{k-1}\right)\right) - G_R^s\left(x^{k-1},\tilde U^{k-1} + \phi^{k-1}\left(x^{k-1}\right)\right)} \\
        &+ \Dnorm{ \sum_{i=2}^{q} i\alpha^{k-1}_i \left(x^{k-1}\right)^{i-1} \left[ G_R^c\left(x^{k-1},U^{k-1} + \phi^{k-1}\left(x^{k-1}\right)\right) - G_R^c\left(x^{k-1},\tilde U^{k-1} + \phi^{k-1}\left(x^{k-1}\right)\right) 
        \right]}.
    \end{align*}
    Similarly to the computation above for $I$ we use Lemma \ref{lem:addition and multiplication of RP} and Lemma \ref{composition:cutoff} to infer that
    \begin{align*}
        &\Dnorm{II\left(x^{k-1},U^{k-1},\balpha^{k-1}\right) - II\left(x^{k-1},\tilde U^{k-1},\balpha^{k-1}\right)} \\
        &\le \left(C_{G^s}[\norm{\Theta_{k-1}W}_\gamma] + q^2 C_{G^c}[\norm{\Theta_{k-1}W}_\gamma]\tilde C_x^{q-1}\left(1+\norm{\Theta_{k-1}W}_\gamma\right)^6 \right) R(\Theta_{k-1}\bfW) \Dnorm{U^{k-1} - \tilde U^{k-1}}.
    \end{align*}
    Combining the estimates for $I$ and $II$ we get 
    \begin{align*}
        &\Dnorm{T\left(\Theta_{k-1}\bfW,x^{k-1},U^{k-1},\balpha^{k-1}\right) - T\left(\Theta_{k-1}\bfW,x^{k-1},\tilde U^{k-1},\balpha^{k-1}\right)}\\
        &\le C[\Xnorm{A}]\Bigg(\norm{I\left(x^{k-1},U^{k-1},\balpha^{k-1}\right) - I\left(x^{k-1},\tilde U^{k-1},\balpha^{k-1}\right)}_\I\\
        &~~~~+ Z(\Theta_{k-1}\bfW)\Dnorm{II\left(x^{k-1},U^{k-1},\balpha^{k-1}\right) - II\left(x^{k-1},\tilde U^{k-1},\balpha^{k-1}\right)}\Bigg)\\
        &\le C[\Xnorm{A}]\Bigg(C_{F^s} + q^2 C_{F^c} \tilde C_x^{q-1} + Z(\Theta_{k-1}\bfW)\\
        &~~~~\left(C_{G^s}[\norm{\Theta_{k-1}W}_\gamma] + q^2 C_{G^c}[\norm{\Theta_{k-1}W}_\gamma] \tilde C_x^{q-1} \right)\Bigg) R(\Theta_{k-1}\bfW)\Dnorm{U^{k-1} - \tilde U^{k-1}}.
    \end{align*}
    Similarly to the proof of Lemma \ref{tilde J maps into V}, we can find for $\hat C_T>0$ some cut-off radius $\tilde R_4(\Theta_{k-1}\bfW)$ such that 
    \begin{align}\label{def:tilde R_4}
        &C[\Xnorm{A}]\Bigg(C_{F^s} + q^2 C_{F^c} \tilde C_x^{q-1}   + Z(\Theta_{k-1}\bfW)\nonumber \\
        &~~~~\left(C_{G^s}[\norm{\Theta_{k-1}W}_\gamma] + q^2 C_{G^c}[\norm{\Theta_{k-1}W}_\gamma]\tilde C_x^{q-1}\left(1+\norm{\Theta_{k-1}W}\right)^6 \right)\Bigg) \tilde R_4(\Theta_{k-1}\bfW) = \hat C_T,
    \end{align}
    where $\hat C_T$ will be given down below. Choosing $\tilde R_4(\bfW)$ such that \eqref{def:tilde R_4} holds, we infer
    \begin{align*}
        &\Dnorm{T\left(\Theta_{k-1}\bfW,x^{k-1},U^{k-1},\balpha^{k-1}\right) - T\left(\Theta_{k-1}\bfW,x^{k-1},\tilde U^{k-1},\balpha^{k-1}\right)}\\
        &\le \hat C_T \Dnorm{U^{k-1} - \tilde U^{k-1}}.
    \end{align*}
Therefore we further get with the definition of $\beta$ \eqref{beta} and the definition of the $BC^\eta$-norm \ref{BC def}
    \begin{align*}
        &e^{-\eta(j-1)}\Bigg[\sum_{k=-\I}^{j-1}M_se^{-\beta(j-1-k)}\Dnorm{ T\left(\Theta_{k-1}\bfW,x^{k-1},U^{k-1},\balpha^{k-1}\right) - T\left(\Theta_{k-1}\bfW,x^{k-1},\tilde U^{k-1},\balpha^{k-1}\right)} \\
        &~~~~ + \Dnorm{ T\left(\Theta_{j-1}\bfW,x^{j-1},U^{j-1},\balpha^{j-1}\right) - T\left(\Theta_{j-1}\bfW,x^{j-1},\tilde U^{j-1},\balpha^{j-1}\right)}\Bigg] \\
        &\le\sum_{k=-\I}^{j-1}M_se^{-\beta(j-1-k)}e^{-\eta (j-1)}\hat{C}_T\Dnorm{U^{k-1} - \tilde U^{k-1}}  + e^{-\eta(j-1)}\hat{C}_T\Dnorm{U^{k-1} - \tilde U^{k-1}}\\
        &\le \hat{C}_T\left(\frac{M_se^{-\eta}}{1-e^{-(\eta+\beta)}}+1 \right)\Bnorm{\U - \tilde \U},
    \end{align*}
    where we used Assumption \ref{beta assumption} so $-(\eta+\beta)<0$.
    Now we can easily estimate 
    \begin{align*}
        &\Bnorm{\tilde J(\bfW,\U,\xi)-\tilde J(\bfW,\tilde\U,\xi)} \\
        &\le\sup_{j\in\Z^-} e^{-\eta(j-1)}\Bigg[\sum_{k=-\I}^{j-1}M_se^{-\beta(j-1-k)} \\
        &~~~~\Dnorm{ T\left(\Theta_{k-1}\bfW,x^{k-1},U^{k-1},\balpha^{k-1}\right) - T\left(\Theta_{k-1}\bfW,x^{k-1},\tilde U^{k-1},\balpha^{k-1}\right)} \\
        &~~~~ + \Dnorm{ T\left(\Theta_{j-1}\bfW,x^{j-1},U^{j-1},\balpha^{j-1}\right) - T\left(\Theta_{j-1}\bfW,x^{j-1},\tilde U^{j-1},\balpha^{j-1}\right)}\Bigg] \\
        &\le  \hat{C}_T\left(\frac{M_se^{-\eta}}{1-e^{-(\eta+\beta)}}+1 \right)\Bnorm{\U - \tilde \U}.
    \end{align*}
    Our goal is to show that there exists a constant $\tilde C < 1$ such that 
    $$\Bnorm{\tilde J(\bfW,\U,\xi)-\tilde J(\bfW,\tilde\U,\xi)} \le \tilde C\Bnorm{\U - \tilde \U}.$$
    Based on the previous computations, the inequality above reduces to
    $$\hat C_T\left(\frac{M_se^{-\eta}}{1-e^{-(\eta+\beta)}}+1 \right)< 1.$$
    Hence we choose $\hat C_T$ such that
    $$\hat C_T < 1/\left(\frac{M_se^{-\eta}}{1-e^{-(\eta+\beta)}}+1 \right)$$
    With this we can determine $\tilde R_4(\bfW)$ with \eqref{def:tilde R_4}.
    We define 
    \begin{align}\label{def:R_4}
        R_4(\bfW):=\min\{\tilde R_4(\bfW),R_3(\bfW)\},
    \end{align}
    where $R_3(\bfW)$ is defined in \eqref{def:R_3}, and get
    $$\Bnorm{\tilde J(\bfW,\U,\xi) - \tilde J(\bfW,\tilde\U,\xi)}\le \tilde C \Bnorm{\U-\tilde\U},$$
    for $\tilde C:= \hat C_T\left(\frac{M_se^{-\eta}}{1-e^{-(\eta+\beta)}}+1 \right) <1$.
    \qed\\
\end{proof}

Combining the previous results we can finally show the main result of this work. 
\begin{theorem}\label{thm:approx:m}
    Let Assumptions \ref{ass:linearpart}, \ref{F}, \ref{ass:alpha}, \ref{beta assumption} and \ref{M assumption} be satisfied. Let $\xi\in B_{\cX^c}(0,\rho(\bfW))$ and let $x$ be the solution of the first equation of \eqref{SDE} with initial datum $\xi$. We further let $x^{j-1}_t=x_{j-1+t}$ be the discretization of $x$ as specified in Lemma \ref{xj bound}, where $j\in\Z^-$ and $t\in[0,1]$. Moreover let for $i\in\{2,\dots,q\},~j\in\Z^-$ $\left(\alpha^{j-1}_i,\left(\alpha_i^{j-1}\right)'\right)$ be the solution of \eqref{discretized:alpha} and $$\phi^{j-1}\left(x^{j-1}_t\right)=\sum_{i=2}^{q}\alpha_i^{j-1}\left(x^{j-1}_t\right)^i.$$
    Then
    $$\Xnorm{h^c(\xi,\bfW)-\phi^{-1}\left(x^{-1}_1\right)}\le K\Xnorm{\xi}^{q+1}.$$
\end{theorem}
\begin{proof}
    Due to Lemma \ref{tilde J maps into V} and Lemma \ref{tilde J contraction} we know that a fixed point $\U^*\in V$ of $\tilde J(\bfW,\cdot,\xi)$ exists truncating  $F^{c/s}$, $G^{c/s}$ with $R_4(\bfW)$ defined in \eqref{def:R_4} and using Assumption \ref{M assumption} with $R_4(\bfW)$. Hence
    \begin{align*}
    \U^*=\tilde J(\bfW,\U^*,\xi)&=P^sJ(\bfW,\U^*+\Phi,\xi)-\Phi
    \end{align*}
    and therefore
    $$\U^*+\Phi = P^sJ(\bfW,\U^*+\Phi,\xi),$$
    where $\Phi$ is defined in \eqref{Phi def}.
    So $\U^*+\Phi$ is a fixed point of $P^sJ$. We also know from Lemma \ref{contraction} that $P^sJ$ has the unique fix point $P^s\Gamma(\xi,\bfW)$. Consequently, it must hold that $\U^*+\Phi=P^s\Gamma$.\\
    Based on this and on the definition of $V$, we get
    \begin{align*}
        \Xnorm{h^c(\xi, \bfW)-\phi^{-1}_1\left(x^{-1}_1\right)}
        &= \Xnorm{P^s\Gamma(\xi,\bfW)[-1,1]-\phi^{-1}_1\left(x^{-1}_1\right)} \\
        &= \Xnorm{P^sJ(\bfW,P^s\Gamma(\xi,\bfW),\xi)[-1,1]-\phi^{-1}_1\left(x^{-1}_1\right)}\\
        &= \Xnorm{P^sJ(\bfW,\U^*+\Phi,\xi)[-1,1]-\phi^{-1}_1\left(x^{-1}_1\right)} \\
        &= \Xnorm{\tilde J(\bfW,\U^*,\xi)[-1,1]} \\
        &\le\norm{\tilde J(\bfW,\U^*,\xi)[-1,\cdot]}_\I\\
        &\le\Dnorm{\tilde J(\bfW,\U^*,\xi)[-1,\cdot], \tilde J(\bfW,\U^*,\xi)[-1,\cdot]'}\\
        &\le K\Xnorm{\xi}^{q+1}.
    \end{align*}
    \qed\\
\end{proof}

\section{Applications}~\label{sec:app}
\subsection{Examples}\label{examples}
Both our approach and the approach in \cite{ChekrounLiuWang} entail Taylor expansions for the center manifolds with coefficients being stationary solutions of SDEs. Hence, we can easily compare our results to those obtained by~\cite{ChekrounLiuWang} and \cite{ChekrounLiuWang example} for linear multiplicative noise. For a better comprehension we first treat this case. 
\begin{example}\label{Chekroun linear example}
We consider the following system driven by a one-dimensional Brownian motion $(W_t)_{t\geq 0}$ as in~\cite[Section 2.3]{ChekrounLiuWang example}
\begin{align}\label{chekroun example}
    \txtd x &= (\lambda x + xy)~\txtd t,~~\lambda\ge0 \nonumber\\
    \txtd y &= (\kappa y - x^2)~\txtd t + \sigma y\circ\txtd W_t,~~\kappa<0,~\sigma>0. 
\end{align}

\begin{remark}
    \begin{itemize}
        \item [1)] This example arises in the finite time blow-up for an averaged 3D Navier-Stokes equation \cite{Tao} 
       and in the study of nonlinear crystals \cite{SchenzleBrand}.
        \item [2)] As $\lambda$ has potentially not zero real-part, we now have a splitting of the phase space $\cX=\R^2$ in a stable and a center-unstable one. This is not exactly the setting of Assumption \ref{ass:linearpart} but we can modify this assumption in order to incorporate this situation as well.~Therefore and due to the motivation mentioned above we consider this example in order to illustrate the approximation techniques and compare them with~\cite[Section 2.3]{ChekrounLiuWang example}.
    \end{itemize}
\end{remark}

We are interested in determining the reduced flow on the center-unstable manifold. 
The reduced flow gives the dynamics on the manifold and is given by the differential equation in $x$ after substituting $y$ by $h^c(x,W)$ or some approximation of it. To this end we need to compute the approximation of $h^c$ given by \eqref{h:cm} and then plug it in the differential equation for $x$. 
The approximation derived in~\cite[Section 2.3]{ChekrounLiuWang example} is 
\[ h^c(x,W)= \alpha(W) x^2 +\cO(x^3), \]
leading to the reduced flow 
$$\dot x = \lambda x_t - \alpha(\theta_t W)x_t^3.$$
Here $\alpha$ is the stationary solution of the linear SDE with Stratonovich noise
\begin{equation}\label{alpha:c}
\txtd \alpha = (1 - (2\lambda-\kappa)\alpha)~\txtd t + \sigma \alpha\circ\txtd W_t
\end{equation}
given by 
\[ \alpha(\theta_t W):= \int_{-\infty}^0 e^{(2\lambda-\kappa)s-\sigma \theta_tW_s}~\txtd s. \]
\begin{remark}
Note that $2\lambda>\kappa$ which means that the previous integral is well-defined.~Restrictions of this type are often referred to as non-resonance conditions. 
\end{remark}

For the sake of completeness, we present three ways in order to compute the approximation of $h^c$ together with the reduced flow. \\\\
1. Flow transformation:
The first option to compute the reduced flow follows a similar idea to \cite{ChekrounLiuWang example}. This approach is discussed in the next section and in Remark \ref{rem:alternative approximation results}.~As we consider linear multiplicative noise, we can use the Doss-Sussmann transformation $y^*_t:=e^{-\sigma 
z(\theta_tW)}y_t$ where $z(\theta_tW)$ is the stationary Ornstein Uhlenbeck process defined in \eqref{ex:stationary}. 
\begin{align*}
    \txtd x_t &= (\lambda x_t + e^{\sigma z(\theta_tW)}x_ty_t^*)~\txtd t,\\
    \txtd y^*_t &= ((\sigma z(\theta_tW)+\kappa) y^*_t - e^{-\sigma z(\theta_tW)}x_t^2)~\txtd t .
\end{align*}
Now we make the ansatz $y^*_t = \phi(x_t) = \sum_{i=1}^4\alpha_i(\theta_tW)x^i_t$ with $\txtd \alpha_i = (A^{\alpha_i}(\theta_tW)\alpha_i+f_i(\theta_tW))~\txtd t.$ We omit the dependence on $\theta_tW$ for notational simplicity. With this ansatz we can calculate $\txtd y_t^*$ in two ways. First, we can use the transformed differential equation for $y^*_t$ and plug in $\phi(x_t)$ which entails
\begin{align*}
    \txtd y_t^* = \left[(\sigma z(\theta_tW) + \kappa)\phi(x_t) - e^{-\sigma z(\theta_tW)}x^2_t \right]~\txtd t.
\end{align*}
Second, we can use the definition of $\phi(x_t)$ and the product rule to get
\begin{align*}
    \txtd y_t^* &= \txtd \phi(x_t)= \sum_{i=1}^4 x^i_t~\txtd \alpha_i + i\alpha_i x^{i-1}_t~\txtd x_t \\
    &= \left[\sum_{i=1}^4A^{\alpha_i}\alpha_i x^i_t + f_ix^i_t + i\lambda \alpha_i x^i_t + e^{\sigma z(\theta_tW)}i\alpha_i\phi(x_t) x^i_t\right] \txtd t.
\end{align*}
Comparing the coefficients of the linear terms $\alpha_ix^i_t$ we infer
$$A^{\alpha_i} = \sigma z(\theta_tW)+\kappa - i\lambda.$$
Comparing the remaining terms we have 
\begin{align*}
    \sum_{i=1}^4 f_ix^i_t = -e^{-\sigma z(\theta_tW)}x^2_t - \sum_{i=1}^4e^{\sigma z(\theta_tW)}i\alpha_i\phi(x_t)x_t^i.
\end{align*}
Simplifying $\sum_{i=1}^4e^{\sigma z(\theta_tW)}i\alpha_i\phi(x_t)x_t^i$ we compute
\begin{align*}
    f_1 &= 0 \\
    f_2 &= -e^{-\sigma z(\theta_tW)} - e^{\sigma z(\theta_tW)} \alpha_1^2 \\
    f_3 &= - e^{\sigma z(\theta_tW)} 3\alpha_1\alpha_2 \\
    f_4 &= - e^{\sigma z(\theta_tW)} \left(4\alpha_1\alpha_3+2\alpha_2^2\right).
\end{align*}
So the coefficients solve the following system with non-autonomous random coefficients
\begin{align*}
    \txtd \alpha_1 &=  \left[\sigma z(\theta_tW)+\kappa - \lambda \right]\alpha_1 ~\txtd t \\
    \txtd \alpha_2 &=  \left[(\sigma z(\theta_tW)+\kappa - 2\lambda)\alpha_2 -e^{-\sigma z(\theta_tW)} - e^{\sigma z(\theta_tW)} \alpha_1^2 \right] ~\txtd t \\
    \txtd \alpha_3 &=  \left[(\sigma z(\theta_tW)+\kappa - 3\lambda)\alpha_3 - e^{\sigma z(\theta_tW)} 3\alpha_1\alpha_2 \right] ~\txtd t \\
    \txtd \alpha_4 &=  \left[(\sigma z(\theta_tW)+\kappa - 4\lambda)\alpha_4 - e^{\sigma z(\theta_tW)} (4\alpha_1\alpha_3 + 2\alpha_2^2) \right] ~\txtd t.
\end{align*}
 We see that the differential equation for $\alpha_1$ is solved by $\alpha_1\equiv0$. Hence, $\alpha_3\equiv0$ solves the third differential equation above. The reduced system is given by
\begin{align*}
    \txtd \alpha_2 &=  \left[(\sigma z(\theta_tW)+\kappa - 2\lambda)\alpha_2 -e^{-\sigma z(\theta_tW)} \right] ~\txtd t \\
    \txtd \alpha_4 &=  \left[(\sigma z(\theta_tW)+\kappa - 4\lambda)\alpha_4 - e^{\sigma z(\theta_tW)} 2\alpha_2^2 \right] ~\txtd t.
\end{align*}
Using the transformation $\alpha_i^*:=e^{\sigma z(\theta_tW)}\alpha_i$ we get the SDEs 
\begin{align*}
    \txtd \alpha_2^* &=  \left[(\kappa - 2\lambda)\alpha_2^* -1 \right] ~\txtd t + \sigma\alpha_2^*\circ\txtd W_t \\
    \txtd \alpha_4^* &=  \left[(\kappa - 4\lambda)\alpha_4^* - 2(\alpha_2^*)^2 \right] ~\txtd t + \sigma\alpha_4^*\circ\txtd W_t.
\end{align*}
\\\\
2. Direct Taylor expansion: We make the ansatz $y_t=\phi(x_t)$, where $\phi(x_t)=\sum\limits_{i=1}^4\alpha_i(\theta_tW)x_t^i$ and $\alpha_i$ solve 
\begin{align*}
    \txtd \alpha_i = [A^{\alpha_i}\alpha_i + f_i]~\txtd t + g_i \circ\txtd W_t,~~i\in\{1,\ldots q\}, 
\end{align*}
for some coefficients $A^{\alpha_i},f_i,g_i$ that we have to determine. 
For notational simplicity we drop the argument $\theta_tW$ for $\alpha_i(\theta_tW)$ and write only $\alpha_i$.
Plugging in the ansatz $y_t=\phi(x_t)$ in \eqref{chekroun example} we get
\begin{align*}
    \txtd y_t &= (\kappa \phi(x_t) - x_t^2)~\txtd t + \sigma \phi(x_t)\circ\txtd W_t \\
    &= \left(\kappa\alpha_1 x_t + (\kappa\alpha_2-1) x_t^2 + \kappa\alpha_3 x_t^3 + \kappa\alpha_4x_t^4 \right) \txtd t \\
    &~~+\left(\sigma\alpha_1 x_t + \sigma\alpha_2 x_t^2 + \sigma\alpha_3 x_t^3 + \sigma\alpha_4 x_t^4 \right)\circ\txtd W_t.
\end{align*}
Moreover, we can compute $\txtd y_t$ as $\txtd \phi(x_t)$. This entails 
\begin{align*}
    \txtd y_t &= \sum_{i=1}^4i\alpha_ix_t^{i-1} ~\txtd x_t + \sum_{i=1}^4 x_t^i~\txtd \alpha_i \\
    &= \left(\sum_{i=1}^4i\alpha_i\lambda x_t^i + i\alpha_ix_t^i\phi(x_t) + \left(A^{\alpha_i}\alpha_i+f_ix_t^i\right)\right)~\txtd t +\sum_{i=1}^4 g_ix_t^i \circ\txtd W_t \\
    &= \big(\left[\lambda\alpha_1 + \left(A^{\alpha_1}\alpha_1+f_1\right) \right] x_t \\
    &~~+\left[2\lambda\alpha_2 + \alpha_1^2 + \left(A^{\alpha_2}\alpha_2+f_2\right)\right] x_t^2 \\
    &~~+\left[3\lambda\alpha_3 + 3\alpha_1\alpha_2 + \left(A^{\alpha_3}\alpha_3+f_3\right)\right] x_t^3 \\
    &~~+\left[4\lambda \alpha_4 + 4\alpha_1\alpha_3 + 2\alpha_2^2 + \left(A^{\alpha_4}\alpha_4+f_4\right) \right] x_t^4\big) ~\txtd t \\
    &~~+ \left(g_1 x_t + g_2 x_t^2 + g_3 x_t^3 + g_4x_t^4\right)\circ\txtd W_t,
\end{align*}
where we dropped all terms of order $x_t^5$ and higher. 
Comparing the coefficients we get for $A^{\alpha_i}$, $f_i$ and $g_i$
\begin{align*}
    \begin{array}{cccccc}
        A^{\alpha_1}&=\kappa-\lambda  & f_1= & 0 & g_1= & \sigma\alpha_1\\
        A^{\alpha_2}&=\kappa-2\lambda & f_2= & - \alpha_1^2 -1 & g_2= & \sigma\alpha_2 \\
        A^{\alpha_3}&=\kappa-3\lambda & f_3= & - 3\alpha_1\alpha_2 & g_3= & \sigma\alpha_3 \\
        A^{\alpha_4}&=\kappa-4\lambda & f_4= & - 4\alpha_1\alpha_3 - 2\alpha_2^2 & g_4= & \sigma\alpha_4.\\
    \end{array}
\end{align*}
First we note that $\alpha_1\equiv0$ solves the SDE
$$\txtd \alpha_1 = (\kappa-\lambda) \alpha_1~\txtd t + \sigma\alpha_1\circ\txtd W_t.$$
Moreover, plugging $\alpha_1\equiv0$ in $f_3$ we get that $\alpha_3$ solves the SDE
$$\txtd \alpha_3 = (\kappa-3\lambda)\alpha_3~\txtd t + \sigma\alpha_3\circ\txtd W_t,$$
which is solved by $\alpha_3\equiv 0$. Hence, the system our coefficients solve simplifies to
\begin{align*}
    \txtd \alpha_2 &= \left((\kappa-2\lambda)\alpha_2-1\right)~\txtd t + \sigma\alpha_2\circ\txtd W_t\\
    \txtd \alpha_4 &= \left((\kappa-4\lambda)\alpha_4-2\alpha_2^2\right)~\txtd t + \sigma\alpha_4\circ\txtd W_t.
\end{align*}

3. Coefficient determining equations: Let us also illustrate a different way to compute $A^{\alpha_i},~f_i$ and $g_i$ using the invariance equations defined in \eqref{invariance equation A}. \eqref{invariance equation f_i} and \eqref{invariance equation g_i}. We use the notation introduced in Section \ref{setting} 
\begin{align*}
\begin{array}{ccc}
    A^c = \lambda , &F^c(x_t,y_t) = x_ty_t, & G^c(x_t,y_t) = 0, \\
    A^s = \kappa , &F^s(x_t,y_t) = -x_t^2, & G^s(x_t,y_t) = \sigma y_t.
\end{array}
\end{align*}
We note that $G^s$ does not fulfill the Assumption \ref{G} and therefore the approximation result Theorem \ref{thm:approx:m} cannot be applied. But it is possible to use a Doss-Sussmann type transformation as mentioned in Remark \ref{rem:alternative approximation results}. 

From now on we will suppress the dependence on $t$ and $W$. 
We start with the linear part. Using $A^{\alpha_i}=A^s-iA^c$ we get
\begin{align*}
    \begin{array}{cc}
    A^{\alpha_1} &=\kappa-\lambda \\
    A^{\alpha_2} &=\kappa-2\lambda \\
    A^{\alpha_3} &=\kappa-3\lambda \\
    A^{\alpha_4} &=\kappa-4\lambda.
\end{array}
\end{align*}
Next we use the invariance equation for the $f_i$ \eqref{invariance equation f_i}
\begin{align*}
    \sum_{i=1}^4 f_ix_t^i 
    &= F^s(x_t,\phi(x_t)) - \sum_{i=1}^4 i\alpha_ix_t^{i-1}F^c(x_t,\phi(x_t))\\
    &= - x_t^2 - \sum_{i=1}^4 i\alpha_ix_t^i\phi(x_t) \\
    &= -x_t^2 -\left[\alpha_1^2x_t^2 + 3\alpha_1\alpha_2x_t^3 + \left(4\alpha_1\alpha_3 + 2\alpha_2^2\right)x_t^4 \right] + \cO(x_t^5),
\end{align*}
where the Landau symbol $\cO$ is defined as 
\begin{align}\label{def:Landau symbol}
    f\in\cO(g) ~~~\textnormal{ if } ~~~~\limsup\limits_{x\to 0}\Xnorm{\frac{f(x)}{g(x)}}<\I.
\end{align}

Comparing the coefficients we get
\begin{align*}
\begin{array}{cc}
    f_1= & 0 \\
    f_2= & -\alpha_1^2+1 \\
    f_3= & -3\alpha_1\alpha_2 \\
    f_4= & - 4\alpha_1\alpha_3 + 2\alpha_2^2.
\end{array}
\end{align*}

Last we use the invariance equation for $g_i$ \eqref{invariance equation g_i}
\begin{align*}
    \sum_{i=1}^4 g_ix_t^i &= G^s(x_t,\phi(x_t)) - \sum_{i=1}^4i\alpha_i x_t^{i-1}G^c(x_t,\phi(x_t))\\
    &= \sigma \phi(x_t) = \sum_{i=1}^4 \sigma\alpha_ix_t^i,
\end{align*}
so comparing the coefficients entails
\begin{align*}
\begin{array}{cc}
    g_1= & \sigma\alpha_1  \\
    g_2= & \sigma\alpha_2  \\
    g_3= & \sigma\alpha_3  \\
    g_4= & \sigma\alpha_4 .
\end{array}
\end{align*}

Hence we obtain the same results as with the second approach. We get again $\alpha_1\equiv0$ and $\alpha_3\equiv0$ and 
\begin{align}\label{alpha2}
    \txtd \alpha_2 &= \left((\kappa-2\lambda)\alpha_2-1\right)~\txtd t + \sigma\alpha_2\circ\txtd W_t  \\
    \txtd \alpha_4 &= \left((\kappa-4\lambda)\alpha_4-2\alpha_2^2\right)~\txtd t + \sigma\alpha_4\circ\txtd W_t.\nonumber
\end{align}
\\\\
We see that all three methods entail the same SDEs for the coefficients.~Plugging in their stationary solutions $\alpha_2,~\alpha_4$, we get the following approximation 
$$h^c(x,W)=\alpha_2(W)x^2 + \alpha_4(W) x^4 + \cO(x^5)$$
and the reduced flow
\begin{align*}
    \txtd x_t = \lambda x_t + \alpha_2(\theta_t W) x_t^3 + \alpha_4(\theta_t W) x_t^5 ~\txtd t.
\end{align*}
\begin{remark}
    Note that the only difference is the sign in front of $\alpha$ given in~\eqref{alpha:c} and $\alpha_2$ in~\eqref{alpha2} which cancels out by the fact that we have $1-(2\lambda-\kappa)$ as the drift for $\alpha$ and $-1-(2\lambda-\kappa)$ as the drift for $\alpha_2$. If we choose to only take the first order of the Taylor approximation, we get a consistent result with \cite{ChekrounLiuWang example}. But we can also take higher order terms into account if it is necessary. 
\end{remark}

Let us now compute the error terms for $M\phi(x_t)$ and $\tilde M\phi(x_t)$. We want to check if they have indeed the structure claimed in Lemma \ref{lem:M bound}. We first calculate $M\phi(x_t)$ which is given by the remainder $\cO(x_t^5)$ because we use the ansatz $M\phi(x_t)=0$ in order to calculate the coefficients $f_i$. To this aim we compute
\begin{align*}
    M \phi(x_t)
    &= \sum_{i=1}^4 f_ix_t^i - F^s(x_t,\phi(x_t)) + \sum_{i=1}^4 i\alpha_ix_t^{i-1}F^c(x_t,\phi(x_t))\\
    &= \left[-\alpha_1^2x_t^2 - 3\alpha_1\alpha_2 x_t^3 - (4\alpha_1\alpha_3 + 2\alpha_2^2)x_t^4 + \sum_{i=1}^4 i\alpha_ix_t^i\phi(x_t)\right] \\
    &=\left[(5\alpha_1\alpha_4 + 5\alpha_2\alpha_3) x_t^5 
    + (6\alpha_2\alpha_4+3\alpha_3^2) x_t^6
    + 7\alpha_3\alpha_4 x_t^7
    + 4\alpha_4^2 x_t^8\right] \\
    &=\left[6\alpha_2\alpha_4 x_t^6
    + 4\alpha_4^2 x_t^8\right],
\end{align*}
where we used in the last equality that $\alpha_1\equiv 0$ and $\alpha_3\equiv 0$. As we can see, the coefficients in front of the $x^i$ are polynomials in $\alpha_1,\dots,\alpha_{i}$. These terms correspond to $\tilde P^{F^c}_i(\alpha_2,\dots,\alpha_{i-1})- P^{F^c}_i(\alpha_2,\dots,\alpha_{i-1})$ appearing in the proof of Lemma \ref{M bound}.
For $\tilde M\phi(x_t)$ we get directly $\tilde M\phi(x_t)=0$ as there are no higher order remainder terms in the calculation and hence $$\tilde M\phi(x_t)=\sum_{i=1}^4g_ix_t^i - G^s(x_t,\phi(x_t)) + \sum_{i=1}^4 i\alpha_ix_t^{i-1}G^c(x_t,\phi(x_t))=0.$$
\end{example}
\begin{example}\label{Chekroun nonlinear example}
    Now we consider the following system driven by a geometric rough path $\mathbf{W}=(W,\mathbb{W})$ 
    \begin{align}\label{chekroun example nonlinear noise}
    \txtd x_t &= (x_ty_t+x_t^3)~\txtd t + \sigma_1 x_t^2y_t~\txtd \bfW_t, \nonumber\\
    \txtd y_t &= (- y_t + x_t^2)~\txtd t + \sigma_2 y_t^3~\txtd \bfW_t,
\end{align}
where $\sigma_1,\sigma_2>0$ represent the noise intensities.
We use again the ansatz $y_t=\phi(x_t)=\sum_{i=1}^6\alpha_ix_t^i$ and $\alpha_i$ solves
$$\txtd \alpha_i = (A^{\alpha_i}\alpha_i + f_i)~\txtd t + g_i~\txtd \bfW_t,~~~i\in\{1,\ldots q\}.$$
We rewrite the system using the notation in Section \ref{setting}. 
\begin{align*}
\begin{array}{ccc}
    A^c = 0 , &F^c(x_t,y_t) = x_ty_t + x_t^3, & ~~G^c(y_t) =  x_t^2y_t ,\\
    A^s = -1  , &F^s(x_t,y_t) = x_t^2, & ~~G^s(y_t) =  y_t^3.
\end{array}
\end{align*}
We start with determining $A^{\alpha_i}$ and $f_i$. The invariance equation for $A^{\alpha_i}$ given in \eqref{invariance equation A} leads to
$$A^{\alpha_i} = -1.$$ 
In order to compute $f_i$ we consider the invariance equation \eqref{invariance equation f_i} to obtain
\begin{align*}
    \sum_{i=1}^6 f_i x^i_t 
    &= F^s(x_t,\phi(x_t)) - \sum_{i=1}^6 i \alpha_ix_t^{i-1}F^c(x_t,\phi(x_t)) \\
    &= x_t^2 - \sum_{i=1}^6 i\alpha_ix_t^{i-1}(x_t\phi(x_t)+x_t^3) \\
    &= (-\alpha_1^2 + 1) x_t^2 - (3\alpha_1\alpha_2 + \alpha_1)x_t^3 - (2\alpha_2^2 + 4\alpha_1\alpha_3 + 2\alpha_2) x_t^4 \\
    &~~~~ - (5\alpha_1\alpha_4 + 5\alpha_2\alpha_3 + 3\alpha_3)x_t^5 - (3\alpha_3^2 + 6\alpha_1\alpha_5 + 6\alpha_2\alpha_4 + 4\alpha_4)x_t^6+ \cO(x_t^7).
\end{align*}
Comparing the coefficients we conclude
\begin{align*}
    f_1&=  0 \\
    f_2&=  -\alpha_1^2+1 \\
    f_3&=  -3\alpha_1\alpha_2 - \alpha_1 \\
    f_4&=  - 4\alpha_1\alpha_3 + 2\alpha_2^2 + 2\alpha_2 \\
    f_5&=  - 5\alpha_1\alpha_4 + 5\alpha_2\alpha_3 + 3\alpha_3 \\
    f_6&=  - 6\alpha_1\alpha_5 + 6\alpha_2\alpha_4 + 3\alpha_3^2 + 4\alpha_4.
\end{align*}
Similarly we calculate $g_i$ using the invariance equation \eqref{invariance equation g_i}. This yields
\begin{align*}
    \sum_{i=1}^6 g_i x_t^i 
    &= G^s(x_t,\phi(x_t)) - \sum_{i=1}^6 i \alpha_ix_t^{i-1}G^c(x_t,\phi(x_t)) \\
    &= \phi(x_t)^3 - \sum_{i=1}^6 i \alpha_i x_t^{i+1}\phi(x_t) \\
    &= (\alpha_1^3 - \alpha_1^2)x_t^3 + (3\alpha_1^2\alpha_2-3\alpha_1\alpha_2)x_t^4  + (3\alpha_1^2\alpha_3 + 3\alpha_1\alpha_2^2 - \alpha_2^2 - 4\alpha_1\alpha_3)x_t^5 \\
    &~~~~+ (\alpha_2^3+6\alpha_1\alpha_2\alpha_3 - 5\alpha_1\alpha_4 - 5\alpha_2\alpha_3)x_t^6 + \cO(x_t^7).
\end{align*}
Comparing the coefficients we get
\begin{align*}
    g_1&=  0 \\
    g_2&=  0 \\
    g_3&=  \alpha_1^3-\alpha_1^2  \\
    g_4&=  3\alpha_1^2\alpha_2-3\alpha_1\alpha_2  \\
    g_5&= 3\alpha_1^2\alpha_3 + 3\alpha_1\alpha_2^2 - \alpha_2^2 - 4\alpha_1\alpha_3  \\
    g_6&=  \alpha_2^3+6\alpha_1\alpha_2\alpha_3 - 5\alpha_1\alpha_4 - 5\alpha_2\alpha_3 .
\end{align*}
For the first three coefficients, we obtain the following equations
\begin{align*}
    \txtd \alpha_1 &= -\alpha_1~\txtd t, \\
    \txtd \alpha_2 &= \left(-\alpha_2-\alpha_1^2+1\right)~\txtd t, \\
    \txtd \alpha_3 &= \left(-\alpha_3-3\alpha_1\alpha_2-\alpha_1\right)~\txtd t + (\alpha_1^3-\alpha_1^2)~\txtd \bfW_t.
\end{align*}
We note that $\alpha_1\equiv0$ solves the first RDE. Moreover, the RDE for $\alpha_3$ is solved by $\alpha_3\equiv0$. With this we continue with the RDEs for $\alpha_4,\alpha_5$ and $\alpha_6$ and plug in $\alpha_1\equiv0,\alpha_3\equiv0$ wherever they appear.
\begin{align*}
    \txtd \alpha_4 &= \left(-\alpha_4-2\alpha_2^2-2\alpha_2\right)~\txtd t, \\
    \txtd \alpha_5 &= -\alpha_5~\txtd t - \alpha_2^2~\txtd \bfW_t , \\
    \txtd \alpha_6 &= \left(-\alpha_6-6\alpha_2\alpha_4-4\alpha_4\right)~\txtd t + \alpha_2^3~\txtd \bfW_t.
\end{align*}
So, the system our coefficients solve is
\begin{align*}
    \txtd \alpha_2 &= \left(-\alpha_2+1\right)~\txtd t \\
    \txtd \alpha_4 &= \left(-\alpha_4-2\alpha_2^2-2\alpha_2\right)~\txtd t  \\
    \txtd \alpha_5 &= -\alpha_5~\txtd t - \alpha_2^2~\txtd \bfW_t \\
    \txtd \alpha_6 &= \left(-\alpha_6-6\alpha_2\alpha_4-4\alpha_4\right)~\txtd t + \alpha_2^3~\txtd \bfW_t.
\end{align*}
Then, the approximation of $h^c(x,\bfW)$ is given by
$$h^c(x,\bfW)= \alpha_2 x^2 + \alpha_4 x^4 + \alpha_5(\bfW)x^5 + \alpha_6(\bfW) x^6 + \cO(x^7)$$
leading to the reduced flow  
\begin{align*}
    \txtd x_t &= \left(\alpha_2+1\right)x_t^3 + \alpha_4x_t^5 + \alpha_5(\Theta_t\bfW)x_t^6+ \alpha_6(\Theta_t\bfW)x_t^7 ~\txtd t\\
    &~~~~+ \sigma_1\left(\alpha_2 x_t^4 + \alpha_4x_t^6 + \alpha_5(\Theta_t\bfW)x_t^7+ \alpha_6(\Theta_t\bfW)x_t^8\right) \txtd\bfW_t  .
\end{align*}
Here $\alpha_2,\alpha_4$ are independent of the rough path $\mathbf{W}$.~In order to capture its dependence, we need to consider a higher order of the Taylor approximation, at least up to order five. 
\end{example}
\begin{remark}
\begin{itemize}
    \item[1)] In \cite[Chapter 5]{Chekroun2} a general formula was obtained for the systems of SDEs determining the coefficients of the approximation. This is also possible with our approach, but  is more involved as we consider general Taylor polynomials and not only the leading-order terms of the nonlinear coefficients as in Appendix~\ref{ChekrounProofWithRP}. Hence, this aspect will be considered in a future work. 
    \item[2)] In the examples we did not consider the cut-off introduced in the sections before. This can be done because we consider a local approximation of the manifold. So, by choosing the neighborhood around $0$ small enough, we can remove the cut-off due to Lemma \ref{lem:omit cut-off}. 
    \item[3)] Example \ref{Chekroun linear example} does not fulfill the assumptions of Theorem \ref{thm:approx:m} because $\txtD G^{c/s}(0,0)\neq 0$, but can be treated by the standard flow transformation. Since for $G^s(x,y)= y^2$, $\txtD^2G(0)\neq 0$, we considered Example \ref{Chekroun nonlinear example} with $G^s(x,y)= y^3,~G^c(x,y)=x^2y$ for which Theorem \ref{thm:approx:m} can be applied. 
    \item[4)] In Example \ref{Chekroun nonlinear example} we have seen that some coefficients might be independent of the noise. In general if we have non-linear noise, we have at least one coefficient independent of the noise. In the setting we consider in this paper this is always the case because of the assumption $\txtD G^{c/s}(0,0)=0$. 
    Moreover, if $G^{c}\equiv0$ we can show that the first coefficient dependent on the noise is of order $i= k  l$ where $k$ is the lowest order of a term only depending on $x$ in $F^s$ and $l$ is the lowest order of $G^s$.
    \item[5)] In Example \ref{Chekroun nonlinear example} the coefficient $\alpha_5\neq 0$ because the nonlinear diffusion coefficient depends on $\alpha_2$.~In case of linear multiplicative noise, $\alpha_5\equiv0$ would obviously be a solution, as it happens in the deterministic case.~In conclusion, due to the presence of nonlinear diffusion coefficients, we obtain additional coefficients in the approximation of the center manifold that do not appear in the deterministic case.
\end{itemize}
\end{remark}

\subsection{Comparison with other approximation results}
We now compare our ansatz with the one by Chekroun et.~al.~\cite{ChekrounLiuWang,ChekrounLiuWang example} which gives a Taylor-like approximation of $h^c$ using the Lyapunov-Perron map.
 However, there are some significant differences concerning:
\begin{itemize}
    \item [1)] the order of the approximation;
    \item [2)] the type of noise.
\end{itemize}
Firstly, in \cite[Theorem 6.1]{ChekrounLiuWang} the non-linear term $F\in C^p$, with $p>l\ge 2$, is approximated by the leading order term $F_l$ of the Taylor approximation. More precisely~\cite{ChekrounLiuWang} consider a Taylor expansion of $F$ such that 
\begin{align}\label{fl} F(u) = F_l(u, \dots,u)+\cO\left(\norm{u}^{l+1}\right),
\end{align}
where Landau symbol $\cO$ is defined in \eqref{def:Landau symbol}. This ansatz eventually leads to an approximation error of the local center manifold of order $\Xnorm{\xi}^l$.  Due to this fact, the order of the approximation error is given by the problem and cannot be chosen. Based on our ansatz, the order of the Taylor approximation can be chosen through $q\le p$. Furthermore, even if we choose to only use the leading-order term of the Taylor approximation for $F$ as~\cite{ChekrounLiuWang} and additionally make the same ansatz for the nonlinear term $G$, Theorem~\ref{thm:approx:m} gives an approximation error order of $\Xnorm{\xi}^{l+1}$. \\

The second main difference in comparison to~\cite{ChekrounLiuWang} is given by the type of noise that can be considered. Due to our pathwise approach, we can treat nonlinear multiplicative noise. Moreover, the type of noise is not restricted to the Brownian motion. In our setting we need that $\mathbf{W}=(W,\mathbb{W})$ is a $\gamma$-H\"older geometric rough path cocycle as introduced in \ref{rough path cocycle}.~This framework is in particular applicable to fractional Brownian motion with Hurst index $H\in(1/3,1/2)$. 

The proof of \cite[Theorem 6.1]{ChekrounLiuWang} uses a flow transformation based on the stationary Ornstein-Uhlenbeck process which is applicable to linear multiplicative Stratonovich noise. For  a better comprehension, we sketch this approach here and compare it with our techniques.
To this aim we consider the SDE driven by a one-dimensional Brownian motion 
\begin{align}\label{eqstrat}
    \txtd u = [A u +F(u)]~\txtd t + \sigma u\circ~\txtd W_t,~~\sigma>0
\end{align}
where $A$ satisfies Assumption~\ref{ass:linearpart}.
We further introduce the stationary Ornstein-Uhlenbeck process, i.e.~the stationary solution of the Langevin equation
$$ \txtd z_t =-z~\txtd t + \txtd W_t,$$
which is given by
$$ z(\theta_{t}W)=\int\limits_{-\infty}^{t} \txte^{-(t-s)}~\txtd W_s=\int\limits_{-\infty}^{0} \txte^{s}~\txtd\theta_{t}W_{s},$$ recall Example~\ref{ex:stationary}.
Here $\theta$ denotes the usual Wiener-shift, i.e.~$\theta_{t}W_{s}:=W_{t+s} -W_{s}$ for $s,t\in\mathbb{R}$. In this case, using the Doss-Sussmann transformation $u^{*}:=u\txte^{-\sigma z(W)}$, the SDE~\eqref{eqstrat} reduces to the nonautonomous random differential equation
\begin{equation}\label{ou}
\txtd u =[ A u + \sigma z(\theta_{t}{W}) u]~\txtd t +e^{-\sigma z(\theta_t W)} F (e^{\sigma z(\theta_tW)}u )~\txtd t,
\end{equation}
where we dropped the $*$-notation. Note that no stochastic integrals appear in~\eqref{ou} and the approximation of center manifolds 
can be done at the level of the differential equation with random-nonautonomous coefficients using the representation~\eqref{fl} of the nonlinear term $F$.~Therefore the problem of approximating stochastic integrals does not occur in this approach. 

\begin{remark}\label{rem:alternative approximation results}
    \begin{itemize}
        \item [1)] The approximation results for center manifolds developed in  \cite{ChekrounLiuWang} carry over to the rough path setting as shown in Appendix \ref{ChekrounProofWithRP}.
        \item [2)] We decided to use the proof idea of \cite{Carr} instead, as we can choose the order of approximation and do not use the leading-order term $F_l$ as stated in~\eqref{fl}.
        \item [3)] As already stated, for
        systems with linear multiplicative noise given by
\begin{align*}
    \txtd x &= A^c x + F^c(x,y) ~\txtd t + \sigma x\circ\txtd W_t \\
    \txtd y &= A^s y + F^s(x,y) ~\txtd t + \sigma y\circ\txtd W_t,
\end{align*}
one can compute the center-manifold approximation using a Doss-Sussmann type transformation.~We refer to~\cite{Boxler1} for a similar setting and proof which is also based on~\cite{Carr}.~Applying the Doss-Sussmann transformation $x^*_t=x_te^{-\sigma z(\theta_tW)}$ to the SDEs above and making the ansatz $y_t=\phi(x_t) = \sum_{i=1}^q\alpha_ix_t^i$, where 
$$\txtd\alpha_i = A^{\alpha_i}\alpha_i~\txtd t + f_i ~\txtd t,$$
one can compute $\txtd\phi(x)$ in two ways and compare the coefficients.~The computation essentially simplifies, since no stochastic differentials occur, recall Example~\ref{examples}. 

    \end{itemize}
\end{remark}

\appendix
\section{Deterministic approach for the approximation of local center manifolds}\label{sec:carr}
For the convenience of the reader, we state here an approximation result of deterministic center manifolds and refer to \cite{Carr} for more details.~We consider the following deterministic system
\begin{align}\label{syst:carr}
\begin{cases}
\frac{\txtd x}{\txtd t}=    \dot x &= Ax + f(x,y) \\
  \frac{\txtd y}{\txtd t}= \dot y &= By + g(x,y),
\end{cases}
\end{align}
where $x\in \R^n,~y\in\R^m$ and $A$, $B$ are constant matrices such that all eigenvalues of $A$ have zero real part and all the eigenvalues of $B$ have negative real part.~The nonlinearities satisfy the assumptions $f,g\in C^2$ with $f(0,0)=\txtD f(0,0)=g(0,0)=\txtD g(0,0)=0.$ In \cite[Theorem 1]{Carr} the existence of a center manifold is established. This is given by the graph of a function $h(x)$, similar to the stochastic case which we treat here. The goal is to approximate this function using a Taylor expansion. To this end, one defines for a function $\phi:\R^n\to\R^m$, which is $C^1$ in a neighborhood of the origin, the function
$$(M\phi)(x) := \txtD\phi(x)\left[Ax + f(x,\phi(x))\right] - B\phi(x) - g(x,\phi(x)).$$
Here $M\phi$ is determined computing $\dot y$ using the ansatz $y=\phi(x)$ and the equation for $y$ given in~\eqref{syst:carr} and taking the difference between the two terms obtained.~Namely, by plugging in $y=\phi(x)$ we get
\begin{align*}
    \dot y 
    &= B\phi(x) + g(x,\phi(x)),
\end{align*}
and considering $\dot y = \txtD\phi(x)\dot x$ we get
\begin{align*}
    \dot y &= \txtD\phi(x)\dot x = \txtD\phi(x)\left[Ax + f(x,\phi(x)) \right].
\end{align*}
The difference of the previous two terms determines $M\phi$. We note that $(Mh)(x)=0$ as on the center manifold we have $y=h(x)$. Hence, we expect $M\phi$ to get close to zero as $\phi$ approximates $h$.
\begin{theorem}{\em(\cite[Theorem 3]{Carr})}\label{Carr Theorem}
    Let $\phi:\R^n\to\R^m$ be a $C^1$ function in a neighborhood of the origin with $\phi(0)=0,~\txtD\phi(0)=0$ and $(M\phi)(x)=\cO(|x|^q)$ as $x\to0$ where $q>1$. Then as $x\to0$,
    $$|h(x)-\phi(x)|=\cO(|x|^q).$$
\end{theorem}
We sketch the main ideas of the proof, as our proof in the stochastic case relies on similar arguments. Let $X$ be the space of bounded Lipschitz functions on $\R^n$.
First, according to \cite[Theorem 1]{Carr} there exists a contraction map $T$ on $X$ having $h$ as fixed point. With this we define $Sz = T(z+\phi)-\phi$ on a closed subset $V\subset X$. As $T$ is a contraction on $X$, $S$ is a contraction on $V$.~Next, we define for $K>0$ $$V=\{z\in X: |z(x)|\le K|x|^q ~~\textnormal{ for all } x\in\R^n\}.$$ If we can find $K$ such that $S$ maps $V$ into $V$ then the claim follows because
\begin{align*}
    |h(x)-\phi(x)| &= |T(h(x))-\phi(x)| \\
    &= |T(z(x)+\phi(x))-\phi(x)| \\
    &= |S(z(x))| \le K|x|^q,
\end{align*}
where we used that for the fixed point $z\in V$ of $S$, $z+\phi$ is a fixed point of $T$.\\ 
So it remains to show that $S$ maps $V$ into $V$. To this goal, it can be shown that $(Sz)(x_0)$ can be rewritten as 
\begin{align}\label{S def}
    (Sz)(x_0) = \int_{-\I}^0e^{-Br}Q(x,z)~\txtd r,
\end{align}
where $Q(x,z)$ depends on the nonlinearities $f,g$ and on $M\phi$. Moreover, it holds that $Q(x,0)=(M\phi)(x)$ and $Q$ is Lipschitz in $z$. So to find a bound for $S$ we first find a bound for $Q$ of the form
\begin{align*}
    |Q(x,z)|&\le |Q(x,0)|+|Q(x,z)-Q(x,0)| \\
    &= |(M\phi)(x)| + |Q(x,z)-Q(x,0)|.
\end{align*}
We have $|(M\phi)(x)|\le C_1|x|^q$ due to the assumption $(M\phi)(x)=\cO(|x|^q)$. Moreover, since $Q$ is Lipschitz, we have $|Q(x,z)-Q(x,0)|\le C_2|z|$. In conclusion,
$$|Q(x,z)|\le (C_1 + KC_2)|x|^q.$$
Using \eqref{S def} for $S$ and the bound for $Q$, we see that $S$ can be bounded. Choosing $K$ large enough, it follows that $S$ maps $V$ into $V$ and hence the claim follows. For more details we refer to \cite[Theorem 3]{Carr}.

\section{Discretization of the map $\tilde J$}
\label{a}
We provide here the proof of Lemma \ref{tilde J}.\\
\begin{proof}
    We proceed in two steps. Firstly, we write $\phi^{j-1}_t$ in integral form similar to the discretized version of $J$, so that we can calculate $J-\phi$ more easily. Secondly, we compute $\tilde J$. For the first step we fix $j\in\Z^-$, $t\in[0,1]$ and use integration by parts to determine $S^s(j-1+t-l) \phi_0^l\left(x_0^l\right)$ for some $l<j-1$.~Thereafter we use the chain rule to compute the differential $\txtd \phi_r(x_r)$, recalling that $\mathbf{W}=(W,\mathbb{W})$ is a geometric rough path.~This leads to
    \begin{align*}
        & S^s(j-1+t-l) \phi^l\left(x_0^l\right)\\
        &= -\int_l^{0} -A^s S^s(j-1+t-r)\phi(x_r)~\txtd r -\int_l^{0} S^s(j-1+t-r) \txtd \phi(x_r)\\
        & = -\int_l^{0} S^s(j-1+t-r) \Bigg[-A^s \phi(x_r) +\sum_{i=2}^{q} i\alpha_i (x_r)^{i-1} [A^c x_r + F^c_R\left(x, U +\phi(x)\right)[r] ]\Bigg]~\txtd r \\
        &~~~~ - \int_l^{0} S^s(j-1+t-r) \tilde\phi(x_r) ~\txtd r 
        - \sum\limits_{i=2}^{q}\int_l^{0} S^s(j-1+t-r) g_i (x_r)^i~\txtd \textbf{W}_r \\
        &~~~~ - \int_l^0 S^s(j-1+t-r) \sum_{i=2}^q i \alpha_i (x_r)^{i-1}G_R^c(x,U+\phi(x))[r]~\txtd \bfW_r,
    \end{align*}
    where $\tilde\phi(x_r)$ is defined in \eqref{tilde phi def}.
    Now we consider the difference 
    $$S^s(j-1+t-l) \phi^l\left(x_0^l\right)-S^s(j-1+t-(j-1+t)) \phi^{j-1}\left(x_t^{j-1}\right)=S^s(j-1+t-l) \phi^l\left(x_0^l\right)- \phi^{j-1}\left(x_t^{j-1}\right)$$ and use the equality above to obtain
    \begin{align}\label{eq:phidif}
        & S^s(j-1+t-l) \phi^l\left(x_0^l\right) - \phi^{j-1}\left(x_t^{j-1}\right)\nonumber\\
        & = -\int_l^{j-1+t} S^s(j-1+t-r) \Bigg[-A^s \phi(x_r) +\sum_{i=2}^{q} i\alpha_i (x_r)^{i-1} [A^c x_r + F^c_R\left(x, U +\phi(x)\right)[r] ]\Bigg]~\txtd r \nonumber\\
        &~~~~ - \int_l^{j-1+t} S^s(j-1+t-r)\tilde\phi(x_r) ~\txtd r - \sum\limits_{i=2}^{q}\int_l^{j-1+t} S^s(j-1+t-r) g_i (x_r)^i~\txtd \textbf{W}_r \nonumber\\
        &~~~~ - \int_l^{j-1+t} S^s(j-1+t-r) \sum_{i=2}^q i \alpha_i (x_r)^{i-1}G_R^c(x,U+\phi(x))[r]~\txtd \bfW_r.
    \end{align}
    Next we discretize the integrals from \eqref{eq:phidif} as in \cite[Section 4.1]{KN23}. This yields
    \begin{align}
        & S^s(j-1+t-l) \phi^l\left(x_0^l\right) - \phi^{j-1}\left(x_t^{j-1}\right)\label{forphi} \\
        &= -\sum_{k=l+1}^{j-1}S^s(j-1+t-k)\Bigg[\int_0^1 S^s(1-r) \Bigg[-A^s \phi^{k-1}\left(x^{k-1}_r\right) +\sum_{i=2}^{q} i\alpha^{k-1}_i \left(x^{k-1}_r\right)^{i-1} \nonumber\\
        &\times\Big[A^c x^{k-1}_r +
         F^c_R\left(x^{k-1}, U^{k-1} +\phi^{k-1}\left(x^{k-1}\right)\right)[r] \Big]+\tilde\phi^{k-1}\left(x^{k-1}_r\right)\Bigg] ~\txtd r \nonumber\\
         &+ \int_0^{1} S^s(1-r) \sum\limits_{i=2}^{q}g_i^{k-1} \left(x^{k-1}_r\right)^i + i \alpha^{k-1}_i \left(x_r^{k-1}\right)^{i-1}G_R^c\left(x^{k-1},U^{k-1}+\phi^{k-1}\left(x^{k-1}\right)\right)[r]~\txtd \Theta_{k-1}\textbf{W}_r\Bigg]\nonumber\\
         &-\int_0^{t}S^s(t-r)\Bigg[-A^s \phi^{j-1}\left(x^{j-1}_r\right) +\sum_{i=2}^{q} i\alpha^{j-1}_i \left(x^{j-1}_r\right)^{i-1} \Big[A^c x^{j-1}_r\nonumber \\
         &+ F^c_R\left(x^{j-1}, U^{j-1} +\phi^{j-1}\left(x^{j-1}\right)\right)[r] \Big]\Bigg]~\txtd r -\int_0^tS(t-r)\tilde\phi^{j-1}\left(x^{j-1}_r\right) ~\txtd r\nonumber \\
         &- \int_0^{t} S^s(t-r) \sum\limits_{i=2}^{q}g_i^{j-1} \left(x^{j-1}_r\right)^i + i \alpha^{j-1}_i \left(x_r^{j-1}\right)^{i-1}G_R^c\left(x^{j-1},U^{j-1}+\phi^{j-1}\left(x^{j-1}\right)\right)[r]~\txtd \Theta_{j-1}\textbf{W}_r.\nonumber
    \end{align}
 We mention that the rough integrals in \eqref{forphi} are well-defined due to Lemma \ref{G-g in D2gamma}.
 Now we let $l\to-\I$ in order to obtain a representation for $\phi^{j-1}\left(x^{j-1}_t\right)$ from~\eqref{forphi}.
    Using \eqref{beta} and \eqref{phi discretized bound} we obtain
    \begin{align*}
        \Xnorm{S^s(j-1+t-l) \phi^l\left(x_0^l\right)}
        &\le M_se^{\beta(l-(j-1+t))}\Xnorm{\phi^l\left(x_0^l\right)} \\
        &\le M_se^{\beta(l-(j-1+t))}\Dnorm{\phi^l\left(x^l\right),\left(\phi^l\left(x^l\right)\right)'}\\
        &\le M_s e^{\beta(l-(j-1+t))}C \tilde C_x^qq^2\left(1+\norm{\Theta_lW}_\gamma\right)^2\Xnorm{\xi} \\
        &\le M_s e^{\beta(l-(j-1+t))}C \tilde C_x^qq^2C[\norm{W}_\gamma]e^{-\eps l}\Xnorm{\xi} \\
        &= M_sC \tilde C_x^qq^2\Xnorm{\xi} e^{-\beta (j-1+t)} e^{(\beta-\eps) l},
    \end{align*}
    where we used the bound \eqref{tabove} to bound for $l\in Z^{-}$ \[\left(1+\norm{\Theta_lW}_\gamma\right)^2 \leq e^{-\varepsilon l} C[\norm{W}_\gamma], \] for a tempered from above random variable $C[\norm{W}_\gamma]$ and $0<\eps<\beta$.
    So $\Xnorm{S^s(j-1+t-l) \phi^l\left(x_0^l\right)}$ tends to $0$ as $l\to-\infty$ due to $\beta-\eps>0$. Using this we obtain the following representation 
    \begin{align*}
        & - \phi^{j-1}\left(x_t^{j-1}\right)\\
        &= -\sum_{k=-\I}^{j-1}S^s(j-1+t-k)\Bigg[\int_0^1 S^s(1-r) \Bigg[-A^s \phi^{k-1}\left(x^{k-1}_r\right) +\sum_{i=2}^{q} i\alpha^{k-1}_i \left(x^{k-1}_r\right)^{i-1} \\
        &\times \Big[A^c x^{k-1}_r 
        + F^c_R\left(x^{k-1}, U^{k-1} +\phi^{k-1}\left(x^{k-1}\right)\right)[r] \Big]\Bigg]~\txtd r 
        + \int_0^{1} S^s(1-r)\tilde\phi^{k-1}\left(x^{k-1}_r\right) ~\txtd r \\
         &+ \int_0^{1} S^s(1-r) \sum\limits_{i=2}^{q}g_i^{k-1} \left(x^{k-1}_r\right)^i+ i \alpha^{k-1}_i \left(x_r^{k-1}\right)^{i-1}G_R^c\left(x^{k-1},U^{k-1}+\phi^{k-1}\left(x^{k-1}\right)\right)[r]~\txtd \Theta_{k-1}\textbf{W}_r\Bigg]\\
         &-\int_0^{t}S^s(t-r)\Bigg[-A^s \phi^{j-1}\left(x^{j-1}_r\right) +\sum_{i=2}^{q} i\alpha^{j-1}_i \left(x^{j-1}_r\right)^{i-1} \\&\times\Big[A^c x^{j-1}_r 
         + F^c_R\left(x^{j-1}, U^{j-1} +\phi^{j-1}\left(x^{j-1}\right)\right)[r] \Big]\Bigg] -\int_0^tS(t-r)\tilde\phi^{j-1}\left(x^{j-1}_r\right) ~\txtd r \\
         &- \int_0^{t} S^s(t-r) \sum\limits_{i=1}^{q}g_i^{j-1} \left(x^{j-1}_r\right)^i+ i \alpha^{j-1}_i \left(x_r^{j-1}\right)^{i-1}G_R^c\left(x^{j-1},U^{j-1}+\phi^{j-1}\left(x^{j-1}\right)\right)[r]~\txtd \Theta_{j-1}\textbf{W}_r.
    \end{align*}
    Combining the expression for $\phi^{j-1}_t$ and applying the stable projection to $J$ given in~\eqref{J discr def}, 
    we get
    \begin{align*}
    &\tilde{J}(\bfW,\U,\xi)[j-1,t] = P^s J(\bfW,\U+\Phi(x),\xi)[j-1,t]-\phi^{j-1}\left(x_t^{j-1}\right) \\
    &= \sum_{k=-\I}^{j-1} S^s(t+j-1-k) ~T_1\left(\Theta_{k-1}\bfW,x^{k-1},U^{k-1},\balpha^{k-1}\right)\\
    &~~~~+T_t\left(\Theta_{j-1}\bfW,x^{j-1},U^{j-1},\balpha^{j-1}\right). 
\end{align*}
    \qed
\end{proof}

\section{Polynomial representation of $F^{c/s}$}\label{Polynomial calculation}
We present here the ideas of the proofs for Lemma \ref{lem:F reformulated} and Lemma \ref{lem:sum F reformulated}. The goal of both Lemmas is 
to have a representation in the form $\sum_i P_i \left(x^{j-1}\right)^i$ where $P_i$ is some term that we can control. We need this representation in order to compare it with $\sum_{i=1}^qf_i^{j-1}\left(x^{j-1}\right)^i$.~This will allow us to choose the coefficients $f_i$ of the RDEs~\eqref{alpha RDE}. 

\begin{refproof}[Proof sketch of Lemma~\ref{lem:F reformulated}]\label{proof:F reformulated}
    As $F^{c/s}\left(x^{j-1},\phi^{j-1}\left(x^{j-1}\right)\right)$ is a polynomial with degree up to $q$ we can write it as 
    $$F^{c/s}\left(x^{j-1},\phi^{j-1}\left(x^{j-1}\right)\right) = \sum_{i=2}^q\sum_{k=0}^i F^{c/s}_{k.i-k} \left(x^{j-1}\right)^k \phi^{j-1}\left(x^{j-1}\right)^{i-k}. $$
    The powers of $\phi^{j-1}\left(x^{j-1}\right)$ are polynomials where each monomial can be written as $C \left(\balpha^{j-1}\right)^\kappa (x^{j-1})^i,$ with $\balpha=(\alpha^{j-1}_2,\dots,\alpha^{j-1}_q)$, $\kappa\in \N^{q-1}$ is a multi-index and $i\in\{4,\dots,q^2\}$. We note that the largest possible monomial is determined by $\phi^{j-1}\left(x^{j-1}\right)^q$ and is $C\left(\alpha_q^{j-1}\right)^q \left(\left(x^{j-1}\right)^q\right)^q = C\left(\alpha_q^{j-1}\right)^q \left(x^{j-1}\right)^{q^2}$. We write $P_i^{F^{c/s}}\left(\alpha_2^{j-1},\dots,\alpha_{i-1}^{j-1}\right)$ for the sum of all $C \left(\balpha^{j-1}\right)^\kappa$ multiplied with $x^i$ to keep track of the coefficients of $x^i$. If $i>q$ we write $P_i^{F^{c/s}}\left(\alpha_2^{j-1},\dots,\alpha_{q}^{j-1}\right)$. Hence we get
    \begin{align*}
        &F^{c/s}\left(x^{j-1},\phi^{j-1}\left(x^{j-1}\right)\right) \\
        &= \sum_{i=2}^q\sum_{k=0}^i F^{c/s}_{k.i-k} \left(x^{j-1}\right)^k \phi^{j-1}\left(x^{j-1}\right)^{i-k} \\
        &=\sum_{i=2}^q \left( P_i^{F^{c/s}}\left(\alpha_2^{j-1},\dots,\alpha_{i-1}^{j-1}\right) + F^{c/s}_{i,0}\right) \left(x^{j-1}\right)^i + \sum_{i=q+1}^{q^2} P_i^{F^{c/s}}\left(\alpha_2^{j-1},\dots,\alpha_{q}^{j-1}\right) \left(x^{j-1}\right)^i.
    \end{align*}
    This shows the first claim as $F^{c/s}_{i,0}$ are just some constants. Moreover, this makes it clear why the second sum does not contain any constants as $F^{c/s}_{i,0}=0$ for $i>q$.\\
    For the second claim we note the monomials $C \left(\balpha^{j-1}\right)^\kappa$ can be bounded in the following way using Lemma \ref{lem:addition and multiplication of RP}
    \begin{align*}
        \Dnorm{C \left(\balpha^{j-1}\right)^\kappa, C \left(\left(\balpha^{j-1}\right)^\kappa\right)'} 
        &\le C[\norm{W}_\gamma] \left(\tilde C_\alpha^{j-1} \right)^{\norm{\kappa}}\\
        &\le C[\norm{W}_\gamma] R(\Theta_{j-1}\bfW)^{\norm{\kappa}}\\
        &\le C[\norm{W}_\gamma] R(\Theta_{j-1}\bfW),
    \end{align*}
    where $\norm{\kappa}$ is just the euclidean norm of $\N^{q-1}$ and $\tilde C^{j-1}_\alpha$ is defined in \eqref{C_alpha bound} and $C[\norm{W}_\gamma]$ is a polynomial with respect to $\norm{W}_\gamma$.\\ 
    Based on the previous deliberations, we get for $i\le q$ an estimate of the form
    \begin{align*}
        \Dnorm{P^{F^{c/s}}_i\left(\alpha^{j-1}_2,\dots,\alpha^{j-1}_{i-1}\right),P^{F^{c/s}}_i\left(\alpha^{j-1}_2,\dots,\alpha^{j-1}_{i-1}\right)'} \le C[\norm{W}_\gamma,F^{c/s},i] R(\Theta_{j-1}\bfW),
    \end{align*}
    where the constant $C$ is increasing in $i$ which is the bound for the number of such monomials in $P_i^{F^{c/s}}\left(\alpha_2^{j-1},\dots,\alpha_{i-1}^{j-1}\right)$. 
    For $i>q$ we have
    \begin{align*}
        \Dnorm{P^{F^{c/s}}_i\left(\alpha^{j-1}_2,\dots,\alpha^{j-1}_{q}\right),P^{F^{c/s}}_i\left(\alpha^{j-1}_2,\dots,\alpha^{j-1}_{q}\right)'} \le C[\norm{W}_\gamma,F^{c/s},i] R(\Theta_{j-1}\bfW).
    \end{align*}
    \qed \\
\end{refproof}

\begin{refproof}[Proof sketch of Lemma \ref{lem:sum F reformulated}]\label{proof:sum F reformulated}
    We consider
    $$\sum_{i=2}^q i \alpha^{j-1}_i \left(x^{j-1}_t\right)^{i-1} F^c\left(x^{j-1}_t,\phi^{j-1}\left(x_t^{j-1}\right)\right)$$
    and plug in the representation for $F^c$ shown in Lemma \ref{lem:F reformulated}. This leads to 
    \begin{align*}
        &\sum_{i=2}^q i \alpha^{j-1}_i \left(x^{j-1}_t\right)^{i-1} \\&~~\left(\sum_{k=2}^q \left( P_k^{F^{c/s}}\left(\alpha_2^{j-1},\dots,\alpha_{k-1}^{j-1}\right) + C\right) \left(x^{j-1}\right)^k + \sum_{k=q+1}^{q^2} P_k^{F^{c/s}}\left(\alpha_2^{j-1},\dots,\alpha_{q}^{j-1}\right) \left(x^{j-1}\right)^k\right).
    \end{align*}
    We rewrite the sum above such that it is again of the form $\sum_{k=2}^q P_k \left(x^{j-1}\right)^k$. We can formally write it as
    \begin{align*}
        &\sum_{i=2}^q i \alpha^{j-1}_i \left(x^{j-1}_t\right)^{i-1} F^c\left(x^{j-1}_t,\phi^{j-1}\left(x_t^{j-1}\right)\right) \\
        &=\sum_{i=2}^q \left(\tilde P^{F^{c}}_i\left(\alpha^{j-1}_2,\dots,\alpha^{j-1}_{i-1}\right) + C\right) \left(x^{j-1}_t\right)^i + \sum_{i=q+1}^{q^3-q^2} \tilde P^{F^{c}}_i\left(\alpha^{j-1}_2,\dots,\alpha^{j-1}_{q}\right) \left(x^{j-1}_t\right)^i,
    \end{align*}
    where the exact form of $\tilde P^{F^{c}}_i\left(\alpha^{j-1}_2,\dots,\alpha^{j-1}_{i-1}\right)$ is not calculated here, but it is a polynomial consisting of monomials $C[i] \left(\balpha^{j-1}\right)^\kappa$. 
    As $\tilde P^{F^{c}}_i$ has the same representation as $ P^{F^{c}}_i$, the claim about the bound follows by the same arguments as in the proof of the Lemma~\ref{lem:F reformulated}. \qed
\end{refproof}

\section{Approximation of local rough center manifolds using the leading order terms of the coefficients}\label{ChekrounProofWithRP}
For the sake of completeness, we sketch the main steps of the approximation result of \cite{ChekrounLiuWang} in the rough path setting. The main difference to the approach  in Section \ref{sec:main} is that we do not rely on a fixed point argument anymore, but impose the additional assumption~\eqref{leading:f} and~\eqref{leading:g} on the coefficients. \\

As in Section~\ref{rp} we consider the rough differential equation on the phase space $\cX$ given by 

\begin{align*}
    \txtd U_t = AU_t + F(U_t) ~\txtd t + G(U_t)~\txtd \bfW_t.
\end{align*}

The operator $A$ satisfies Assumption \ref{ass:linearpart} and for the coefficients we impose the following restrictions.
\begin{assumptions}\label{ass:Appendix D}
    \begin{itemize}
        \item $F:\cX\to\cX\in C^m$  with $F(0)=\txtD F(0)=0$. 
        \item $G:\cX\to\cL(\cV, \cX)\in C^{m+3}$ with $G(0)=\txtD G(0)=\txtD^2G(0)=0$.
    \end{itemize}
\end{assumptions}

As in Subsection~\ref{sec:cutoff}, we truncate the coefficients $F$ and $G$ using the path-dependent cut-off function.  Let the cut-off function
$\chi:\cD^{2\gamma}_W([0,1],\cX)\to\cD^{2\gamma}_W([0,1],\cX) $ be a Lipschitz function defined as
$$\chi_R(U)[\cdot] := \begin{cases}
    U & \Dnorm{U,U'}\le R/2 \\
    0 & \Dnorm{U,U'} \ge R,
\end{cases}$$
for $R>0$. With this we define $F_R(U)[\cdot]:=F\circ\chi_R(U)[\cdot]$ and $G_R(U)[\cdot]:=G\circ\chi_R(U)[\cdot]$ leading to the RDE with path dependent coefficients

\begin{align}\label{main:app}
    \txtd U_t = AU_t + F_R(U)[t] ~\txtd t + G_R(U)[t]~\txtd \bfW_t.
\end{align}
The approximation of $h^c(\xi,\bfW)$ given in \eqref{h:cm} is twofold. Firstly, we consider a leading-order approximation of $F$ and $G$. This means that we assume the following representation for $F$ 
\begin{align}\label{leading:f}
F(U)=F_{l}(U,\dots,U) + \cO\left(\Xnorm{U}^{l+1}\right),
\end{align}
and $G$
\begin{align}\label{leading:g} G(U)=G_{l}(U,\dots,U) + \cO\left(\Xnorm{U}^{l+1}\right).
\end{align}
The Landau symbol $\cO$ is defined in \eqref{def:Landau symbol}. Without loss of generality we assume that $F$ and $G$ have the same leading order $l< m$.~Using Lemma \ref{est:order 2} we get that 
\begin{align}\label{F_R,l bound}
    \norm{F_R(U)-F_{R,l}(U,\dots,U)}_\I \le C\Dnorm{U,U'}^{l+1},
\end{align}
and
\begin{align}\label{G_R,l bound}
    \Dnorm{G_R(U)-G_{R,l}(U,\dots,U), (G_R(U)-G_{R,l}(U,\dots,U))'} \le C\Dnorm{U,U'}^{l+1}.
\end{align}
Secondly, we do not plug in a solution $U$ of \eqref{main:app} into $F_{R,l}$ and $G_{R,l}$ but the solution of the linear equation
\begin{align*}
\begin{cases}&\txtd U_t = A^c U_t~\txtd t \\
& U_0=\xi
\end{cases}
\end{align*}
given by 
$$S^c(t)\xi=P^cS(t)\xi=P^ce^{At}\xi.$$ Our goal is to bound the rough path norm of the difference between $h^c(\xi,\bfW)$ and the approximation $h^{app}(\xi,\bfW)$ as defined in Theorem~\ref{approx:leading:order}. \\


For notational simplicity we set for $(U,U')\in\cD^{2\gamma}_W([0,1];\cX)$ and $t\in[0,1]$ 
\begin{align*}
    T_R^s(\bfW,U)[t] &:= \int_0^t S(t-r)P^sF_R(U)[r] \txtd r + \int_0^t S(t-r)P^sG_R(U)[r] \txtd \mathbf{W}_r \\
    T_{R,k}^s(\bfW,U)[t] &:= \int_0^t S(t-r)P^sF_{R,l}(U)[r] \txtd r + \int_0^t S(t-r)P^sG_{R,l}(U)[r] \txtd {\mathbf{W}}_r.
\end{align*}
 Moreover we define the discretized Lyapunov-Perron map corresponding to \eqref{main:app} for $j\in\Z^-$ and $t\in[0,1]$ as follows 
\begin{align*}
    J(\bfW,\mathbb{U},\xi)[j-1,t] &:= S^c(t+j-1)\xi\\
    &-\sum_{k=0}^{j+1} S^s(t+j-1-k)\int_0^1 S^c(1-r)P^cF_R(U^{k-1})(r)\ \txtd r \\
    &- \sum_{k=0}^{j+1} S^s(t+j-1-k)\int_0^1 S^c(1-r)P^cG_R(U^{k-1})(r)\ \txtd \Theta_{k-1}\bfW_r\\
    &-\int_t^1 S^c(t-r)P^cF_R(U^{j-1})(r)\ \txtd r +\int_t^1 S^c(t-r)P^cG_R(U^{j-1})(r)\ \txtd \Theta_{j-1}\bfW_r\\
    &+\sum_{k=-\I}^{j-1} S^s(t+j-1-k)\int_0^1 S^s(1-r)P^sF_R(U^{k-1})(r)\ \txtd r \\
    &+ \sum_{k=-\I}^{j-1} S^s(t+j-1-k)\int_0^1 S^s(1-r)P^sG_R(U^{k-1})(r)\ \txtd \Theta_{k-1}\bfW_r\\
    &+\int_0^t S^s(t-r)P^sF_R(U^{j-1})(r)\ \txtd r +\int_0^t S^s(t-r)P^sG_R(U^{j-1})(r)\ \txtd \Theta_{j-1}\bfW_r.
\end{align*}
The Lyapunov-Perron map possesses a unique fixed point $\Gamma(\xi,\bfW)$, as proven in \cite[Theorem 4.7]{KN23}.
\begin{theorem}\label{approx:leading:order}
    We impose Assumption \ref{ass:linearpart} and Assumption \ref{ass:Appendix D}. Moreover, we let $\xi \in \cX^c$ and define
    \begin{align*}
        h(\xi,\bfW) &:= P^s\Gamma(\xi,\bfW)[-1,1] \\
        &= \sum_{k=-\I}^0 S^s(-k)\int_0^1 S^s(1-r)P^sF_R(\Gamma(\xi,\bfW)[k-1,r])\ \txtd r \\
        &+ \sum_{k=-\I}^0 S^s(-k)\int_0^1 S^s(1-r)P^sG_R(\Gamma(\xi,\bfW)[k-1,r])\ \txtd \Theta_{k-1}\bfW_r,
    \end{align*}
    and 
    \begin{align*}
        h^{app}(\xi,\bfW) 
        &:= \sum_{k=-\I}^0 S^s(-k)\int_0^1 S^s(1-r)P^sF_{R,l}(S^c(k-1+r)\xi)\ \txtd r \\
        &+ \sum_{k=-\I}^0 S^s(-k)\int_0^1 S^s(1-r)P^sG_{R,l}(S^c(k-1+r)\xi)\ \txtd \Theta_{k-1}\bfW_r.
    \end{align*}
    Then
    \begin{align*}
        \norm{h(\xi,\bfW)-h^{app}(\xi,\bfW), (h(\xi,\bfW)-h^{app}(\xi,\bfW))'}_{\cD_W^{2\gamma}} \le C[\Xnorm{A}, F, G] \Xnorm{\xi}^l.
    \end{align*}
\end{theorem}
\begin{proof}
We sketch the main steps of the proof, omitting Gubinelli's derivative for notational simplicity, whenever this is clear from the context. 
\begin{enumerate}
\item We first bound $\Dnorm{h-h^{app}}$ by $\Dnorm{T_R^s-T_{R,l}^s}$. 
To simplify the notation let $u(k-1,\cdot): = \Gamma(\xi,\bfW)[k-1,\cdot]$ and $v(k-1,\cdot):=S^c(k-1+\cdot)\xi$. We recall that $\Gamma(\xi,\bfW)$ is the fixed point of $J(\bfW,\U,\xi)$.
By definition we get
\begin{align*}
    &\norm{h(\xi,\bfW) - h^{app}(\xi,\bfW)}_{\cD_W^{2\gamma}} \\
    &\le \sum_{k=-\I}^0C_sM_se^{\beta k} \norm{T^s_R(\Theta_{k-1}\bfW,u(k-1,\cdot))
    - T^s_{R,l}(\Theta_{k-1}\bfW,v(k-1,\cdot))}_{\cD_W^{2\gamma}}.
\end{align*}
\item We split $\Dnorm{T_R^s-T_{R,l}^s}$ into terms depending on $u$ and on $u-v$.~Now we focus on $\norm{T^s_R(\bfW,u(k-1,\cdot))
    - T^s_{R,l}(\bfW,v(k-1,\cdot))}_{\cD_W^{2\gamma}}$. 
\begin{align*}
    &\norm{T^s_R(\bfW,u(k-1,\cdot))
    - T^s_{R,l}(\bfW,v(k-1,\cdot))}_{\cD_W^{2\gamma}}\\
    &\le \norm{T^s_R(\bfW,u(k-1,\cdot))
    - T^s_{R,l}(\bfW,u(k-1,\cdot))}_{\cD_W^{2\gamma}} \\
    &+ \norm{T^s_{R,l}(\bfW,u(k-1,\cdot))
    - T^s_{R,l}(\bfW,v(k-1,\cdot))}_{\cD_W^{2\gamma}} \\
    &\le \tilde C[\Theta_{k-1}\bfW]R(\Theta_{k-1}\bfW)\norm{u(k-1,\cdot)}_{\cD_W^{2\gamma}}^{l+1} \\
    &+ \hat C[\Theta_{k-1}\bfW]R(\Theta_{k-1}\bfW) \norm{u(k-1,\cdot)-v(k-1,\cdot)}_{\cD_W^{2\gamma}},
\end{align*}
for some constants $\tilde C[\Theta_{k-1}\bfW],~\hat C[\Theta_{k-1}\bfW]$, where we used Lemma \ref{integral bound}, \eqref{F_R,l bound} and \eqref{G_R,l bound}. 
\item We estimate the terms containing $u$.
Due to Lemma \cite[Lemma 4.9]{KN23} we know that the map $\cX^c\to BC^\eta(\cD^{2\gamma}_W),~\xi \mapsto\Gamma(\xi,W)$ is Lipschitz continuous and we denote by $L_\Gamma$ its Lipschitz constant. Moreover in \cite[Lemma 4.11]{KN23} it is shown, that for all $k\in\Z^-$ we have
$$\norm{\Gamma(\xi,\bfW)[k-1,\cdot]}_{\cD_W^{2\gamma}} \le L_\Gamma\Xnorm{\xi}.$$
Consequently we get for all $k\in\Z^-$
\begin{align*}
     \norm{u(k-1,\cdot)}_{\cD_W^{2\gamma}}^{l+1} 
    &= \norm{\Gamma(\xi,\bfW)[k-1,\cdot]}_{\cD_W^{2\gamma}}^{l+1} \\
    &\le L_\Gamma^{l+1}\Xnorm{\xi}^{l+1}.
\end{align*}
\item We now estimate the term containing the $u-v$ term.~To this aim, we use that $u=\Gamma(\xi,\bfW)$ is the fixed point of $J(\bfW,\U,\xi)$ and hence the difference $u(k-1,\cdot)-v(k-1,\cdot)$ is a sum of $T_R^s(\bfW,\Gamma(\bfW,\xi))$ and $T_R^c(\bfW,\Gamma(\bfW,\xi))$. Then we apply Lemma \ref{integral bound} and the Lipschitz continuity of $\Gamma(\cdot,\bfW)$ to get 
\begin{align*}
    &\norm{u(k-1,\cdot)-v(k-1,\cdot)}_{\cD_W^{2\gamma}} \\
    &\le C[\Xnorm{A},R,\Theta_{k-1}\bfW] \norm{\Gamma(\xi,\bfW)[k-1,\cdot]}_{\cD_W^{2\gamma}}^l \le C[\Xnorm{A},R,\Theta_{k-1}\bfW]L_\Gamma^l\Xnorm{\xi}^l,
\end{align*}
where we used $\norm{F(U)}_\I\le C \Dnorm{U,U'}^l$ and $\Dnorm{G(U),G(U)'}\le C\Dnorm{U,U'}^l$ due to \eqref{leading:f} and \eqref{leading:g}.
\item Combining the steps 2--4, we obtain
\begin{align*}
    &\norm{T^s_R(\bfW,u(k-1,\cdot))
    - T^s_{R,l}(\bfW,v(k-1,\cdot))}_{\cD_W^{2\gamma}} \\
    &\le 2 \tilde C[\Theta_{k-1}\bfW]R(\Theta_{k-1}\bfW)L_\Gamma^{l+1}\Xnorm{\xi}^{l+1} \\
    &+ \hat C[\Theta_{k-1}\bfW]R(\Theta_{k-1}\bfW)C[\Xnorm{A},R,\Theta_{k-1}\bfW] L_\Gamma^{l}\Xnorm{\xi}^l\\
    &\le \bar C \Xnorm{\xi}^l.
\end{align*}
To obtain the previous estimate, one also chooses $R$ depending on $\norm{W}_\gamma$ and $\norm{\W}_{2\gamma}$ with $R$ similar to Lemma \ref{tilde J maps into V} and Lemma \ref{tilde J contraction}. 
\item Plugging in all the bounds above, we get
\begin{align*}
    &\norm{h(\xi,\bfW) - h^{app}(\xi,\bfW)}_{\cD_W^{2\gamma}} \\
    &\le \sum_{k=-\I}^0C_sM_se^{\beta k} \norm{T^s_R(\Theta_{k-1}\bfW,u(k-1,\cdot))
    - T^s_{R,l}(\Theta_{k-1}\bfW,v(k-1,\cdot))}_{\cD_W^{2\gamma}}\\
    &\le \sum_{k=-\I}^0C_sM_se^{\beta k} \bar C \Xnorm{\xi}^l = C[\Xnorm{A},F,G]\Xnorm{\xi}^l,
\end{align*}
which proves the statement.
\end{enumerate}
\qed
\end{proof}
\section{Estimates for the leading order terms of the coefficients in the rough path norm}\label{order:cutoff}
We assume \ref{ass:Appendix D}. Let $1\le l <m$ and assume for all $t\in[0,1]$ that
\begin{align}\label{f:order}
\Xnorm{F(U_t)-F_l(U_t,\dots,U_t)}\le C\Xnorm{U_t}^{l+1},
\end{align}
and 
\begin{align}\label{g:order}
\Xnorm{G(U_t)-G_l(U_t,\dots,U_t)}\le C\Xnorm{U_t}^{l+1}.
\end{align}
We want to show
$$\norm{F_R(U)[\cdot] - F_{R,l}(U,\dots,U)[\cdot]}_\I\le C\norm{U}_\I^{l+1},$$
and 
$$\Dnorm{G_R(U)[\cdot]-G_{R,l}(U,\dots,U)[\cdot],\left(G_R(U)[\cdot]-G_{R,l}(U,\dots,U)[\cdot]\right)'}\le C\Dnorm{U,U'}^{l+1}.$$
This is done in two steps. First we establish the bound for $F_R=F \circ \chi_R$, $G_R= G \circ \chi_R$ in the space norm and in a second step we compute the $\cD^{2\gamma}_W$-norm given the assumptions~\eqref{f:order} and~\eqref{g:order}.
\begin{lemma}\label{est:order}
Let $l\geq 1$, $F:\cX\to\cX\in C_b^1$ with $F(0)=\txtD F(0)=0$ and $G:\cX\to \cL(\cV,\cX)\in C_b^3$ with $G(0)=\txtD G(0)=\txtD^2G(0)=0$. Then for all $(U,U')\in\cD^{2\gamma}_W$ and $t\in[0,1]$
    $$\Xnorm{F(U_t)-F_l(U_t,\dots,U_t)}\le C\Xnorm{U_t}^{l+1}$$
    implies that
    $$\Xnorm{F_R(U)[t]-F_{R,l}(U,\dots,U)[t]}\le C\Xnorm{U_t}^{l+1}.$$
    Analogously
    $$\Xnorm{G(U_t)-G_l(U_t,\dots,U_t)}\le C\Xnorm{U_t}^{l+1}$$
    implies that
    $$\Xnorm{G_R(U)[t]-G_{R,l}(U,\dots,U)[t]}\le C\Xnorm{U_t}^{l+1}.$$
\end{lemma}
\begin{proof}
    We use the definition of the cut-off $\chi_R$ specified in Definition~\ref{cut-off def} and get 
    \begin{align*}
        \Xnorm{F_R(U)[t]-F_{R,l}(U,\dots,U)[t]}
        &= \Xnorm{F(\chi_R(U)[t])-F_{l}(\chi_R(U)[t],\dots,\chi_R(U)[t])} \\
        &\le \Xnorm{F(U_t)-F_{l}(U_t,\dots,U_t)} \\
        &\le C\Xnorm{U_t}^{l+1}.
    \end{align*}
    With the same argument we get the claim for $G$. \qed
\end{proof}
\begin{lemma}\label{est:order 2}
    Let $F:\cX\to\cX\in C_b^1$ with $F(0)=\txtD F(0)=0$ and $G:\cX\to \cL(\cV,\cX)\in C_b^3$ with $G(0)=\txtD G(0)=\txtD^2G(0)=0$. Then for all $(U,U')\in\cD^{2\gamma}_W$ and $t\in[0,1]$ we have
    $$\Xnorm{F_R(U)[t]-F_{R,l}(U,\dots,U)[t]}\le C\Xnorm{U_t}^{l+1},$$
    which implies that
    $$\norm{F_R(U)[\cdot] - F_{R,l}(U,\dots,U)[\cdot]}_\I\le C\norm{U}_\I^{l+1} \le C\Dnorm{U,U'}^{l+1},$$
    and analogously
    $$\Xnorm{G_R(U)[t]-G_{R,l}(U,\dots,U)[t]}\le C\Xnorm{U_t}^{l+1},$$
    which implies that
    $$\Dnorm{G_R(U)[\cdot]-G_{R,l}(U,\dots,U)[\cdot],\left(G_R(U)[\cdot]-G_{R,l}(U,\dots,U)[\cdot]\right)'}\le C\Dnorm{U,U'}^{l+1}.$$
\end{lemma}
\begin{proof}
    First we show the claim for $G$. Due to Lemma \ref{est:order} there exists a constant $C$ such that for all $t\in[0,1]$
    $$\Xnorm{G_R(U)[t]-G_{R,l}(U,\dots,U)[t]}\le C\Xnorm{U^{l+1}_t}\le C\Xnorm{U_t}^{l+1},$$
    and we show that
    $$\Dnorm{G_R(U)[\cdot]-G_{R,l}(U,\dots,U)[\cdot],\left(G_R(U)[\cdot]-G_{R,l}(U,\dots,U)[\cdot]\right)'}\le C\Dnorm{U,U'}^{l+1}.$$
    We recall that for $(Y,Y')\in\cD^{2\gamma}_W$
    $$\Dnorm{Y,Y'}=\norm{Y}_\I + \norm{Y'}_\I + \norm{Y'}_\gamma + \norm{R^Y}_{2\gamma}.$$
    Hence, we compute the controlled rough path norm of the difference $(G_R(U)-G_{R,l}(U,\dots,U),G_R(U)'-G_{R,l}(U,\dots,U)')$ in several steps. 
    \begin{enumerate}
        \item We start with the supremum norm 
        \begin{align*}
        \norm{G_R(U)[\cdot]-G_{R,l}(U,\dots,U)[\cdot]}_\I
        &= \sup_{t\in[0,1]}\Xnorm{G_R(U)[t]-G_{R,l}(U,\dots,U)[t]} \\
        &\le C \sup_{t\in[0,1]} \Xnorm{U_t}^{l+1}\\
        &\le C\norm{U}_\I^{l+1}.
        \end{align*}
        
        \item Next we bound the supremum norm of the Gubinelli derivative.
        First we note that for all $t\in[0,1]$ we have 
        $$\Xnorm{\txtD G_R(U)[t] - \txtD G_{R,l}(U,\dots,U)[t]}\le C\Xnorm{U_t}^l.$$
        This can be seen directly by differentiating the Taylor expansion of $G$. 
        Hence we get
        \begin{align*}
        \norm{G_R(U)[\cdot]'-G_{R,l}(U,\dots,U)[\cdot]'}_\I
        &= \sup_{t\in[0,1]}\Xnorm{G_R(U)[t]'-G_{R,l}(U,\dots,U)[t]'} \\
        &= \sup_{t\in[0,1]}\Xnorm{\left(\txtD G_R(U)[t]-\txtD G_{R,l}(U,\dots,U)[t]\right)U'} \\
        &\le C \sup_{t\in[0,1]} \Xnorm{U_t}^{l}\Xnorm{U'_t}\\
        &\le C\norm{U}_\I^{l}\norm{U'}_\I.
        \end{align*}

        \item For the $\gamma$-H\"older norm of $\left(G_R(U)[\cdot]-G_{R,l}(U,\dots,U)[\cdot]\right)'$ we obtain \begin{align*}
            &\norm{\left(G_R(U)[\cdot]-G_{R,l}(U,\dots,U)[\cdot]\right)'}_\gamma\\
            &= \norm{\left(\txtD G_R(U)[\cdot]-\txtD G_{R,l}(U,\dots,U)[\cdot]\right)U'}_\gamma \\
            &= \sup_{s,t\in[0,1]} \frac{\Xnorm{\left(\txtD G_R(U)[t]-\txtD G_{R,l}(U,\dots,U)[t]\right)U_t' - \left(\txtD G_R(U)[s]-\txtD G_{R,l}(U,\dots,U)[s]\right)U_s'}}{|t-s|^\gamma} \\
            &\le \sup_{s,t\in[0,1]} \frac{\norm{\left(\txtD G_R(U)[\cdot]-\txtD G_{R,l}(U,\dots,U)[\cdot]\right)}_\I\Xnorm{U'_t-U'_s}}{|t-s|^\gamma}\\
            &\le C\norm{U}^l_\I\sup_{s,t\in[0,1]} \frac{\Xnorm{U'_t-U'_s}}{|t-s|^\gamma}\\
            &\le C\norm{U}^l_\I\norm{U'}_\gamma.
        \end{align*}
        \item We finally treat the remainder. Therefore, we have
        \begin{align*}
            &\Xnorm{R^{G(U)}_{s,t}} \\
            &= \Big\| \left(G_R(U)[t]-G_{R,l}(U,\dots,U)[t]\right) - \left(G_R(U)[s]-G_{R,l}(U,\dots,U)[s]\right) \\
            &~~~~-\left(\txtD G_R(U)[s]-\txtD G_{R,l}(U,\dots,U)[s]\right) U_s'W_{s,t} \Big\|_\cX\\
            &= \Big\|\left(G_R(U)[t]-G_{R,l}(U,\dots,U)[t]\right) - \left(G_R(U)[s]-G_{R,l}(U,\dots,U)[s]\right) \\
            &~~~~-\left(\txtD G_R(U)[s]-\txtD G_{R,l}(U,\dots,U)[s]\right) (U_t-U_s) \\
            &~~~~+ \left(\txtD G_R(U)[s]-\txtD G_{R,l}(U,\dots,U)[s]\right)R_{s,t}^U\Big\|_\cX \\
            &\le \big\|\left(G_R(U)[t]-G_{R,l}(U,\dots,U)[t]\right) - \left(G_R(U)[s]-G_{R,l}(U,\dots,U)[s]\right)\\
            &~~~~+\left(\txtD G_R(U)[s]-\txtD G_{R,l}(U,\dots,U)[s]\right) (U_t-U_s)\big\|_\cX \\
            &~~~~+ \Xnorm{\left(\txtD G_R(U)[s]-\txtD G_{R,l}(U,\dots,U)[s]\right)R_{s,t}^U} \\
            &\le \frac{1}{2}\norm{\txtD^2G_R(U)[\cdot]-\txtD^2G_{R,l}(U,\dots,U)[\cdot]}_\I\Xnorm{U_t-U_s}^2\\& + \norm{\txtD G_R(U)[\cdot]-\txtD G_{R,l}(U,\dots,U)[\cdot]}_\I\Xnorm{R^U_{s,t}}\\
            &\le C \norm{U}_\I^{l-1}\Xnorm{U_t-U_s}^2
            + C\norm{U}^l_\I \Xnorm{R_{s,t}^U}.
        \end{align*}
        Hence dividing by $|t-s|^{2\gamma}$ and taking the supremum over $s,t\in[0,1]$ we get
        \begin{align*}             \norm{R^{G(U)}}_{2\gamma} \le C\norm{U}^{l-1}_\I\norm{U}^2_{\gamma} + C\norm{U}_\I^l\norm{R^U}_{2\gamma}
        \end{align*}     
        Combining the 4 terms above we get
        \begin{align*}
            &\Dnorm{G_R(U)-G_{R,l}(U,\dots,U),G_R(U)'-G_{R,l}(U,\dots,U)'} \\
            &\le C\left(\norm{U}^{l+1}_\I + \norm{U}^l_\I\norm{U'}_\I + \norm{U}_\I^l\norm{U'}_\gamma + \norm{U}_\I^{l-1}\norm{U}^2_\gamma + \norm{U}_\I^l\norm{R^U}_{2\gamma}\right) \\
            &\le C\left(\Dnorm{U,U'}^{l+1} + \Dnorm{U,U'}^{l+1} + \Dnorm{U,U'}^{l+1}+(1+\norm{W}_\gamma)\Dnorm{U,U'}^{l+1} + \Dnorm{U,U'}^{l+1}\right) \\
            &\le C~(1+\norm{W}_\gamma)\Dnorm{U,U'}^{l+1}.
        \end{align*}
        The claim for $F$ follows from the first step of the proof.
        \qed
    \end{enumerate}
\end{proof}



\end{document}